\documentclass[a4paper,11pt]{amsart}

\usepackage[top=3.3cm, bottom=3cm, left=2.5cm, right=2.5cm]{geometry}
\usepackage[english]{babel}
\usepackage{latexsym,amsmath,amsthm,verbatim,ifthen,amssymb,mathtools,booktabs,amscd}
\usepackage{enumerate}
\usepackage{mathpazo}
\usepackage{paralist}
\usepackage{url}
\usepackage{hyperref}
\usepackage[all]{xy}
\usepackage{xcolor}
\usepackage{graphicx}
\usepackage{lscape}

\usepackage{tikz}
\usetikzlibrary{positioning,arrows,cd}

\DeclareMathOperator{\id}{id}
\DeclareMathOperator{\im}{im}

\newcommand{\NN}{\mathbb{N}}
\newcommand{\ZZ}{\mathbb{Z}}
\newcommand{\Z}{\mathbb{Z}}

\newcommand{\RR}{\mathbb{R}}

\newcommand{\quot}[2]{\left.\raisebox{.2em}{$#1$}\middle/\raisebox{-.2em}{$#2$}\right.}

\newcommand{\Am}{\mathrm{Am}}

\newcommand{\FF}{\mathcal{F}}
\newcommand{\HH}{\mathcal{H}}
\newcommand{\EE}{\mathcal{E}}

\newcommand{\II}{\mathcal {I}}

\newcommand{\ZGF}{\mathbb{Z}[G/\mathcal{F}]}
\newcommand{\ZGH}{\mathbb{Z}[G/H]}
\newcommand{\ZGHH}{\mathbb{Z}[G/\HH]}

\newcommand{\RGF}{\mathbb{R}[G/\mathcal{F}]}

\newcommand{\Hom}{\mathrm{Hom}}
\newcommand{\Ext}{\mathrm{Ext}}

\DeclareMathOperator{\TC}{\mathsf{TC}}

\DeclareMathOperator{\cd}{\mathrm{cd}}

\DeclareMathOperator{\cat}{\mathsf{cat}}
\DeclareMathOperator{\catAm}{\mathsf{cat}_{\mathrm{Am}}}
\DeclareMathOperator{\secat}{\mathsf{secat}}

\newcommand{\IF}{I_{\mathcal{F}}}
\newcommand{\IH}{I_{\mathcal{H}}}

\newcommand{\pr}{\mathrm{pr}}

\newcommand{\GG}{\mathcal{G}}

\newcommand{\cell}{\mathrm{cell}}

\newcommand{\ZG}{\mathbb{Z}[G]}
\newcommand{\ZH}{\mathbb{Z}[H]}
\newcommand{\ZS}{\mathbb{Z}[S]}
\newcommand{\RG}{\mathbb{R}[G]}
\newcommand{\RS}{\mathbb{R}[S]}

\newcommand{\CH}{C^{\mathcal{H}}}
\newcommand{\comp}{\mathrm{comp}}

\newcommand{\op}{\mathrm{op}}

\newcommand{\uu}{\mathfrak{u}}
\newcommand{\cc}{\mathfrak{c}}

\theoremstyle{plain}
\newtheorem{theorem}{Theorem}[section]
\newtheorem{prop}[theorem]{Proposition}
\newtheorem{lemma}[theorem]{Lemma}
\newtheorem{corollary}[theorem]{Corollary}

\newtheorem*{theorem*}{Theorem}
\newtheorem*{cor*}{Corollary}

\newtheorem*{question}{Question}

\theoremstyle{definition}
\newtheorem{definition}[theorem]{Definition}

\theoremstyle{remark}
\newtheorem{remark}[theorem]{Remark}
\newtheorem{example}[theorem]{Example}

\numberwithin{equation}{section}

\begin{document}
	\setlength{\parindent}{0cm}
	

\title[On open covers satisfying $\pi_1$-constraints and bounded Adamson cohomology]{On open covers of aspherical spaces satisfying $\pi_1$-constraints and bounded Adamson cohomology}

	\author[A. Espinosa Baro]{Arturo Espinosa Baro}
	\address{Institute of Mathematics of the Polish Academy of Sciences, Jana i J\c edrzeja \'Sniadeckich 8, 00-656 Warsaw, Poland.}
	\email{arturo.espinosabaro@gmail.com, aespinosabaro@impan.pl}
	
	\author[S. Mescher]{Stephan Mescher}
	\address{Institut f\"ur Mathematik,
		Martin-Luther-Universit\"at Halle-Wittenberg, Theodor-Lieser-Strasse 5, 06120 Halle (Saale), Germany.}
	\email{stephan.mescher@mathematik.uni-halle.de}

	
	\begin{abstract}
		Given a topological space $X$ and a family of subgroups $\mathcal{F}$ of $\pi_1(X)$, the $\mathcal{F}$-category of $X$ is given as one less than the cardinality of the smallest cover of $X$ by open subsets whose fundamental groups lie in $\mathcal{F}$. In this article we study $\FF$-categories of aspherical spaces and obtain cohomological lower bounds, a maximality result and bounded cohomology classes whose vanishing properties are crucial for determining $\FF$-categories. For this purpose, we study the bounded Adamson cohomology of a family of subgroups and discuss universality properties of bounded cohomology classes. We apply our techniques to derive some applications to the study of the monotonicity of $\FF$-categories under degree-one maps between manifolds and to provide a counterexample to a question of Capovilla, L\"oh and Moraschini.	\end{abstract}
	
	\maketitle
	
	\setcounter{tocdepth}{1}

	\let\oldtocsection=\tocsection
	
	\let\oldtocsubsection=\tocsubsection
	
	\let\oldtocsubsubsection=\tocsubsubsection
	
	\renewcommand{\tocsection}[2]{\hspace{0em}\oldtocsection{#1}{#2}}
	\renewcommand{\tocsubsection}[2]{\hspace{1em}\oldtocsubsection{#1}{#2}}
	\renewcommand{\tocsubsubsection}[2]{\hspace{2em}\oldtocsubsubsection{#1}{#2}}

	\tableofcontents
	
	
		\section*{Introduction}
\subsection*{Counting elements of open covers}	Throughout the last decades, several numerical homotopy invariants of topological spaces have been defined via counting minimal cardinalities of open covers with special properties. All of these invariants are motivated by the Lusternik-Schnirelmann (LS) category of a space $X$, a classical invariant of this kind defined as the minimal cardinality of an open cover of $X$ each of whose elements is contractible in $X$. A detailed introduction to LS category is found in \cite{CLOT}. 

Among those numerical invariants are those defined via covers with fundamental group constraints. Given a space $X$ and a class of groups $\GG$, one defines the $\GG$-category $\cat_\GG(X)$ of $X$ as one less than the minimal cardinality of an open cover of $X$, for which the fundamental group of each element of the cover maps into an element of $\GG$ under the inclusion. Likewise, one defines the $\FF$-category $\cat_{\FF}(X)$, where $\FF$ is a family of subgroups of $\pi_1(X)$, i.e. a set of subgroups closed under taking subgroups and conjugates. The former was first introduced for the class of amenable groups under the name of amenable category by J. C. G\'omez-Larra\~naga, F. Gonz\'alez-Acu\~na and W. Heil in \cite{GLGAH13} and generalized by the same authors in \cite{GLGAH}, while the latter was first discussed by P. Capovilla, C. L\"oh and M. Moraschini in \cite{CLM}.

In another line of research motivated by LS category, M. Farber has defined the topological complexity $\TC(X)$ of a space $X$ in \cite{FarberTC} as a numerical invariant motivated by the motion planning problem from robotics. It has become an interesting and active topic of research and it has been studied, modified and generalized ever since. Topological complexity (TC) and most of its variants are special cases of the notion of sectional category (or Schwarz genus) of a fibration as introduced by A. Schwarz in \cite{SchwarzGenus}.

Connections between the two lines of research have been studied by Capovilla, L\"oh and Moraschini in \cite{CLM}, where several new results for amenable category have been established and its relation to topological complexity has been investigated. Another connection was provided by M. Farber, M. Grant, G. Lupton and J. Oprea in \cite{FGLO17}, who studied relations between $\TC(X)$ and the $\mathcal{D}$-category of $X \times X$, where $\mathcal{D}$ denotes the subgroup family generated by the diagonal subgroup of $\pi_1(X)\times \pi_1(X)$. \medskip

\emph{Convention.} \enskip Throughout this article, we will use the \emph{reduced} versions of the definitions of the abovementioned invariants, i.e. we will write $\cat_{\GG}(X)=n$ or $\cat_\FF(X)=n$, where $n \in \NN_0$, if $n+1$ is the minimal cardinality of an open cover of $X$ with the desired property, and $\cat_{\FF}(X)=\infty$ if no such finite cover exists. This convention is in line with many recent articles on TC and related topics and is more convenient for studying cohomological lower bounds, while compared to most articles on open covers with $\pi_1$-constraints, e.g. \cite{CLM} or \cite{GLGAH}, our definition differs by $-1$ from the definitions used therein.

\subsection*{Invariants of aspherical spaces} We recall that a topological space $X$ is \emph{aspherical} if it is path-connected and if its homotopy groups satisfy $\pi_i(X)=\{0\}$ for all $i \geq 2$. As a consequence of Whitehead's theorem, any homotopy invariant of an aspherical CW complex $X$ depends only on its fundamental group $G:= \pi_1(X)$. The Eilenberg--Ganea theorem from \cite{EG57} together with the Stallings--Swan theorem states that $\cat(X) = \cd(\pi_1(X))$ for any such $X$, where $\cd$ denotes the cohomological dimension. 

Motivated by this result, it is	 natural to look for algebraic expressions of numerical homotopy invariants of aspherical spaces. For TC and its variants many estimates and computations for special cases have been obtained, see e.g. \cite{FGLO17}, \cite{FarMes20}, \cite{Grant22} and \cite{EFMO} and the references given therein. However, general algebraic formulas remain to be found in most of these settings. Concerning $\GG$-categories and $\FF$-categories, much less has been explored. A lower bound on the $\FF$-category of an aspherical space in terms of Bredon cohomology has been provided by Capovilla, L\"oh and Moraschini in \cite[Corollary 7.7]{CLM}. They further ask in the same article if the amenable category of a closed aspherical manifold can be described in terms of the non-vanishing of comparison maps from bounded cohomology.

Reviewing the existing literature one further realizes that the common methods used to study $\FF$-categories on the one hand and TC-style invariants of aspherical spaces on the other hand differ significantly. In this article we aim to bridge this methodological gap and to transfer algebraic methods used to study the TC of aspherical spaces to the setting of $\FF$-categories and to the context of bounded cohomology. Our approach is motivated by methods used in \cite{FarMes20}, \cite{BCE22} and \cite{EFMO}, which are in part based on A. Dranishnikov and Y. Rudyak's proof of the Berstein--Schwarz theorem from \cite{DranishRudyak}.

\subsection*{Amenable category and bounded cohomology}  The main reason for the interest in amenable category is its connection to simplicial volume and bounded cohomology, which is a very active field of research in geometric topology. It was initially sparked by M. Gromov's article \cite{Gro82}, see also \cite{FrigerioBook} for an algebraically-minded exposition.  The importance of amenable covers, i.e. covers by open subsets with amenable fundamental groups, lies in the amenable nerve theorem, which was first shown by Gromov and later reproven by C. L\"oh and R. Sauer in \cite{LoehSauer} using classifying spaces of families of subgroups. An interesting open problem in this context is the following question of Gromov.

\begin{question}[Gromov, {\cite[p. 232]{Gro93}}]
Let $M$ be a closed oriented aspherical manifold whose simplicial volume vanishes. Does the Euler characteristic $\chi(M)$ vanish as well?
\end{question}

A very readable and thorough survey about this question and its context has been written by C. L\"oh, M. Moraschini and G. Raptis, see \cite{LMR}. Its relation to amenable category is given by a result of R. Sauer. In \cite[Theorem B]{Sauer}, Sauer has shown for a closed oriented aspherical manifold $M$ that if $\catAm(M)<\dim M$, then all $L^2$-Betti numbers of its universal cover vanish, which implies that $\chi(M)=0$ by a result from $L^2$-cohomology. 

Here, it is important to have in mind that by a general upper bound for LS categories the inequality $\cat_{\Am}(M) \leq \dim M$ always holds in this setting. Thus, Sauer's theorem can be interpreted as a result about the non-maximality of amenable category. In the light of Gromov's question, it is thus interesting to study the relation between (non-)maximality results of amenable category and simplicial volume.

\subsection*{Adamson cohomology and (Adamson)-Johnson classes } In \cite{Adamson}, I. Adamson introduced a relative version of group cohomology with respect to a given subgroup, henceforth called Adamson cohomology. It was used in \cite{BCE22} by Z. B\l{}aszczyk, J. Carrasquel-Vera and the first-named author of this article to provide lower bounds for sectional categories of subgroup inclusions. 

The construction of Adamson cohomology can be generalized to a relative cohomology of a group with respect to a family of subgroups. This notion of Adamson cohomology for families was introduced by M. Grant, K. Li, E. Meir and I. Patchkoria in \cite{GLMP}. It provides an algebraic model for the classifying space of a family of subgroups and can be identified with a special case of Bredon cohomology, see \cite[Proposition 2.12]{GLMP}.

In this article we will introduce a notion of \emph{bounded} Adamson cohomology of families of subgroups. While the definition of bounded cohomology of discrete groups in terms of resolutions is a standard construction, see e.g. \cite[Chapter 4]{FrigerioBook}, the definition of bounded Adamson cohomology requires a type of relatively injective resolution introduced by K. Li in \cite{LiBounded}. 

We will further introduce so-called Berstein cohomology classes and Adamson canonical classes relative to families of subgroups and their counterparts in bounded cohomology. The former have been introduced in \cite{BCE22} relative to a single subgroup and our definition is a straightforward generalization. 

In the bounded cohomology of a discrete group $G$ we will introduce \emph{Johnson classes}, which are bounded cohomology classes of the form 
$$\gamma_k \in H^k_b(G;\II^k_G),\quad k \in \NN,$$ 
where $\II^k_G$ denotes the bidual space of the $k$-fold tensor power of the real augmentation ideal of $G$, i.e. the kernel of the augmentation $\varepsilon_G\colon  \RG \to \RR$.

The class $\gamma_1$ corresponds to a class introduced by B. Johnson in \cite{Johnson}. As Johnson has shown, this class detects the amenability of a group, see \cite[Section 3.4]{FrigerioBook} for an exposition. Given a family of subgroups $\FF$ of $G$, we shall further introduce bounded Adamson cohomology classes
$$\Gamma_{\FF,k} \in H^k_b([G:\FF];\II^k_\FF),\quad k \in \NN,$$ 
where $\II^k_\FF$  is a certain Banach $G$-module associated with $\FF$. As we shall see below, these classes are closely related to $\FF$-categories.

\subsection*{Main results} Throughout the following let $G$ be a group and let $\FF$ be a family of subgroups of $G$. Using methods from obstruction theory we obtain a complete cohomological characterization of the maximality of $\FF$-categories for aspherical spaces, as well as an analogue of such a result for spaces that are not necessarily aspherical.  

\begin{theorem*}[Theorem \ref{thm:CostaFarberCatF} and Corollary \ref{cor:CostaFarberCatF}]
Let $G$ be a discrete group, $\FF$ a family of subgroups of $G$ and $\omega_\FF \in H^1(G;I_\FF)$ the Berstein class of $G$ relative to $\FF$. Let $X$ be a connected CW complex with $\pi_1(X)=G$ and $\dim(X) \geq 3$. Let $\mathfrak{c}_X\colon X \to BG$ be a classifying map of the universal cover of $X$. Then
$$\cat_\FF(X) = \dim(X) \qquad \Leftrightarrow \quad \mathfrak{c}_X^*(\omega_\FF^{\dim(X)}) \neq 0 \in H^{\dim(X)}(X;\IF^{\dim(X)})$$
where $\omega_\FF^{k}$ denotes the $k$-fold power of $\omega_\FF$ with respect to the cup product.
\medskip

Furthermore, if $G$ is torsion-free and $X$ is aspherical with $\dim X=\cd(G)$, then $$\cat_\FF(X) = \cd(G) \qquad \Leftrightarrow \quad \omega_\FF^{\cd(G)} \neq 0 \in H^{\cd(G)}(G;\IF^{\cd(G)}).$$ 
\end{theorem*}

We refrain from giving the definition of the coefficient module $I_\FF$ at this point. The theorems are a generalization of previous results on topological complexity and sectional categories of subgroup inclusions, see \cite[Theorem 7]{Costa} and \cite[Theorem 2.7]{BCE22}. The upshot is that we can describe $\cat_\FF(X)$ by the existence of equivariant maps between $G$-CW complexes, whose obstruction classes can be identified with the powers of $\omega_\FF$.\medskip

Bridging the gap between (Adamson-)Johnson classes and the $\FF$-categories of aspherical spaces, we further derive a lower bound for the latter in terms of the former, namely the following.	
	\begin{theorem*}[Theorem \ref{TheoremcatFFlowergamma}]
		Let $\FF$ be an $H^*_b$-admissible family of subgroups of $G$ and let $X$ be an aspherical space with $\pi_1(X)=G$. Then 
	$$\cat_\FF(X) \geq \sup \{k \in\NN \mid \comp^k_{\II_G^k}(\gamma_k) \neq 0\}.$$
	\end{theorem*}
We will recall the notion of $H^*_b$-admissibility in the course of the article. Here, we want to point out that the family of \emph{amenable} subgroups of $G$ is indeed $H^*_b$-admissible, so that the theorem provides a lower bound for amenable category. Note that the result is similar to Gromov's Vanishing Theorem which only deals with the amenable family, but extends beyond the aspherical case. An analogous lower bound for $\cat_\FF$ for $H^*_b$-admissible families by comparison maps with $\RR$-coefficients was obtained by C. L\"oh and R. Sauer in \cite[Theorem 5.3]{LoehSauer}. We will further provide in Example \ref{ExampleCircleBundle} a case of an aspherical space for which $\comp^2_{\II^2_G}(\gamma_2)\neq 0$, but $\comp^2_\RR =0$, namely the total space of a circle bundle over a hyperbolic surface with non-vanishing Euler class. This shows that our result indeed improves the one of L\"oh and Sauer and provides a counterexample to a question raised in \cite[Remark 3.16]{CLM}.

Besides connections between (Adamson-)Johnson classes, Berstein and Adamson canonical classes, the crucial observation for the proof of this lower bound is that (Adamson-)Johnson classes are \emph{universal} for dual coefficient modules in the sense of the following result. 

\begin{theorem*}[Theorem \ref{TheoremBoundedUniv}]
Let $V$ be a normed $G$-module and let $k \in \NN$.
	\begin{enumerate}[a)]
		\item For each $u \in H^k_b([G:\FF];V')$ there exists a bounded $\RG$-homomorphism $\psi\colon \II^k_{\FF}\to V'$ with 
		$u = \psi_*(\Gamma_{\FF,k}).$
		\item For each $v \in H^k_b(G;V')$ there exists a bounded $\RG$-homomorphism $\eta\colon \II^k_G \to V'$ with $v=\eta_*(\gamma_k)$.
	\end{enumerate}
\end{theorem*}

In addition to the above lower bound we also provide new upper bounds for $\FF$-categories of aspherical spaces in Theorem \ref{TheoremcatGupper} and Proposition \ref{PropcatGnormal} by relating them to sectional categories of subgroup inclusions and applying results from \cite{EFMO}. We make use of our techniques to establish as well a criterion to study the monotonicity of $\FF$-categories under maps between manifolds, in terms of pullbacks of powers of relative Berstein classes by classifying maps. This yields results on the monotonicity of $\GG$-categories with respect to degree-one maps between closed connected oriented $n$-manifolds.

\begin{theorem*}[Corollary \ref{cor:monotonicitydegoneF}, Corollary \ref{cor:monotonicitydegoneG}]
Let $\GG$ be an isq-class of groups, let $n \in \NN$ and let $f \colon M \to N$ be a map  between closed connected oriented smooth $n$-manifolds. Put $k := \cat_{\GG}(N)$ and let $\GG_N=\{H \leq \pi_1(N) \mid H \in \GG\}$. Then $\cat_{\GG}(M) \geq \cat_{\GG}(N) $ if one of the following conditions holds:
\begin{enumerate}[(a)]
	\item $\cc^*_N(\omega^k_{\GG_N}) \neq 0$ and $\deg f= \pm 1$,
	\item $\cc^*_N(\omega^k_{\GG_N, \RR}) \neq 0$ and $\deg f \neq 0$. 
\end{enumerate}   
 If $n \geq 3$ and $\deg f = \pm 1$, then $\cat_\GG(N) = n$ further implies $\cat_\GG(M) = n$. 
\end{theorem*}

In particular, this result gives a positive answer to the maximal case of \cite[Question 6.3]{CLM}.

 \medskip

In our cohomological considerations, we further obtain a new necessary and sufficient criterion for the non-vanishing of simplicial volume. The \emph{simplicial volume} of a closed oriented manifold $M$, denoted by $\|M\|$,  is a homotopy invariant taking values in $[0,+\infty)$ that was introduced by M. Gromov in \cite{Gro82}. It is roughly given by the infimum of the $\ell^1$-norms of $n$-cycles representing the real fundamental class of $M$, see \cite{LoehSimplicial} or \cite[Chapter 7]{FrigerioBook} for an overview of its definition and most important properties. Building on Gromov's duality principle about simplicial volume and bounded cohomology, we establish the following necessary and sufficient criterion on the non-vanishing of the simplicial volume.

\begin{theorem*}[Corollary \ref{CorSimpVolBerstein}]
Let $M$ be a closed connected orientable aspherical manifold with $G=\pi_1(M)$ and let $n=\dim M$. Then $$\|M\|>0 \qquad \Leftrightarrow \qquad \big\|b^n_{\mathbb R}\cap[M] \big\|_{(J^n_G)_G}>0 ,$$   
	where $\|M\|$ denotes the simplicial volume and $[M]\in H_n(M;\ZZ)$ denotes the fundamental class  of $M$.
\end{theorem*}
Here, $\|\cdot \|_{(J^n_G)_G}$ is the seminorm on the coinvariants of the coefficient module $J^n_G$ of the $n$-th power of the real Berstein class $b^n_\RR$ that is induced by a certain norm that we make precise in the course of the article. Note that the right-hand assertion does not invoke bounded cohomology.

\subsection*{Overview of the article} In the first section we recall Adamson cohomology for families of subgroups and construct a certain relatively projective resolution that we use in our computations throughout the article. Afterwards, we introduce Berstein classes and Adamson canonical classes relative to subgroup families and discuss their properties in Section 2. We further investigate how to pass from resolutions over $\ZG$ to those over $\RG$ for $\RG$-modules as coefficients. In Section 3, we transfer our constructions to bounded cohomology and introduce the abovementioned bounded Adamson cohomology. Then we introduce Johnson classes and Adamson--Johnson classes in Section 4 and establish their connections to Berstein and Adamson canonical classes via comparison maps. We introduce $\FF$-categories in Section 5 and show how the previous constructions yield new cohomological lower bounds for these invariants. Finally, in Section 6, we will recall the notion of isq-classes of groups, and show how the previous techniques apply in this more restrictive setting, including an application to the monotonicity problem for maps of non-zero degree between manifolds.

\subsection*{Notation and conventions} Throughout this article let $G$ be a group equipped with the discrete topology. Given a commutative ring with unit $R$ we will call a left $R[G]$-module simply an $R[G]$-module. We  adopt the convention that for any two $\ZG$-modules $A$ and $B$ their tensor product over $\ZZ$ will be denoted by $A \otimes B$, while tensor products over other rings will always carry the ring as an index. We further consider $A \otimes B$ and $\Hom_{\ZZ}(A,B)$ as $\ZG$-modules with respect to the diagonal $G$-actions. Given a CW complex $X$ we shall denote its $k$-skeleton by $X_k$ for every $k \in \NN_0$. 

\subsection*{Tool use} The authors used Claude Opus 5.5 for proofreading and checking grammar, for locating relevant references and for occasional general mathematical queries. The authors independently developed and verified all mathematical arguments, results, proofs and examples, and are responsible for their correctness. All of the writing and typesetting has been carried out by the authors.

\section*{Acknowledgements}
The authors thank John Oprea for valuable feedback on an earlier draft of the manuscript and Piotr Nowak for helpful comments. During the first parts of this work, A. E. B. was supported by National Science Center, Poland, research grant UMO-2022/45/N/ST1/02814 and by a doctoral scholarship of the Adam Mickiewicz University of Pozna\'{n}. St. M. further thanks the Adam Mickiewicz University for its kind hospitality on a research visit in which parts of the work on this article were carried out.

\section{Adamson cohomology with respect to subgroup families}

\subsection{Definition and basic notions}
We want to carry out the construction of Adamson cohomology for families of subgroups as first discussed in \cite{GLMP}, for which we need the following basic notions.   

\begin{definition}
	\begin{enumerate}[a)]
		\item We call a non-empty set $\FF$ of subgroups of $G$ a \emph{family of subgroups of $G$} if it is closed under taking subgroups and conjugates. 
		\item Given a family of subgroups $\FF$ of $G$, we call $\HH \subset \FF$ a \emph{generating set of $\FF$} and write $\FF=\langle\HH\rangle$ if $\FF$ is the smallest family of subgroups containing $\HH$, i.e. if 
		$$\FF= \{H \leq G \mid \exists K \in \HH, \ g \in G: \ H \leq gKg^{-1}\} .$$
	\end{enumerate}
\end{definition}

We continue by introducing notions from relative homological algebra that will be used in the definition of Adamson cohomology and that have been defined by K. Li in \cite{LiBounded}.

\begin{definition}
	Let $\FF$ be a family of subgroups of $G$ and let $R$ be a commutative ring with unit. 
	\begin{enumerate}[a)]
		\item Let $A$ and $B$ be $R[G]$-modules. A map $f\in \Hom_{R[G]}(A,B)$ is called \emph{$\FF$-strongly surjective} if for each $H \in \FF$ there exists $\alpha_H \in \Hom_{R[H]}(B,A)$ with $f\circ \alpha_H = \id_B$.  
		\item 	An $R[G]$-module $M$ is \emph{relatively $\FF$-projective over $R[G]$} if the following holds: for all $R[G]$-modules $A$ and $B$, all $\FF$-strongly surjective $f \in \Hom_{R[G]}(A,B)$ and all $\varphi \in \Hom_{R[G]}(M,B)$ there exists $\widetilde{\varphi} \in \Hom_{R[G]}(M,A)$ with $f \circ \widetilde{\varphi}=\varphi$.
		\item Let $A$ be an $R[G]$-module. A \emph{relatively $\FF$-projective resolution of $A$ over $R[G]$} is an exact chain complex of $R[G]$-modules
		$$ \dots \longrightarrow P_k \stackrel{p_k}{\longrightarrow} P_{k-1} \longrightarrow \dots \longrightarrow P_1 \stackrel{p_1}{\longrightarrow} P_0 \stackrel{p_0}{\longrightarrow}  A \longrightarrow 0,$$ for which $P_k$ is relatively $\FF$-projective for each $k \in \NN_0$.
		\item An exact chain complex $(C_k, \partial_k)_{k \in\ZZ}$ of $R[G]$-modules is called \emph{$\FF$-strong} if for each $H \in \FF$ there exists a contraction $\tau_{H,*}\colon C_* \to C_{*+1}$ of $C_*$ as $R[H]$-modules, i.e. for which $ \tau_{H,k} \in \Hom_{R[H]}(C_k,C_{k+1})$  and
		$$\partial_{k+1} \circ \tau_{H,k} + \tau_{H,k-1} \circ \partial_k=\id_{C_k}\qquad \forall k \in \ZZ.$$ 
	\end{enumerate}
\end{definition}

The following lemma shows that for the $\FF$-strongness of an exact chain complex it suffices to check the existence of contractions only with respect to a generating set of the family. 

\begin{lemma}
	\label{LemmaGFexactgener}
	Let $\FF$ be a family of subgroups of $G$, let $\HH$ be a generating set of $\FF$ and let $R$ be a commutative ring with unit. Let $(C_k,\partial_k)_{k \in \ZZ}$ be an exact chain complex of $R[G]$-modules. If for each $H \in \HH$ there exists a contraction $\tau_{H,*}\colon C_*\to C_{*+1}$ of $(C_*,\partial_*)$ as $R[H]$-modules, then $(C_*,\partial_*)$ is $\FF$-strong. 
\end{lemma}
\begin{proof}
	We need to show that for $K \in \FF$ there exists a contraction $\tau_{K,*}\colon C_*\to C_{*+1}$ as $R[K]$-modules. Fix $K$ and let $g \in G$ and $H \in \HH$, such that $K \leq gHg^{-1}$. 
	For each $s \in \ZZ$ let	
	$$\tau_{K,s}\colon C_s \to C_{s+1}, \quad \tau_{K,s}(x) := g\tau_{H,s}(g^{-1}x).$$
	Then for all $k \in K$ and $x \in C_s$ we compute from the $R[H]$-linearity of $\tau_{H,s}$ that, since $g^{-1}kg \in H$, 
	\begin{align*}
		\tau_{K,s}(k\cdot x) &= g \tau_{H,s}(g^{-1}k\cdot x) = g\tau_{H,s}(g^{-1}kg\cdot g^{-1} x) 
		= gg^{-1}kg \tau_{H,s}(g^{-1}x)= k \tau_{K,s}(x),
	\end{align*}
	so $\tau_{K,s} \in \Hom_{R[K]}(C_s,C_{s+1})$ for each $s \in \ZZ$. Moreover, since the $\tau_{H,s}$ form a contraction of $C_*$ as $R[H]$-modules, we compute for all $x \in C_s$ that
	\begin{align*}
		(\partial_{s+1}\circ \tau_{K,s})(x)+(\tau_{K,s-1}\circ \partial_s)(x) 
		&=\partial_{s+1}(g\tau_{H,s}(g^{-1}x))+ g \tau_{H,s}(g^{-1} \partial_s(x)) \\
		&= g(\partial_{s+1}\circ \tau_{H,s} + \tau_{H,s-1}\circ \partial_s ) (g^{-1}x) = gg^{-1}x=x.
	\end{align*}
	Hence, $\tau_{K,*}\colon C_* \to C_{*+1}$ is a contraction of $C_*$ as $R[K]$-modules. 
\end{proof}

The following is a consequence of the immediate analogue of the fundamental lemma of homological algebra, carried out in a relative setting. We omit its proof, since it is given along the lines of the proof of the original statement.

\begin{lemma}\label{fundamentalrelhom}
	Let $\FF$ be a family of subgroups of $G$, let $R$ be a commutative ring with unit, let $A$ be an $R[G]$-module and let $P_*$ and $Q_*$ be relatively $\FF$-projective $\FF$-strong resolutions of $R$ over $R[G]$. Then 
	$$H^*(\Hom_{R[G]}(P_*,A)) \cong H^*(\Hom_{R[G]}(Q_*,A)).$$
\end{lemma}

In the remainder of this section, we will only use the case $R=\ZZ$ of the preceding results. \medskip

Adamson cohomology was originally defined by I. Adamson \cite{Adamson} as a variant of group cohomology reflecting information of the structure relative to the inclusion of a subgroup. It was generalized by G. Hochschild in \cite{Hoch} to develop a theory of relative homological algebra. The following definition for a family of a subgroups of $G$ first appeared in \cite{GLMP}.

\begin{definition}
	Let $\FF$ be a family of subgroups of $G$, let $A$ be a $\ZG$-module and let $(P_*,d_*)$ be a relatively $\FF$-projective $\FF$-strong resolution of $\ZZ$ over $\ZG$. The \emph{Adamson cohomology relative to $\FF$ with coefficients in $A$} is given by 
	$$H^*([G:\FF];A) := H^*(\Hom_{\ZG}(P_*,A)).$$ 
\end{definition}

\begin{remark}
Similar to the projective setting, one can define a notion of relatively $\FF$-injective $\FF$-strong resolution of a $\ZG$-module and define the Adamson cohomology relative to $\FF$ with coefficients in $A$ by $$ H^*([G:\FF];A) = H^*(\Hom_{\ZG}(\ZZ,J^*)), $$ where $J^*$ is a relatively  $\FF$-injective $\FF$-strong resolution of $A$. The equivalence of the two approaches follows as in the classical case from the bi-balancedness of both functors $\Hom_{\ZG}(\cdot,\cdot)$ and $\Ext_{\ZG}(\cdot,\cdot)$.
\end{remark}

\subsection{A construction of relatively $\FF$-projective resolutions} 
We need to recall some basic terminology. Let $S$ be a $G$-set and consider
$$ \ZZ[S]= \bigoplus_{s \in S} \ZZ \cdot s $$
as a $\ZG$-module. The \emph{augmentation map with respect to $S$} is given by 
$$\varepsilon_S\in \Hom_{\ZG}(\ZS, \ZZ), \qquad  \varepsilon_S \Big(\sum_{s \in S} n_s s \Big) = \sum_{s \in S} n_s.$$ Furthermore, the \emph{augmentation ideal} associated to $S$ is given by $$I_S:= \ker \varepsilon_S,$$
and seen as a $\ZG$-submodule of $\ZZ[S]$.

\begin{definition} \label{defrelgroupring}
	Let $\HH$ be a non-empty set of subgroups of $G$.  We define $$G/\HH:= \bigsqcup_{H \in \HH} G/H$$ seen as a $G$-set by summandwise considering the $G$-action on the respective orbit space. We further put
	$$I_{\HH} := I_{G/\HH} \subset \ZZ[G/\HH], \qquad \varepsilon_\HH := \varepsilon_{G/\HH} \in \Hom_{\ZG}(\ZGHH,\ZZ).$$ 
\end{definition}

Since $G/\HH$  decomposes as a union of $G$-sets, we obtain the obvious decomposition
	\begin{equation}
		\label{EqZGFsum}
		\ZGHH \otimes M = \bigoplus_{H \in \HH} \ZZ[G/H]\otimes M,
	\end{equation}
as a sum of $\ZG$-modules for all $\ZG$-modules $M$. In the following, we will generalize constructions carried out in \cite[Section 3.2]{BCE22} in the case of a single subgroup of $G$. 

\begin{example}
\label{ExampleStandard}
The natural example of a relatively $\FF$-projective resolution of $\ZZ$ over $\ZZ[G]$, in full analogy to the non-relative case, is the \emph{$\FF$-relative standard resolution} 
\begin{equation}
	\label{EqFFrelstandard}
	\mathcal{E}_{*}^{\FF} = \left( \ZZ[(G/\FF)^{*+1}], b^{\FF}_{*} \right),
\end{equation} where $b^\FF_k\colon \ZZ[(G/\FF)^{k+1}] \to \ZZ[(G/\FF)^{k}]$ is for each $k \in \NN_0$ given by $\ZZ$-linearly extending $$ b^{\FF}_{k}(g_0H_0, \dots, g_kH_k) = \sum_{i=0}^{k} (-1)^i (g_0H_0, \dots, g_{i-1}H_{i-1},g_{i+1}H_{i+1}, \dots, g_kH_k).$$
The isotropy group of any $(g_0H_0,\cdots,g_kH_k)$ is given by $\bigcap_{i=0}^k g_iH_ig_i^{-1} \in \FF$, since $\FF$ is closed under taking conjugates and subgroups. Thus, $\ZZ[(G/\FF)^{k+1}]$ is isomorphic to a  sum of permutation modules of the form $\ZZ[G/K]$, where $K \in\FF$. Thus, $\ZZ[(G/\FF)^{k+1}]$ is relatively $\FF$-projective for every $k \in \NN_0$. Since for each $H \in \FF$  we can define the contraction $$\tau_{H,k}(g_0 H_0, g_1 H_1, \cdots, g_k H_k) = (H,g_0H_0, \cdots, g_kH_k)$$ for $k \geq 0$ and $\tau_{H,-1}(1) = (H)$, this chain complex is indeed a relatively $\FF$-projective $\FF$-strong resolution of $\ZZ$ over $\ZG$.
 
 In particular, for the trivial family $\FF = \{1\}$ we denote $$ \mathcal{E}_{*}^G = \left( \ZZ[G^{*+1}], b_{*} \right), $$ the standard resolution of $\ZZ$ as a $\ZG$-module.
 \end{example}

We want to introduce and work with another relatively $\FF$-projective $\FF$-strong resolution, which generalizes a resolution constructed in \cite[Section 3]{DranishRudyak}. Further variants of this construction have also been used in \cite{FarMes20}, \cite{BCE22} and \cite{EFMO} in the study of topological complexity and sectional categories of subgroup inclusions.  It will allow us simple descriptions of certain Adamson cohomology classes that will be used in the course of this article.  

\begin{lemma}
	\label{LemmaZGFproj}
	Let $\FF$ be a family of subgroups of $G$, let $\HH\subset \FF$ and let $M$ be a $\ZG$-module. Then $\ZGHH \otimes M$ is relatively $\FF$-projective over $\ZG$. 
\end{lemma}
\begin{proof}
	Let $A$ and $B$ be $\ZG$-modules and $p \in \Hom_{\ZG}(A,B)$, such that for each $H \in \FF$ there exists $q_H \in \Hom_{\ZZ[H]}(B,A)$ with $p\circ q_H=\id_B$. Let $\varphi \in \Hom_{\ZG}(\ZGHH\otimes M,B)$. For each $H \in \HH$ we put $\varphi_H:=\varphi|_{\ZGH \otimes M}$. As discussed in \cite[Section 3.2]{BCE22}, $\ZGH \otimes M$ is $(G,H)$-projective, which implies in the present situation that there exists $\overline{\varphi}_H \in \Hom_{\ZG}(\ZGH\otimes M,A)$ making the following diagram commutative:
	
	$$\begin{tikzcd}
		&  \ZGH\otimes M \ar[d,"\varphi_H"] \ar[dl,dotted,"\overline{\varphi}_H"'] \\
		A \ar[r, "p"] & B \ar[l, bend left=20, "q_H"] \ar[r]& 0 .
	\end{tikzcd}$$
	For each $H \in \HH$ we put $\widetilde\varphi_H\colon  \ZGHH \otimes M\to A$, $\widetilde\varphi_H:= \overline\varphi_H \circ \pr_H$,  where $\pr_H$ denotes the projection onto the corresponding summand in \eqref{EqZGFsum}. Since both $\pr_H$ and $\overline\varphi_H$ are $\ZG$-homomorphisms, $\widetilde\varphi_H$ is a $\ZG$-homomorphism as well. We define 
	$$\widetilde\varphi \in \Hom_{\ZG}(\ZGHH \otimes M, A), \qquad \widetilde\varphi := \sum_{H \in \HH} \widetilde\varphi_H.$$
	Then
	$$p \circ \widetilde\varphi = \sum_{H \in \HH} p \circ \widetilde{\varphi}_H= \sum_{H \in \HH}p \circ \overline\varphi_H \circ \pr_H= \sum_{H \in \HH} \varphi_H \circ \pr_H= \varphi,$$
	so $\widetilde\varphi$ lifts $\varphi$ along $p$. Since $A$, $B$ and $p$ were chosen arbitrarily, this shows that $\ZGHH \otimes M$ is relatively $\FF$-projective over $\ZG$.
\end{proof}

In the following discussion we transfer constructions of \cite{DranishRudyak} to our setting. Given a $G$-set $S$ we consider the short exact sequence
\begin{equation}
	\label{EqBasicSES}
	0 \longrightarrow I_S \stackrel{\iota_S}\longrightarrow \ZZ[S] \stackrel{\varepsilon_S}{\longrightarrow} \ZZ \longrightarrow 0,
\end{equation}
where $\iota_S\colon I_S \to \ZZ[S]$ denotes the inclusion. For each $k \in \NN$ we put 
$$I^k_S := I^{\otimes k}_S = \underbrace{I_S \otimes I_S \otimes \ldots \otimes I_S}_{k \text{ times}}.$$
Since $I^k_S$ is $\ZZ$-free, tensoring the sequence from  \eqref{EqBasicSES} with $I^k_S$ yields a short exact sequence 
\begin{equation}
	\label{EqBasicSESk}
	0 \longrightarrow I_S^{k+1} \stackrel{f_k^S}\longrightarrow \ZZ[S]\otimes I^k_S \stackrel{g_k^S}{\longrightarrow} I^k_S \longrightarrow 0,
\end{equation}
where $f_k^S := \iota_S \otimes \id_{I^k_S}$ and $g_k^S := \varepsilon_S \otimes \id_{I^k_S}$. Splicing the short exact sequences from \eqref{EqBasicSESk} for all values of $k$ we obtain an exact sequence of $\ZG$-modules
\begin{equation}
	\label{Eqspliced}
	\ldots \longrightarrow\ZZ[S] \otimes I^k_S \stackrel{d^S_{k}}{\longrightarrow} \ZZ[S] \otimes I^{k-1}_S \longrightarrow \ldots \longrightarrow \ZZ[S]\otimes I_S \stackrel{d_1^S}\longrightarrow \ZZ[S] \stackrel{\varepsilon_S}{\longrightarrow} \ZZ \to 0,
\end{equation}
where 
$$d_{k}^S \colon \ZZ[S] \otimes I^k_S \to \ZZ[S] \otimes I^{k-1}_S, \qquad d_{k}^S := f^S_{k-1} \circ g^S_k,$$
for each $k \in \NN$. \medskip

\begin{lemma} \label{Adamson:projective}
Let $\FF$ be a family of subgroups of $G$, let $\HH$ be a generating set for $\FF$ and put 
$$\CH_k := \ZGHH \otimes I^k_\HH \qquad \forall k \in \NN_0.$$
Then
	\begin{equation}
		\label{EqResZGF}
		\dots \longrightarrow \CH_k \stackrel{d^{\HH}_k}{\longrightarrow} \CH_{k-1} \longrightarrow \dots \longrightarrow \CH_0 \stackrel{\varepsilon_\HH}{\longrightarrow} \ZZ \to 0 
	\end{equation}
	is a relatively $\FF$-projective $\FF$-strong resolution of $\ZZ$ over $\ZG$. 
\end{lemma}

\begin{proof}
	By Lemma \ref{LemmaZGFproj}, $\CH_k$ is relatively $\FF$-projective over $\ZG$ for each $k \in \NN_0$. It is evident that \eqref{EqResZGF} is exact as it is obtained via splicing the short exact sequences from \eqref{EqBasicSESk}, so that $(\CH_*,d^\HH_*)$ is a relatively $\FF$-projective resolution of $\ZZ$ over $\ZG$. By Lemma \ref{LemmaGFexactgener}, to show $\FF$-strongness it suffices to show that for each $H \in \HH$ there is a contraction of $\CH_*$ as $\ZH$-modules. 
	
	Let $H \in \HH$, $k \in \NN_0$ and define $\tau_{H,k}: \CH_k \to \CH_{k+1}$ by $\ZZ$-linearly extending 
	$$\tau_{H,k}(x \otimes y) = H \otimes (x- \varepsilon_\HH(x)\cdot H)\otimes y \qquad \forall x \in \ZGHH, \ y \in \IH^k. $$
	as well as $\tau_{H,-1}\colon \ZZ \to C_0$,  $\tau_{H,-1}(s) = s \cdot H$.
	We compute for all $h \in H$, $x \in \ZGHH$ and $y \in \IH^k$ that
	\begin{align*}
		\tau_{H,k}(h \cdot x \otimes y) &= \tau_{H,k}(hx\otimes hy) \\
		&= H \otimes (hx-\varepsilon_\HH(hx)\cdot H)\otimes hy \\
		&= h \cdot( H \otimes (x-\varepsilon_{\HH}(x)H)\otimes y) = h \cdot \tau_{H,k}(x \otimes y).
	\end{align*}
	Hence, $\tau_{H,k} \in \Hom_{\ZH}(C^\HH_k,C^\HH_{k+1})$. We further compute for all $x \in \ZGHH$ and $y \in \IH^k$ that
	\begin{align*}
		(d^{\HH}_{k+1}\circ \tau_{H,k})(x\otimes y) &= d^{\HH}_{k+1}(H \otimes (x-\varepsilon_\HH(x)\cdot H)\otimes y) \\
		&=x \otimes y - \varepsilon_\HH(x)\cdot H \otimes y.
	\end{align*}
	Note that by definition of $\tau_{H,k-1}$, it holds that $\tau_{H,k-1}(f^\HH_{k-1}(y)) = H \otimes y$, which yields
	\begin{align*}
		(d^{\HH}_{k+1}\circ \tau_{H,k})(x\otimes y)  &=x \otimes y - \varepsilon_{\HH}(x) \cdot \tau_{H,k-1}(f_{k-1}^{\HH}(y)) \\
		&= x \otimes y - (\tau_{H,k-1} \circ d^{\HH}_k)(x\otimes y).
	\end{align*}
	Thus, we have shown that $d^{\HH}_{k+1} \circ \tau_{H,k} + \tau_{H,k-1} \circ d^{\HH}_{k} = \id_{C^\HH_k}$
	for all $k\in \NN_0$, where $d^{\HH}_0:=\varepsilon_{\HH}$. We further compute for all $r \in \ZZ$ that $(\varepsilon_\HH \circ \tau_{H,-1})(r)= \varepsilon_\HH(r\cdot H) = r$.
Thus, the maps $(\tau_{H,k})_{k \geq -1}$ form a contraction of the resolution, which shows that $(\CH_*,d^{\HH}_*)$ is a relatively $\FF$-projective $\FF$-strong resolution. \end{proof}

We derive from Lemma \ref{Adamson:projective} that for a family of subgroups $\FF$ of $G$, a generating set $\HH$ of $\FF$ and a left $\ZG$-module $A$, it holds that
$$H^k([G:\FF];A) = H^k(\Hom_{\ZG}(\CH_*,A),(d^\HH_*)^*) \qquad \forall k \in \NN_0.$$

\begin{remark}
\label{RemarkReso}
	\begin{enumerate}
		\item In the case of $\FF=\{1\}$ we observe that $(C^{\{1\}}_*,d_*^{\{1\}})\cong (\ZG\otimes I_G^*,d_*)$ and recover from Lemma \ref{Adamson:projective} the result of \cite{DranishRudyak} that $(\ZG\otimes I_G^*, d_*)$ forms a projective resolution of $\ZZ$ over $\ZG$. 
		\item If $\FF=\left<H\right>$, i.e. if $\FF$ is generated by a single subgroup $H \leq G$, then taking $\HH=\{H\}$ yields the resolution $(\ZGH \otimes I_{G/H}^*,d^{\{H\}}_*)$, which has been used in \cite{BCE22} to compute the Adamson cohomology of $G$ relative to $H$. Thus, Lemma \ref{Adamson:projective} particularly implies that
		$$H^*([G:\langle H \rangle];A) \cong H^*([G:H];A)$$
		for each left $\ZG$-module $A$.
		\item  Observe that we can define chain maps comparing the constructed resolutions with the (relative) standard resolutions by taking  \begin{equation}\label{chainmapsresolutions}
			\psi_G^k \colon \ZZ[G^{k+1}] \rightarrow \ZG \otimes I^{k}_G \qquad \mbox{and} \qquad \psi^k_\FF \colon \ZZ[(G/\FF)^{k+1}] \rightarrow \ZGF \otimes I^k_\FF
		\end{equation} respectively by $$ \psi^k_G(g_0, g_1, \cdots, g_k) := g_0 \otimes (g_1-g_0) \otimes \cdots \otimes (g_k-g_{k-1})$$ and $$ \psi^k_\FF(g_0H_0, \cdots, g_kH_k) := g_0H_0 \otimes (g_1H_1 - g_0H_0) \otimes \cdots \otimes (g_kH_k - g_{k-1}H_{k-1}). $$ Using a relative version of the fundamental lemma from homological algebra, one can show that
		 $\psi^*_\FF$ defines a chain homotopy equivalence of $\ZG$-chain complexes.
	\end{enumerate}	
\end{remark}

\section{Berstein and Adamson canonical classes relative to families}

Throughout this section let $\FF$ be a fixed family of subgroups of $G$. 

\subsection{Constructions of cohomology classes} In the following we shall consider the relatively $\FF$-projective resolution $(C^\HH_*,d_*^\HH)$ of $\ZZ$ over $\ZG$ from the previous section in the case of $\HH=\FF$. We consider for each $\ZG$-module $A$ the Adamson cohomology $H^*([G:\FF];A)$ as the cohomology of the cochain complex $(C^*([G:\FF];A),d^*_{\FF})$, where
\begin{align*}
	&C^k([G:\FF];A) := \Hom_{\ZG}(C_k^{\FF},A) = \Hom_{\ZG}(\ZGF\otimes I_{\FF}^k,A), \\ 
	&d^k_\FF\colon  C^k([G:\FF];A) \to C^{k+1}([G:\FF];A), \qquad d^k_\FF := (d_{k+1}^\FF)^*,
\end{align*}
for each $k \in \NN_0$. We further consider the cohomology groups of $G$ as defined with respect to the projective resolution of Dranishnikov and Rudyak, i.e. as the cohomology of the cochain complex
$$C^k(G;A)= \Hom_{\ZG}(\ZG\otimes I^k_G,A), \qquad d^k\colon C^k(G;A) \to C^{k+1}(G;A), \quad d^k= (d_{k+1})^*.$$
Next we define a chain map between the two resolutions used so far. We will use this map below to define an important map from Adamson to regular group cohomology. 

\begin{definition} \label{def:sigmaH}
	For $H \in \FF$ and $k \in \NN_0$ we consider the map
	$$\sigma_{H,k}\colon \ZG \otimes I_G^k \to \ZGF \otimes \IF^k, $$
	given by $\ZZ$-linearly extending
	$$\sigma_{H,k}(g_0 \otimes (g_1-1) \otimes \dots \otimes (g_k-1))=g_0H\otimes (g_1H-H)\otimes \dots \otimes (g_kH-H) $$
	for all $g_0,g_1,\dots,g_k \in G$.
\end{definition}

We observe that $\sigma_{H,*}$ is independent of the choice of $H \in \FF$ up to homotopy, as the following lemma shows. 

\begin{lemma}
	\label{LemmasigmaH}
	For each $H \in \FF$, the map $\sigma_{H,*}\colon \ZG\otimes I_G^* \to\ZGF \otimes I_\FF^*	$ is a chain map and for all $H_1,H_2 \in \FF$ the maps $\sigma_{H_1,*}$ and $\sigma_{H_2,*}$ are chain-homotopic.
\end{lemma}
\begin{proof}
	It is evident from the definition of the differentials and the augmentations that $\sigma_{H,*}$ is a chain map and that the following diagram commutes for each $H \in \FF$:
	 $$\begin{tikzcd}[cramped]
			\cdots \ar[r] & \ZG \otimes I_G^k \ar[d,"\sigma_{H,k}"] \ar[r,"d_k"] & \ZG \otimes I_G^{k-1} \ar[d,"\sigma_{H,k-1}"] \ar[r] & \cdots   \ar[r,"d_1"] & \ZG \ar[d,"\sigma_{H,0}"] \ar[r,"\varepsilon_G"] & \ZZ \ar[d,equal] \ar[r] & 0 \\
			\cdots \ar[r] & \ZGF \otimes I_{\FF}^k \ar[r,"d_k^\FF"] & \ZGF \otimes I_{\FF}^{k-1} \ar[r] & \cdots  \ar[r,"d_1^\FF"] & \ZGF \ar[r,"\varepsilon_{\FF}"] & \ZZ \ar[r] & 0.
	\end{tikzcd}$$
	Since the top row is a projective resolution of $\ZZ$ over $\ZG$ and the bottom row is an exact chain complex of $\ZG$-modules, it follows from the fundamental lemma of homological algebra that $\sigma_{H,*}$ is unique up to chain homotopy.
\end{proof}

The following cohomology classes will be important tools throughout the remainder of this article. 

\begin{definition}
	\label{DefCanonicalBerstein}
	\begin{enumerate}[a)]
		\item Consider the cochain $\beta_\FF \in C^1([G:\FF];I_\FF)$, given by $\ZZ$-linearly extending
		$$\beta_\FF(x \otimes y) = \varepsilon_{\FF}(x) \cdot y$$
		for all $x \in \ZGF$, $y \in I_\FF$. One checks without difficulties that $\beta_\FF$ is a cocycle. We define the \emph{Adamson canonical class of $\FF$} as its cohomology class
		$$\phi_\FF \in H^1([G:\FF];I_\FF), \qquad \phi_\FF:= [\beta_\FF]. $$
		\item Given $H \in \FF$, $k \in \NN_0$ and a $\ZG$-module $A$, we consider the map 
		$$\rho^k_\FF\colon H^k([G:\FF];A) \longrightarrow H^k(G;A), \qquad \rho^k_\FF:= [\sigma_{H,k}^*], $$
		i.e. the map induced in cohomology by $\sigma_{H,k}^*$.
		Note that by Lemma \ref{LemmasigmaH}, $\rho^*_\FF$ is independent of the choice of $H \in \FF$.
		\item The \emph{Berstein class of $G$ relative to $\FF$} is defined as
		$$\omega_\FF \in H^1(G;\IF), \qquad \omega_\FF:= \rho^1_\FF(\phi_{\FF}). $$
	\end{enumerate}
\end{definition}

In the case of the trivial family $\FF = \{1\}$, one easily identifies $\omega_{\{1\}}$ with the usual Berstein class $b \in H^1(G;I_G)$, originally defined in \cite{Berstein}. Given $H \in \FF$, we put $I_H := I_{G/H}$ and consider the Berstein class of $G$ relative to $H$, denoted by $\omega_H \in H^1(G;I_H)$, as defined in \cite{BCE22}. 

\begin{prop}
	Let $H \in \FF$ and let $j_H\colon I_H \hookrightarrow \IF$ denote the inclusion. Then 
	$$(j_H)_{*}(\omega_H)= \omega_\FF.$$
\end{prop}
\begin{proof}
	By its definition in \cite{BCE22}, $\omega_H=[\xi]$, where $\xi \in C^1(G;I_H)=\Hom_{\ZG}(\ZG\otimes I_G,I_H)$ is given by $\ZZ$-linearly extending 
	$$	\xi(g_0\otimes (g_1-1)) = g_1H-H $$
	for all $g_0,g_1 \in G$. In terms of the maps considered in this section, we compute that 
	\begin{align*}
		(j_H \circ \xi)(g_0 \otimes (g_1-1))&=\varepsilon_{\FF}(g_0H)\cdot (g_1H-H) \\
		&= \beta_\FF(g_0H\otimes (g_1H-H)) \\
		&= (\beta_\FF \circ \sigma_{H,1})(g_0 \otimes (g_1-1)) 	
	\end{align*}
	for all $g_0,g_1 \in G$, which yields that $j_H\circ \xi = \beta_\FF \circ \sigma_{H,1}$ and thus
	$$(j_H)_*(\omega_H) = [j_H \circ \xi] = [\beta_\FF \circ \sigma_{H,1}]= (\sigma_{H,1})^*([\beta_\FF])=\rho_\FF^1(\phi_\FF)= \omega_\FF.$$
\end{proof}

\subsection{Cup products and universality of Adamson canonical classes} In \cite{BCE22} an alternative description of cup products in group cohomology was given and further employed to define cup products in Adamson cohomology relative to a single subgroup. This construction generalizes to families of subgroups as we shall carry out next. \medskip

Let $A$ be a $\ZG$-module, let $p \in \NN_0$ and let $a \in C^p([G:\FF];A)= \Hom_{\ZG}(\ZGF \otimes I_F^p,A)$ be a cocycle, i.e. 
$$0 =a \circ d_{p+1}^\FF=a \circ f_p^\FF \circ g_{p+1}^\FF .$$
Since $g_{p+1}^\FF$ is surjective by exactness of \eqref{EqBasicSESk}, this implies $a \circ f_p^\FF=0$ and the same exactness yields the existence of $\widehat{a} \in \Hom_{\ZG}(\IF^p,A)$
making the following diagram commutative:
\begin{equation}
	\label{EqahatAdam}
	\begin{tikzcd}
		0 \ar[r] & I_\FF^{p+1} \ar[r, "f^\FF_p"] & \ZGF \otimes I_\FF^p \ar[d,"a"] \ar[r,"g^\FF_p"] & I^p_\FF \ar[dl,dotted,"\widehat{a}"] \ar[r]& 0.\\
		& & A
	\end{tikzcd}
\end{equation}

\begin{definition} \label{def:Adamson-cup-prod}
Let $A$ and $B$ be $\ZG$-modules, let $p,q \in \NN_0$, let $a \in C^p([G:\FF];A)$ and $b \in C^q([G:\FF];B)$ be cocycles and let $\widehat{a}\colon I_\FF^p \to A$ and $\widehat{b}\colon I_\FF^q \to B$ be given as in \eqref{EqahatAdam}. We define the relative \emph{cup product} $[a]\cup[b]\in H^{p+q}([G:\FF];A\otimes B)$ as the class represented by $$ [a] \cup [b] = \left[(\widehat{a} \otimes \widehat{b}) \circ g^\FF_{p+q} \right]. $$
\end{definition}

\begin{remark}\label{RemarkCup}
\begin{enumerate}[(1)]
	\item One checks without difficulties that $(\widehat{a} \otimes \widehat{b}) \circ g^\FF_{p+q}$ is indeed a cocycle in Definition \ref{def:Adamson-cup-prod}. Moreover, for any two cocycles $a$ and $b$ one easily checks that 
	$$ ((\widehat{a} \otimes \widehat{b}) \circ g^\FF_{p+q}) \circ \psi^{p+q}_\FF = (a \circ \psi^p_\FF) \smile (b \circ \psi^q_\FF) $$ where $\psi_\FF^*$ is given as in \eqref{chainmapsresolutions} and where
	 $$ (\alpha \smile \beta)(x_0, \cdots, x_{p+q}) = \alpha(x_0, \cdots, x_p) \otimes \beta(x_p, \cdots, x_{p+q}) $$ 
	for any $\alpha \in \Hom_\ZG(\EE^\FF_p,A)$ and $\beta \in \Hom_\ZG(\EE^\FF_q,A)$. Observe that $\smile$ is induced by the diagonal approximation 
	$$ \Delta_\FF \colon \EE^\FF_n \to \displaystyle{\bigoplus_{p+q = n} \EE^\FF_p \otimes \EE^\FF_q}, \qquad \Delta_\FF(x_0, \cdots, x_n) = \displaystyle{\sum_{p+q=n} (x_0, \cdots, x_p) \otimes (x_p, \cdots, x_n)}, $$ which is clearly a chain homomorphism of $\ZG$-modules. Hence, in an analogue way to the classic case, one sees that with respect to the differential of $\EE^\FF_*$ it satisfies 
	$$(\alpha \smile \beta) \circ b^\FF_{p+q} = (\alpha \circ b^\FF_p) \smile \beta + (-1)^p \alpha \smile (\beta \circ b^\FF_q).$$ 
	Since $\psi^*_\FF$ induces an isomorphism at cohomology level, this shows that $\left[(\widehat{a} \otimes \widehat{b}) \circ g^\FF_{p+q} \right]$ depends only on $[a]$ and $[b]$, and so the relative cup product is well-defined.
	\item It follows from \cite[Proposition 2.5]{BCE22} that in the case of $\FF=\{1\}$, the construction of Definition \ref{def:Adamson-cup-prod} indeed recovers the cup product in group cohomology as defined e.g. in \cite[Section V.3]{Brown}.
\end{enumerate}
\end{remark}

The cup product that we just defined is natural with respect to the previously established homomorphism between Adamson and classic group cohomology. 

\begin{lemma} \label{NaturalityRelCup}
Let $A$ and $B$ be $\ZG$-modules, $p,q \in \NN_0$, $u_1 \in H^p([G:\FF];A)$ and $u_2 \in H^q([G:\FF];B)$. Then 
$$\rho^{p+q}_\FF(u_1 \cup u_2)= \rho^p_\FF(u_1)\cup \rho^q_{\FF}(u_2).$$
\end{lemma}
\begin{proof}
Let $a \in C^p([G:\FF];A)$ and $b \in C^q([G:\FF];B)$ be cocycles representing $u_1$ and $u_2$, respectively. Then 
\begin{align*}
\rho^{p+q}_\FF(u_1 \cup u_2)&=\rho^{p+q}_\FF([a] \cup [b]) = \left[ \left((\widehat{a} \otimes \widehat{b}) \circ g^\FF_{p+q} \right) \circ (\sigma_{H,p+q}) \right], \\
\rho^p_{\FF}(u_1) \cup \rho^q_{\FF}(u_2) &= \rho^p_\FF([a]) \cup \rho^q_\FF([b]) = \left[ \left( (\widehat{a \circ \sigma_{H,p}} \otimes \widehat{b \circ \sigma_{H,q}} \right) \circ (g^{\{1\}}_{p+q}) \right],
\end{align*} where we used Remark \ref{RemarkCup}.(2) in the second equality.  Let us compare the representing cocycles. For all $x = x_0 \otimes (x_1-1) \otimes \cdots \otimes (x_{p+q}-1)\in \ZG\otimes I_G^{p+q}$, where $x_0,\dots,x_{p+q} \in G$, we obtain 
{\small\begin{align*}
		&\left(\left((\widehat{a} \otimes \widehat{b}) \circ g^\FF_{p+q} \right) \circ \sigma_{H,p+q}\right)(x) = \left(\left( (\widehat{a} \otimes \widehat{b}) \circ g^\FF_{p+q} \right)\circ  \sigma_{H,p+q}\right) \left(x_0 \otimes (x_1-1) \otimes \cdots \otimes (x_{p+q}-1)\right)\\
		&= \widehat{a} \left((x_1H-H) \otimes \cdots \otimes (x_pH-H)\right) \otimes \widehat{b} \left((x_{p+1}H-H) \otimes \cdots \otimes (x_{p+q}H-H)\right),
		\end{align*}}
and
{\small\begin{align*}
		&\left(\left( \widehat{a \circ \sigma_{H,p}} \otimes \widehat{b \circ \sigma_{H,q}} \right) \circ g^{\{1\}}_{p+q} \right)(x) = \left(\left( \widehat{a \circ \sigma_{H,p}} \otimes \widehat{b \circ \sigma_{H,q}} \right) \circ g^{\{1\}}_{p+q}\right) \left(x_0 \otimes (x_1-1) \otimes \cdots \otimes (x_{p+q}-1)\right) \\
		&= \left( \widehat{a \circ \sigma_{H,p}} \right) \left( (x_1-1) \otimes \cdots \otimes (x_p-1)\right)  \otimes \left(\widehat{b \circ \sigma_{H,q}}\right) \left((x_{p+1}-1)\otimes \cdots \otimes (x_{p+q}-1)\right) \\
		&= \widehat{a}\left((x_1H-H) \otimes \cdots \otimes (x_pH-H)\right)  \otimes \widehat{b}\left((x_{p+1}H-H)\otimes \cdots \otimes (x_{p+q}H-H)\right).
\end{align*}} The claim then immediately follows from comparing both expressions. \end{proof}

The main benefit of our definition of the cup product is that it allows for a simple description of the powers of both the Adamson canonical class and the relative Berstein class with respect to $\FF$.

\begin{prop}
\label{Propphikrepr}
Let $k \in \NN$. Then 
$$\phi^k_\FF:=\phi_\FF \cup \phi_\FF \cup \dots \cup \phi_\FF \in H^k([G:\FF];\IF^k)$$
is represented by $g_k^\FF\colon \ZGF \otimes \IF^k \to \IF^k$, seen as an Adamson cocycle. 

\end{prop}
\begin{proof}
Let $\beta_\FF$ be the cocycle from  Definition \ref{DefCanonicalBerstein}.a). We will prove the claim by induction. The base case is trivial as, evidently, $\beta_F=g_1^\FF$. 
Proceeding inductively, we assume that $\phi^k_\FF$ is represented by $g_k^\FF$ and compute that
$$\phi^{k+1}_\FF= \phi_\FF^k\cup \phi_\FF= \left[\left(\widehat{g}_k^\FF \otimes \widehat{\beta}_\FF\right) \circ g_{k+1}^\FF\right], $$
Note that by definition, $g_k^\FF= \widehat{g}_k^\FF \circ g_{k}^\FF$ and $\beta_\FF= \widehat{\beta}_\FF \circ g_{1}^\FF$, which immediately implies that $\widehat{g}^\FF_k=\id_{\IF^k}$ and $\widehat{\beta}_\FF=\id_{\IF}$. Thus, we derive that
$\phi^{k+1}_\FF=\left[\id_{\IF^{k+1}} \circ g_{k+1}^\FF \right]=\left[g_{k+1}^\FF\right],$
which completes the induction step.
\end{proof}

\begin{corollary} \label{Corphikrepr}
Let $k \in \NN$. Then 
$$\omega_\FF^k := \omega_\FF \cup \omega_\FF \cup \dots \cup \omega_\FF \in H^k(G;\IF^k)$$
is represented by $g_k^\FF \circ \sigma_{H,k}\in \Hom_{\ZG}(\ZG \otimes I_G^k, \IF^k)$, where $H \in \FF$ is chosen arbitrarily.
\end{corollary}
\begin{proof}
By naturality of the cup product from Lemma \ref{NaturalityRelCup}, it holds that
$$\omega_{\FF}^k= (\rho_\FF(\phi_\FF))^k = \rho_{\FF}(\phi_\FF^k).$$
By Proposition \ref{Propphikrepr}, this class is represented by $(\sigma_{H,k})^*(g_k^\FF)= g_k^\FF \circ \sigma_{H,k}$. 
\end{proof} 

In particular, for the trivial family $\FF = \{1\}$ the above corollary recovers the description of the $k$-th power of the Berstein class $b^k \in H^k(G;I_G^k)$ that is obtained in \cite{DranishRudyak}. \medskip

Similar to the case of a single subgroup, $\phi_\FF$ is \emph{universal} in the sense that every Adamson cohomology class of degree $k \geq 1$ is given as a pushforward of a power of the Adamson canonical class, as the following result shows. 

\begin{prop} \label{PropUniv}
Let $k \in \NN$, let $A$ be a $\ZG$-module and let $u \in H^k([G:\FF];A)$. Then there exists a $\ZG$-homomorphism $\psi\colon I_{\FF}^k \to A$ with 
$$u= \psi_*(\phi_{\FF}^k).$$
\end{prop}
\begin{proof}
Let $c\in \Hom_{\ZG}(\ZGF \otimes \IF^k,A)$ be a cocycle representing  $u$ and let $\widehat{c}\in \Hom_{\ZG}(\IF^k,A)$ be given as in \eqref{EqahatAdam}. Then 
$$u=[c]=[\widehat{c}\circ g_k^\FF] = \left(\widehat{c}\right)_*([g_k^\FF])= \left(\widehat{c}\right)_*(\phi_\FF^k),$$
by Proposition \ref{Propphikrepr}, so the claim follows by taking $\psi:= \widehat{c}$.
\end{proof}

\subsection{Real Adamson canonical and Berstein classes} \label{SubsectReal}

To connect the cohomology classes from this section to bounded cohomology, we will require ``realifications'' of the Adamson canonical and Berstein classes, which are relevant for the study of $\RG$-modules. We will devote this subsection to a proper treatment of these classes. \medskip  

We recall that any $\RR[G]$-module can be seen as a $\ZG$-module via restriction of scalars. Conversely, given any $\ZG$-module $A$, we can view $\RR \otimes A$ as an $\RG$-module via $\RR$-linearly extending 
$$\sum_{g \in G} r_g g \cdot r_0 \otimes a = \sum_{g \in G} r_gr_0 \otimes ga \qquad \forall r_0 \in \RR, \ a \in A, \ \sum_{g \in G} r_gg \in \RG.$$ 
Consider a projective resolution $(P_*,\partial_*)$ of $\ZZ$ over $\ZG$. Since $\RR$ is a flat $\ZZ$-module, it follows that $(\RR\otimes P_*, \id_{\RR}\otimes\partial_*)$ is a projective resolution of $\RR$ over $\RG$. The maps
\begin{equation}
\label{Eqpsis}
\psi_{k,A}\colon \Hom_{\ZG}(P_k,A) \to \Hom_{\RG}(\RR\otimes P_k,A), \quad \overline{\psi}_{k,A}(f)(r\otimes x) = r\cdot f(x),
\end{equation}
whose inverses are given by 
$$\overline{\psi}_{k,A}^{-1}\colon \Hom_{\RG}(\RR\otimes P_k,A) \to \Hom_{\ZG}(P_k,A), \qquad \overline{\psi}_{k,A}^{-1}(f)(x) = f(1\otimes x),$$
then form an isomorphism of cochain complexes. 
The following proposition gives an important property of such ``realified'' resolutions.

\begin{prop}
\label{PropRealFprojres}
	Let $(P_*,p_*)$ be a relatively $\FF$-projective $\FF$-strong resolution of $\ZZ$ over $\ZG$. Then $(\RR\otimes P_*,\id_{\RR} \otimes p_*)$ is a relatively $\FF$-projective $\FF$-strong resolution of $\RR$ over $\RG$.
\end{prop}
\begin{proof}
We first want to show that $\RR \otimes P_k$ is relatively $\FF$-projective over $\RG$ for each $k \in \NN_0$. Let $f \in \Hom_{\RG}(A,B)$ be $\FF$-strongly surjective and let $\varphi \in \Hom_{\RG}(\RR \otimes P_k,B)$. Regarding $A$ and $B$ as $\ZG$-modules, it is evident that $f$ is $\FF$-strongly surjective over $\ZG$ as well. Consider $\varphi' \in \Hom_{\ZG}(P_k,B)$, with $\varphi'(x) = \varphi(1 \otimes x)$. Since $P_k$ is relatively $\FF$-projective over $\ZG$, there exists $\widetilde{\varphi}' \in \Hom_{\ZG}(P_k,A)$ such that $f \circ \widetilde{\varphi}' = \varphi'$. Put  $\widetilde{\varphi} \in \Hom_{\RG}(\RR \otimes P_k,A)$, defined by $\widetilde{\varphi}(r \otimes x) = r \cdot \widetilde{\varphi}'(x)$. As $$ (f \circ \widetilde{\varphi}) (r \otimes x) = f(r \widetilde{\varphi'}(x)) = r \varphi'(x) = \varphi(r \otimes x) \qquad \forall r \in \RR, \ x \in P_k,$$ it follows that $f \circ \widetilde{\varphi} = \varphi$, thus $\RR \otimes P_k$ is relatively $\FF$-projective over $\RG$. Since $\RR$ is $\ZZ$-flat, it further holds that $(\RR \otimes P_*,\id_{\RR}\otimes p_*)$ is an exact chain complex. Thus, it is a relatively $\FF$-projective resolution of $\RR$ over $\RG$ and it only remains to show that it is $\FF$-strong. 

 Since $(P_*,p_*)$ is $\FF$-strong over $\ZG$, for each $H \in \FF$ there is a contraction $\tau_{H,k} \colon P_k \to P_{k+1}$, $k \geq -1$, of $P_*$ by $\ZH$-homomorphisms, where $P_{-1}:=\ZZ$. Then for each $k \in \NN_0$ we observe that $\id_{\RR} \otimes \tau_{H,k} \in \Hom_{\RR[H]}(\RR\otimes P_k,\RR\otimes P_{k+1})$ and $$(\id_{\RR} \otimes p_{k+1})\circ(\id_{\RR} \otimes \tau_{H,k}) + (\id_{\RR} \otimes \tau_{H,k-1})\circ (\id_{\RR} \otimes p_k) = \id_{\RR \otimes P_k}.$$ It further holds that 
 $$(\id_{\RR} \otimes p_0)\circ (\id_{\RR} \otimes \tau_{H,-1})= \id_{\RR}\otimes (p_0 \circ \tau_{H,-1}) = \id_{\RR}.$$
 Thus, $\id_{\RR}\otimes \tau_{H,*}\colon \RR\otimes P_* \to \RR \otimes P_{*+1}$ is a contraction of $\RR\otimes P_*$ over $\RR[H]$. Since $H \in \FF$ was chosen arbitrarily, this shows that $(\RR\otimes P_*,\id_{\RR} \otimes p_*)$ is $\FF$-strong over $\RG$.\end{proof}
 
We derive from the previous considerations and Proposition \ref{PropRealFprojres} that for any $\RG$-module $A$ and any relatively $\FF$-projective $\FF$-strong resolution $(P_*,p_*)$ of $\ZZ$ over $\ZG$, we obtain that
\begin{equation}
\label{EqRealAdamson}
H^k([G:\FF];A) \cong H^k(\Hom_{\RG}(\RR\otimes P_*,A), (\id_\RR \otimes p_*)^*) \qquad \forall k \in \NN_0.
\end{equation}
 We note that for any $G$-set $S$ there is an obvious isomorphism of $\RG$-modules 
$$\Psi_S\colon \RR \otimes \ZZ[S] \stackrel{\cong}{\longrightarrow} \RS, \quad \Psi_S(r \otimes x) = r \cdot x.$$
 For any such $S$ we let $$\varepsilon_{S,\RR}\colon\RR[S] \to \RR, \quad \varepsilon_{S,\RR}\Big(\sum_{s \in S} r_ss\Big)= \sum_{s \in S} r_s,$$ denote the augmentation and observe that $ \varepsilon_{S,\RR} \circ \Psi_S = \mu \circ (\id_\RR \otimes \varepsilon_S)$, where $\mu \colon \RR \otimes \ZZ \xrightarrow{\cong} \RR$. For each $k \in \NN$ we put 
 $$J_S^k := \RR\otimes I^k_S$$
 with the abovementioned $\RG$-module structure. We further put $J_S^0 := \RR$ and $J_S:= J_S^1$. In the special case of $S=G/\FF$, we further put  $J_\FF^k:=J_{G/\FF}^k= \RR\otimes I^k_\FF$ for each $k \in \NN$ and let $J_\FF:= J_\FF^1$. \medskip

One checks without difficulties that for each $k \in \NN_0$ the map  
\begin{equation}
\label{Eqalphak}
\begin{aligned}
&F_k\colon \RR \otimes (\ZGF\otimes I^k_\FF)\stackrel\cong{\longrightarrow} \RGF\otimes_{\RR} J^k_\FF,  \quad F_k(r \otimes (x \otimes y))=(r \otimes x) \otimes_\RR (1 \otimes y) ,
\end{aligned}
\end{equation}
is an isomorphism of $\RG$-modules whose inverse is given by
 $$F_k^{-1}((r_1 \otimes a) \otimes_\RR (r_2 \otimes b))= r_1 r_2 \otimes (a \otimes b) \qquad \forall r_1 \otimes a \in \RGF, \ r_2\otimes b \in J^k_\FF.$$ 
 In particular $\RR \otimes (\ZG \otimes I^k_G) \cong \RG \otimes_\RR J^k_G$. Since $\RR$ is a flat $\ZZ$-module, tensoring the sequence from \eqref{EqBasicSESk} with $\RR$ and using the isomorphisms from above yields a short exact sequence 
  \begin{equation}
  \label{EqSESreal} 
  0 \to J^{k+1}_\FF \xrightarrow{f^\FF_{k,\RR}} \RGF \otimes_\RR J^k_\FF \xrightarrow{g^\FF_{k,\RR}} J^k_\FF \to 0  
  \end{equation}
 of $\RG$-modules for each $k \in \NN_0$. In analogy with the sequences \eqref{EqBasicSESk}, we can splice these short exact sequences to obtain a resolution of $\RR$ over $\RG$. Combining the identifications made above with Lemma \ref{Adamson:projective} and Proposition \ref{PropRealFprojres} then yields the following statement whose detailed proof we omit.
  
  \begin{lemma}
  \label{LemmaRealProjDRres}
  Let $C^{\FF,\RR}_k := \RGF \otimes_\RR J_{\FF}^k$ for each $k \in \NN_0$ and let $$d_k^{\FF,\RR}\colon C^{\FF,\RR}_k \to C^{\FF,\RR}_{k-1},\qquad  d_k^{\FF,\RR}:=f_{k-1}^{\FF,\RR} \circ  g_k^{\FF,\RR},$$ for each $k \in \NN$. Then 
  $$\dots \longrightarrow C^{\FF,\RR}_k \stackrel{d_k^{\FF,\RR}}\longrightarrow C^{\FF,\RR}_{k-1} \longrightarrow \dots \stackrel{d_1^{\FF,\RR}}\longrightarrow C^{\FF,\RR}_0 \stackrel{\varepsilon_{\FF,\RR}}{\longrightarrow} \RR \to 0$$
 is a relatively $\FF$-projective $\FF$-strong resolution of $\RR$ over $\RG$.
  \end{lemma}

Note that in the case of $\FF=\{1\}$, we obtain a projective resolution of $\RR$ over $\RG$. Given an $\RG$-module $A$, we consider the cochain complexes 
\begin{align*}
	C_\RR^k(G;A)&:= \Hom_{\RG}(\RG\otimes_\RR J_G^{k},A),	 \\
	C_\RR^k([G:\FF];A) &:= \Hom_{\RG}(\RR[G/\FF]\otimes_\RR J_{\FF}^k,A), 
\end{align*}
for all $k \in \NN_0$, equipped with the differentials
\begin{align*}
	&d^k_\RR\colon C^k_\RR(G;A) \to C^{k+1}_\RR(G;A),  \quad d^k_\RR := (d^\RR_{k+1})^*, \\ 
	&d^k_{\FF,\RR} \colon	C^k_\RR([G:\FF];A) \to C^{k+1}_\RR([G:\FF];A), \qquad d^k_{\FF,\RR} := (d^{\FF,\RR}_{k+1})^*.
\end{align*} 
for each $k \in \NN_0$. By \eqref{EqRealAdamson} and the identifications from above, we obtain 
\begin{align*}
	H^k(G;A) &= H^k(C^*_\RR(G;A),d^*_\RR), \\ H^k([G:\FF];A) &= H^k(C^*_\RR([G:\FF];A),d_{\FF,\RR}^*) .
\end{align*}
 We want to use these cochain complexes to define analogues of the cohomology classes from Definition \ref{DefCanonicalBerstein}.

\begin{definition} \label{RealAdamsonBersteinclass}
	\begin{enumerate}[a)]
		\item Consider the cocycle $\beta_{\FF,\RR} \in C^1_\RR([G:\FF];J_\FF)$, given by $\RR$-linearly extending
		$$\beta_{\FF,\RR}(x \otimes y) = \varepsilon_{\FF,\RR}(x) \cdot y$$
		for all $x \in \RGF$, $y \in J_\FF$. The \emph{real Adamson canonical class of $\FF$} is its cohomology class
		$$\phi_{\FF,\RR} \in H^1([G:\FF];J_\FF), \qquad \phi_{\FF,\RR}:= [\beta_{\FF,\RR}]. $$
		\item The \emph{real Berstein class of $G$ relative to $\FF$} is defined as
		$$\omega_{\FF,\RR} \in H^1(G;J_\FF), \qquad \omega_{\FF,\RR}:= \rho^1_\FF(\phi_{\FF,\RR}). $$
	\end{enumerate}
\end{definition} 

In the following we will use the simple fact that the relative cup product of two cohomology classes with coefficients in $\RG$-modules $A$ and $B$, respectively, formed by the resolution used in Lemma \ref{LemmaRealProjDRres}, yields a cohomology class with coefficients in $A\otimes_\RR B$.

Arguing as in the proof of Proposition \ref{Propphikrepr}, one shows that $\phi^k_{\FF,\RR} = \phi_{\FF,\RR} \cup \dots \cup \phi_{\FF,\RR} $ is represented by $g^{\FF,\RR}_k \in C^k([G:\FF];J_\FF^k)$. To obtain an explicit description of $\omega_{\FF,\RR}^k$ for $k \in \NN$ we can follow a similar procedure as for the powers of $\omega_{\FF}$. For all $H \in \FF$ and $k \in \NN_0$ we let 
\begin{equation*}
	\sigma_{H,k}^{\RR}\colon \RG \otimes_\RR J^k_G \to \RGF \otimes_\RR  J_\FF^k
\end{equation*}
be given in analogy with Definition \ref{def:sigmaH} by $\RR$-linearly extending
\begin{equation}
\label{EqsigmaR}
\sigma^{\RR}_{H,k}(g_0\otimes (g_1-1)\otimes \dots \otimes (g_k-1)) = g_0H\otimes (g_1H-H)\otimes \dots \otimes (g_kH-H). 
\end{equation}
In complete analogy with Lemma \ref{LemmasigmaH}, one shows that each $\sigma^\RR_{H,*}$ is a chain map and that $\sigma_{H_1,*}^\RR$ and $\sigma_{H_2,*}^\RR$ are chain homotopic for any two $H_1,H_2 \in \FF$. A straightforward computation then shows that $\omega^k_{\FF,\RR} = \omega_{\FF,\RR} \cup \dots \cup \omega_{\FF,\RR}$ is represented by $g_k^{\FF,\RR} \circ \sigma_{H,k}^\RR$ for each $k \in \NN$.

The relation between the real Adamson canonical and Berstein classes and the original ones can be explored via the change of coefficients homomorphism 
$$\alpha_{\FF}^k\colon I^k_\FF \to J^k_\FF \qquad \alpha_\FF^k(x)=1 \otimes x.$$ 
The precise relation between the classes is as follows.

\begin{lemma}
\label{LemmaRealBerstein}
For all $k \in \NN$ it holds that   
$$ \phi^k_{\FF,\RR} = (\alpha_\FF^k)_*(\phi^k_\FF) \in H^k([G:\FF];J_\FF^k) \qquad \mbox{and} \qquad \omega^k_{\FF,\RR}= (\alpha_{\FF}^k)_*(\omega^k_\FF) \in H^k(G;J_\FF^k).  $$
\end{lemma}
\begin{proof}
Fix $k \in \NN$ and consider the isomorphisms of cochain complexes
\begin{align*}
 &\Psi^k_G \colon \Hom_{\RG}(\RG\otimes_\RR J_G^*,J_\FF^k) \to \Hom_{\ZG}(\ZG\otimes I_G^*,J_\FF^k), \\
 &\Psi^k_\FF \colon \Hom_{\RG}(\RGF\otimes_\RR J_\FF^*,J_\FF^k) \to \Hom_{\ZG}(\ZGF\otimes I_\FF^*,J_\FF^k), 
 \end{align*} 
 both given in terms of the maps from \eqref{Eqpsis} and \eqref{Eqalphak} by $\psi_{*,J_\FF^k}^{-1}\circ F_*^*$ applied to the respective resolutions. A straightforward calculation shows that 
$\Psi^k_\FF(g_k^{\FF,\RR})= \alpha^k_\FF\circ g_k^\FF, $
which yields on cohomology level that, up to the isomorphism induced by $\Psi_\FF^*$, 
$$\phi_{\FF,\RR}^k = [g_k^{\FF,\RR}] = [\alpha^k_{\FF} \circ g_k^\FF] = (\alpha^k_\FF)_*([g_k^\FF]) = (\alpha^k_\FF)_*(\phi^k_\FF)$$
by Proposition \ref{Propphikrepr}. 
Another straightforward calculation shows that 
$\Psi^k_G(g_{k}^{\FF,\RR} \circ \sigma_{H,k}^{\RR})= \alpha_\FF^k \circ g_k^\FF \circ \sigma_{H,k} $
and the second part of the claim follows analogously using Corollary \ref{Corphikrepr}. \end{proof}

In particular, in the case of $\FF=\{1\}$ we again identify $I_G^k \cong I^k_{\{1\}}$ and $J_G^k \cong J_{\{1\}}^k$ for each $k \in \NN$ and consider the change of coefficient homomorphism
$$\alpha_G^k:= \alpha_{\{1\}}^k\colon I_G^k \to J_G^k.$$
We further consider the class  $$ b_{\RR}^k := (\alpha_G^k)_*(b^k) \in H^k(G;J_G^k).$$
One checks without difficulties that,  indeed, $b^k_\RR = b_\RR \cup \dots \cup b_\RR$ as the $k$-fold cup product, so $b^k_\RR$ is the $k$-th power of the real Berstein class of $G$.

Finally, observe that the proof of Proposition \ref{PropUniv} can be straightforwardly translated into the realified case. Consequently, given any $\RG$-module $A$ and any $u \in H^k([G:\FF];A)$ there exists an $\RG$-homomorphism $\psi \colon J^k_\FF \to A$ with \begin{equation}\label{PropUnivReal}
u = \psi_*(\phi^k_{\FF,\RR}). 
\end{equation}

\section{Bounded Adamson cohomology with respect to subgroup families}
We proceed by introducing and studying a bounded version of Adamson cohomology and establish analogues of the cohomology classes defined in the previous section in the setting of bounded cohomology. \medskip

Throughout this section we let $\FF$ be a family of subgroups of $G$. Since we will only work with $\RG$-modules in this section, we will use the constructions from Subsection \ref{SubsectReal}, but drop the subscript and superscript $\RR$ from the notation of the cochain complexes. More precisely, we will use the notation
$$C^k([G:\FF];A)= \Hom_{\RG}(\RGF\otimes_\RR J_\FF^k,A), \quad C^k(G;A)= \Hom_{\RG}(\RG\otimes_\RR J_G^k,A),$$
and write $d_k^\FF$, $f_k^\FF$, $g_k^\FF$ and $\sigma_{H,k}$ instead of $d_k^{\FF,\RR}$, $f^{\FF,\RR}_k$, $g_k^{\FF,\RR}$ and $\sigma_{H,k}^{\RR}$, respectively. 

\subsection{The definition of bounded cohomology and bounded Adamson cohomology}
We recall that a \emph{normed $G$-module} is an $\RG$-module equipped with a $G$-invariant norm and that a \emph{Banach $G$-module} is a normed $G$-module whose underlying normed vector space is a Banach space. We will usually suppress the norm of a normed $G$-module $V$ from the notation and denote it by $\|\cdot \|_V$ if nothing else is specified. \medskip 

Before proceeding we fix some notation:
\begin{itemize}
	\item We shall always view $\RR$ as a normed $G$-module with respect to the trivial $G$-action and the absolute value.
	\item Given two normed $G$-modules $V$ and $W$, we let $B(V,W)$ denote the space of bounded linear maps $V \to W$ and view it as a normed $G$-module with respect to the operator norm $\|\cdot\|_\op$ and the diagonal $G$-action given by
	$$ (g\cdot f)(v) = g f(g^{-1}v) \qquad \forall v \in V, \ g \in G, \ f \in B(V,W).$$
	\item For each normed $G$-module $V$ we let $V':= B(V,\RR)$ denote its topological dual space. Note that the diagonal action of $G$ on $V'$ is then given by 
	$$(g\cdot f)(v) =f(g^{-1}v) \quad \forall g \in G, \ f \in V', \ v \in V. $$
	\item We let $B_G(V,W)$ denote the space of all bounded linear $G$-maps from $V$ to $W$, i.e.
	$$B_G(V,W) = B(V,W) \cap \Hom_{\RG}(V,W).$$
	\item Given normed vector spaces $X$, $Y$ and $Z$ and a bounded linear map $f\colon X \to Y$, we further denote 
	$$f'\colon B(Y,Z) \to B(X,Z), \quad f'(g) := g \circ f.$$
\end{itemize}

We briefly recall the definition of bounded cohomology of discrete groups by means of homological algebra as given in \cite[Chapter 4]{FrigerioBook} whose terminology we will follow, see also \cite[Section 2.7]{LoehLectureNotes}.

\begin{definition} \label{relinjstrong}
	\begin{enumerate}[a)]
		\item Let $M$ and $N$ be normed $G$-modules. A map $i \in B_G(M,N)$ is called \emph{strongly injective} if there exists $r\in B(N,M)$ with $\|r\|_\op \leq 1$ and $r \circ i = \id_M$. 
		\item A normed $G$-module $V$ is called \emph{relatively injective} if for all normed $G$-modules $M$ and $N$, all strongly injective $i \in B_G(M,N)$ and all $f \in B_G(M,V)$, there exists $F \in B_G(N,V)$ with $f = F\circ i$ and $\|F\|_\op \leq \|f\|_\op$.
		\item A \emph{relatively injective strong resolution} of a normed $G$-module $V$ is a cochain complex $(J^k,j^k)_{k \in \NN_0}$ of relatively injective Banach $G$-modules together with a map $\alpha \in B_G(V,J^0)$, for which the augmented cochain complex
		$$ 0 \longrightarrow V \stackrel{\alpha}{\longrightarrow} J^0 \stackrel{j^0}\longrightarrow J^1 \longrightarrow\dots \longrightarrow J^k \stackrel{j^{k}}\longrightarrow J^{k+1} \longrightarrow \dots$$
		is exact and admits a contraction $\tau_*\colon J^* \to J^{*-1}$ with $\|\tau_k\|_\op \leq 1$ for each $k \in \NN_0$, where $J^{-1}:=V$. 
		\item Let $V$ be a Banach $G$-module. The \emph{bounded cohomology of $G$ with coefficients in $V$} is given by 
		$$H^k_b(G;V):= H^k((J^*)^G,j^*)\qquad \forall k \in \NN_0,$$
		where $(J^*,j^*)$ is a relatively injective strong resolution of $V$ and where $(\cdot )^G$ denotes the invariants functor.
	\end{enumerate}
\end{definition}
By \cite[Corollary 4.15]{FrigerioBook}, the bounded cohomology of a group with coefficients in a Banach $G$-module is independent of the choice of relatively injective strong resolution up to isomorphism. 

In order to define a notion of bounded Adamson cohomology, we need the bounded equivalent of relative injectivity and strongness with respect to a family of subgroups $\FF$. The following notions were first introduced by K. Li in \cite{LiBounded}.

\begin{definition}
	\label{DefBoundedCoh}
	\begin{enumerate}[a)]
		\item Let $M$ and $N$ be normed $G$-modules. A map $i \in B_G(M,N)$ is called \emph{strongly $\FF$-injective} if for each $H \in \FF$ there exists $r_H \in B_H(N, M)$ with $\|r_H\|_\op \leq 1$ and $r_H \circ i = \id_M$. 
		\item A normed $G$-module $V$ is \emph{relatively $\FF$-injective} if for all normed $G$-modules $M$ and $N$, all strongly $\FF$-injective $i \in B_G(M,N)$ and all $f \in B_G(M,V)$, there exists $F \in B_G(N,V)$ with $f = F\circ i$ and $\|F\|_\op \leq \|f\|_\op$.
		\item Let $V$ be a Banach $G$-module. A \emph{relatively $\FF$-injective resolution} of a Banach $G$-module $V$ is a cochain complex $(J^k,j^k)_{k \in\NN_0}$ of relatively $\FF$-injective Banach $G$-modules together with a map $\alpha \in B_G(V,J^0)$, such that the augmented cochain complex
		$$ 0 \longrightarrow V \stackrel{\alpha}{\longrightarrow} J^0 \stackrel{j^0}\longrightarrow J^1 \longrightarrow\dots \longrightarrow J^k \stackrel{j^{k}}\longrightarrow J^{k+1} \longrightarrow \dots$$
		is exact.
			\item An exact cochain complex $(C^k,\partial^k)_{k \in \ZZ}$ of normed $\RG$-modules is called \emph{bounded $\FF$-strong} if  for each $H \in \FF$ there exists a contraction $\tau_{H,*}\colon C^* \to C^{*-1}$ with $\tau_{H,k} \in B_H(C^k,C^{k-1})$ and $\|\tau_{H,k}\|_\op \leq 1$ for each $k \in \ZZ$.
		\item Let $V$ be a Banach $G$-module. A \emph{bounded $\FF$-injective resolution of $V$} is a relatively $\FF$-injective resolution of $V$ which is bounded $\FF$-strong. 
	\end{enumerate}
\end{definition}

To define bounded Adamson cohomology in terms of resolutions, we need the following standard lemma, see also \cite[Proposition 3.9]{LiBounded}. It is proven by standard methods from homological algebra which can be transferred to the bounded setting without any difficulties.

\begin{lemma} 
	Let $V$ be a Banach $G$-module and let $(J^*_1,j^*_1)$ and $(J^*_2,j^*_2)$ be bounded  $\FF$-injective resolutions of $V$. Then
	$$H^k((J^*_1)^G,j^*_1) \cong H^k((J^*_2)^G,j^*_2) \quad \forall k \in \NN_0.$$
\end{lemma}

This lemma leads us to the following definition. 

\begin{definition}
	Let $V$ be a Banach $G$-module. We define the \emph{bounded Adamson cohomology of $(G,\FF)$ with coefficients in $V$} by 
	$$H^k_b([G:\FF];V) := H^k((J^*)^G,j^*) \quad \forall k \in \NN_0, $$
	where $(J^*,j^*)$ is a bounded $\FF$-injective resolution of $V$.
\end{definition}

\begin{remark}
	\label{RemarkBounded}
	If $\FF=\{1\}$, then by definition a bounded $\FF$-injective resolution of a Banach $G$-module $V$ is nothing but a relatively injective strong resolution as recalled in Definition \ref{relinjstrong}.c).  Thus, by \cite[Corollary 4.15]{FrigerioBook} we obtain that
	$$H_b^k([G:\{1\}];V) \cong H^k_b(G;V) \qquad \forall k \in \NN_0.$$
\end{remark}

\begin{remark}
In \cite{LiBounded}, K. Li introduced and studied a notion of bounded Bredon cohomology with respect to a family of subgroups $\FF$, as a variation of the cohomology relative to the $G$-set $G/\FF$ of B. Nucinkis, see \cite{Nuc99}. It is natural to expect, though we shall not discuss this aspect any further, that  Li's construction is related to our notion of bounded Adamson cohomology in the same way as their classical analogues.
\end{remark}

\subsection{Construction of a bounded $\FF$-injective resolution} 
To use cochain complexes similar to those used in Section 2, we are going to equip the $G$-modules occurring in these cochain complexes with $G$-invariant norms for which we outline some general notions. \medskip

Given a $G$-set $S$, we consider $\RR[S]$ as equipped with the $\ell^1$-norm, given by
$$\|x\|_1 := \sum_{s \in S} |r_s| \qquad \forall x=\sum_{s\in S}r_s \cdot s \in \RR[S].$$
Evidently, $\| g\cdot x\|_1 = \|x\|_1$ for all $g \in G$ and $x \in \RR[S]$, hence $(\RR[S],\|\cdot \|_1)$ is a normed $G$-module. Likewise, we consider $J_S\subset \RR[S]$ with the $\ell^1$-norm as a normed $G$-submodule of $\RR[S]$.  

We further note that its augmentation $\varepsilon_S\colon \RR[S] \to \RR$ is bounded $G$-linear with $\|\varepsilon_S\|_\op =1$. Again, we put  $\varepsilon_\FF:= \varepsilon_{G/\FF} \in B_G(\RGF,\RR)$.

Given $n \in \NN$ and normed vector spaces $V_1,\dots,V_n$, we shall always consider their tensor product $V_1\otimes_{\RR} V_2\otimes_{\RR} \dots \otimes_{\RR} V_n$ as equipped with the \emph{projective norm}, given by 
\begin{align*}
	\|u\|_{\pi} := \inf \Big\{\sum_{i=1}^k \|v_{i1}\|_{V_1} \cdot \|v_{i2}\|_{V_2}\dots \|v_{in}\|_{V_n} \ \Big| \ &k \in \NN, v_{ij} \in V_j  \ \forall i \in \{1,2,\dots,k\}, \ j \in \{1,2,\dots,n\}, \\ 
	&\text{s.t. } \ u = \sum_{i=1}^k v_{i1}\otimes v_{i2} \otimes \dots \otimes v_{in} \Big\}
\end{align*}	
for all $u \in V_1\otimes_{\RR} \dots \otimes_{\RR} V_n$. The projective norm is discussed in this general form in \cite[Appendix B.2]{RundeAmenab} and in detail in the case of $n=2$ in \cite[Chapter 2]{Ryan} and  \cite[Section I.3]{DefantFloret}. It is a norm on $V_1\otimes_{\RR} \dots \otimes_{\RR} V_n$ and it satisfies 
\begin{equation*}
	\|v_1 \otimes v_2 \otimes \dots \otimes v_n\|_{\pi} = \|v_1\|_{V_1} \cdot \|v_2\|_{V_2} \cdot \ldots \cdot \|v_n\|_{V_n} \qquad \forall v_j \in V_j, \ j \in \{1,2,\dots,n\}.			
\end{equation*}
Moreover, it is easy to see that a tensor product of normed $G$-modules is again a normed $G$-module when equipped with the projective norm. \medskip 

\emph{Throughout the rest of the article, we consider for any $k \in \NN$ the $\RG$-modules $J_G^k$, $J^k_\FF$, $\RG \otimes_\RR J^k_G$ and $\RGF \otimes_\RR J^k_\FF$, as equipped with the projective norm with respect to the $\ell^1$-norm on each of the individual factors.} \medskip

Let $d_k\colon \RG \otimes_\RR J^k_G \to \RG \otimes_\RR J^{k-1}_G$ and $d^{\FF}_k\colon \RGF \otimes_\RR J^k_\FF \to \RGF \otimes_\RR J^{k-1}_\FF$  
be given as in Subsection \ref{SubsectReal} for every $k \in \NN$. We compute for all $x \in \RGF$, $y \in J_\FF$ and $z \in J^{k-1}_\FF$ that
\begin{align*}
	\|d_k^\FF(x \otimes y \otimes z)\|_{\pi} &= \|\varepsilon_\FF(x)\otimes \iota_\FF(y)\otimes z\|_{\pi} = |\varepsilon_\FF(x)| \cdot \|\iota_\FF(y)\|_1 \cdot \|z\|_{ J^{k-1}_\FF} \\
	&\leq \|x\|_1 \cdot \|y\|_1 \cdot \|z\|_{J^{k-1}_\FF} = \|x \otimes y \otimes z\|_{\pi}, 
\end{align*}
By definition of the projective norm, this yields that $\|d_k^\FF\|_\op \leq 1$ for each $k \in \NN$.   An analogous computation shows the boundedness of $d_k$. Thus, for all normed $G$-modules $V$ and $k \in \NN_0$, there are well-defined induced maps
\begin{align*}
	&\partial^k_\FF:= (d^{\FF}_{k+1})'\colon B(\RGF \otimes_\RR J^k_\FF,V) \to B(\RGF \otimes_\RR J^{k+1}_\FF,V), \\
	&\partial^k:= (d_{k+1})'\colon B(\RG \otimes_\RR J^k_G,V) \to B(\RG \otimes_\RR J^{k+1}_G,V).
\end{align*}
For each $k \in \NN_0$ we consider both $B(\RGF \otimes_\RR J^k_\FF,V)$ and $B(\RG \otimes_\RR J^k_G,V)$ as normed $G$-modules with respect to the norm  $$\|f\|_{C^k_b}:= 2^k \cdot \|f\|_\op.$$
The reason for this rescaling of the operator norm by a factor depending on $k$ will become apparent in the proof of Theorem \ref{TheoremSpecialInjResol}.

\begin{prop}
	\label{PropCFVinjective} Let $V$ be a Banach $G$-module and let $k \in \NN_0$. Then $B(\RR[G/\FF]\otimes_\RR J_{\FF}^k,V)$ is a relatively $\FF$-injective Banach $G$-module.
\end{prop}
\begin{proof}
	Fix $k \in \NN_0$ and put $C^k := B(\RR[G/\FF]\otimes_\RR J_{\FF}^k,V)$. It is an elementary result of functional analysis that since $V$ is a Banach $G$-module, $C^k$ is a Banach $G$-module. Let $M$ and $N$ be normed $G$-modules, let $i \in B_G(M,N)$ be strongly $\FF$-injective and for each $H \in \FF$ let $r_H \in B_H(N,M)$ with $r_H \circ i = \id_M$. 
	
	Given $f \in B_G(M,C^k)$ we consider the map $F\colon N \to C^k,$ given by $\RR$-linearly extending
	$$F(q)(gH \otimes y) = 	f((g\cdot r_H)(q))(gH\otimes y) \qquad \forall q \in N, \ g \in G, \ H \in \FF, \ y \in J_\FF^k.$$
	Since each $r_H$ is an $\RR[H]$-homomorphism, $F$ is independent of the chosen representatives of $gH \in G/\FF$, hence well-defined. Since $i$ and $f$ are $\RG$-homomorphisms, we compute for all $p \in M$, $H \in\FF$, $g \in G$ and $y \in J_\FF^k$ that
	\begin{align*}	
		(F \circ i)(p)(gH \otimes y) 
		&= f(g r_H(g^{-1}i(p)))(gH\otimes y)\\
		&= (g \cdot f(r_H(i(g^{-1}p))))(gH \otimes y) \\
		&= (g \cdot f(g^{-1}p))(gH \otimes y) = f(p)(gH\otimes y),
	\end{align*}
	so it holds that $F \circ i = f$. We further compute for all $H \in \FF$, $q \in N$, $y \in J_\FF^k$ and $g_0,g \in G$ that
	\begin{align*}
		(g_0 \cdot F(q))(gH\otimes y) &= g_0 (F(q)(g_0^{-1}gH \otimes g_0^{-1}y))\\
		&= g_0 (f(g_0^{-1}g  r_H(g^{-1}g_0 q))(g_0^{-1}gH \otimes g_0^{-1}y)) \\
		&=  f(g  r_H(g^{-1}g_0 q))(gH \otimes y) \\
		&= F(g_0\cdot q)(gH \otimes y).
	\end{align*}
	Extending $\RR$-linearly, this shows that $F \in \Hom_{\RG}(N,C^k)$. To show its boundedness, we further compute for all $H$, $q$, $y$ and $g$ as above that
	\begin{align*}
		\|F(q)(gH \otimes y)\|_V 
				&\leq  \|f(g r_H(g^{-1}q))\|_{\op}\cdot \|gH\otimes y\|_{\pi}  \\
		&\leq	 2^{-k} \cdot \|f(g r_H(g^{-1}q))\|_{C^k_b}\cdot \|gH\otimes y\|_{\pi} \\
		&\leq 2^{-k} \cdot\|f\|_{\op} \cdot \|g r_H(g^{-1}q)\|_M \cdot \|g H\|_1 \cdot \|y\|_{J^k_{\FF}} \\
		&\leq 2^{-k} \cdot\|f\|_{\op} \cdot \|r_H\|_{\op} \cdot \|q\|_N  \cdot  \|y\|_{J^k_{\FF}}  .
	\end{align*}
	Since $\|r_H\|_\op \leq 1$, this yields
		\begin{equation}
		\label{EqFgenerators}
		\left\|F(q)(gH \otimes y)\right\|_V	\leq 2^{-k} \cdot\|f\|_{\op} \cdot \|q\|_N \cdot \|y\|_{J^k_\FF} \quad \forall q \in N, \ gH \in G/\FF, \ y \in J^k_\FF.
	\end{equation}
	Let $u \in \RR[G/\FF]\otimes_\RR J^k_\FF$ be arbitrary. Let $I$ be a finite set and let $\sigma_i= \sum_{gH \in G/\FF} r_{gH,i}\cdot gH \in \RR[G/\FF]$, $y_i \in J^k_\FF$ for each $i \in I$, such that $u = \sum_{i \in I} \sigma_i \otimes y_i$. Using the triangle inequality and \eqref{EqFgenerators}, we obtain for each $q \in N$ that
	\begin{align*}
		\|F(q)(u)\|_V 
		&\leq \sum_{i \in I} \sum_{gH\in G/\FF} |r_{gH,i}| \cdot 2^{-k} \cdot\|f\|_\op \cdot \|q\|_N \cdot \|y_i\|_{J^k_\FF} \\
		&=2^{-k} \cdot\|f\|_\op \cdot \|q\|_N \cdot  \sum_{i \in I} \|\sigma_i\|_1 \cdot \|y_i\|_{J^k_\FF}  .
	\end{align*}
	Since this holds for every representation of $u$ as a finite sum of homogeneous tensors, we derive
	$$\|F(q)(u)\|_V \leq 2^{-k} \cdot\|f\|_\op \cdot \|q\|_N \cdot \|u\|_\pi \qquad \forall q \in N, \ u \in \RR[G/\FF]\otimes_\RR J_\FF^k,$$
	hence $\|F(q)\|_{C^k_b} = 2^{k}\cdot  \|F(q)\|_\op \leq \|f\|_\op \cdot \|q\|_N$ for all $q \in N$. 
	Thus, $F$ is bounded linear with $\|F\|_\op \leq \|f\|_\op $, which eventually shows that $C^k$ is relatively $\FF$-injective.
\end{proof}

Using the previous results, we can construct an explicit bounded $\FF$-injective resolution 
which allows us to describe bounded Adamson cohomology similar to the approach of Section 2.

\begin{theorem}
	\label{TheoremSpecialInjResol}
	Let $V$ be a Banach $G$-module. Then 
	{\small \begin{equation*}
			0 \longrightarrow V \stackrel{\varepsilon_{\FF}'}{\longrightarrow} B(\RGF,V)  \stackrel{\partial^0_\FF}\longrightarrow \ldots \longrightarrow  B(\RGF\otimes_\RR J_\FF^k,V) \stackrel{\partial^k_\FF}\longrightarrow B(\RGF\otimes_\RR J^{k+1}_\FF,V)\longrightarrow \ldots  
	\end{equation*}}
	is a bounded $\FF$-injective resolution of $V$.
\end{theorem}
\begin{proof}
	Let again $C^k := B(\RGF\otimes_\RR J_\FF^k,V)$ for each $k \in \NN_0$. By Proposition \ref{PropCFVinjective}, $C^k$ is relatively $\FF$-injective for each $k \in \NN_0$, so it suffices to show that it is exact and bounded $\FF$-strong. Let $H \in \FF$ be fixed. For all $k \in \NN_0$ we define $t_{H,k}\colon \RGF \otimes_\RR J^k_\FF \to \RGF\otimes_\RR J^{k+1}_\FF$ by $\RR$-linearly extending
	$$t_{H,k} (gK \otimes y) = H\otimes (gK-H)\otimes y \qquad \forall  K \in \FF, \ g\in G, \ y \in J^k_\FF.$$
	For each $k \in \NN$ we then put
	$$\tau_{H,k}\colon C^k \to C^{k-1}, \quad \tau_{H,k} := (t_{H,k-1})'.$$
	In complete analogy with the proof of Lemma \ref{Adamson:projective}, one shows that $t_{H,k-1}$ and thus $\tau_{H,k}$ is an $\RR[H]$-homomorphism for each $k \in \NN$. We further put $\tau_{H,0}\colon C^0 \to V$, $\tau_{H,0}(\xi)=\xi(H)$, which is evidently an $\RR[H]$-homomorphism and which satisfies
	\begin{align*}
		\|\tau_{H,0}(\xi)\|_V = \|\xi(H)\|_V \leq \|\xi\|_\op \cdot \|H\|_1 = \|\xi\|_{C^0_b}
	\end{align*}
	for all $\xi \in C^0$, such that $\|\tau_{H,0}\|_\op \leq 1$. \medskip
	
	Fix $k \in \NN$. For all $K \in \FF$, $g \in G$ and $y \in J^{k-1}_\FF$, we compute that
	\begin{align*}
		\|t_{H,k-1}(gK\otimes y)\|_\pi 
		&= \|H\|_1 \cdot \|gK-H\|_1 \cdot  \|y\|_{J^{k-1}_\FF} 
		\leq 2\cdot  \|y\|_{J^{k-1}_\FF} = 2\cdot  \|gK\otimes y\|_\pi.
	\end{align*}
In analogy with the proof of Proposition \ref{PropCFVinjective}, we derive that $\|t_{H,k-1}\|_\op \leq 2$. At this point, our rescaling of the operator norms by $2^k$ comes into play. We derive for each $\xi \in C^k$ that 
	$$\|\tau_{H,k}(\xi)\|_{C^{k-1}_b} = 2^{k-1}\cdot \|\tau_{H,k}(\xi)\|_\op = 2^{k-1} \cdot \|\xi\circ t_{H,k-1}\|_{\op} \leq 2^k\cdot \|\xi\|_\op = \|\xi\|_{C^k_b}.$$
	Hence, $\|\tau_{H,k}\|_\op \leq 1$. It remains to show that the $\tau_{H,k}$ form a contracting  homotopy of $(C^*,\partial_\FF^*)$. Let $\xi \in C^k$. For all $g_0, g_1 \in G$, $H_0,H_1,H_2 \in \FF$ and $y_1\in J^{k-1}_{\FF}$ we compute with $y:= (g_1H_1-H_2) \otimes y_1$ that
	\begin{align*}
		(\tau_{H,k+1}\circ \partial_\FF^k)(\xi)(g_0H_0 \otimes y) &= (\xi \circ d^{\FF}_{k+1}\circ t_{H,k})(g_0H_0\otimes y) \\
		&=  \xi(\varepsilon_{\FF}(H) (g_0H_0\otimes y - H \otimes y))\\
		&= \xi(g_0H_0 \otimes y) - \xi(H \otimes y)
	\end{align*}
	and
	\begin{align*}
		(\partial^{k-1}_\FF\circ \tau_{H,k})(\xi)(g_0H_0 \otimes y) 
		&= (\xi \circ t_{H,k-1})(g_1H_1\otimes y_1) - (\xi \circ t_{H,k-1})(H_2\otimes y_1) \\
		&= \xi(H\otimes (g_1H_1-H)\otimes y_1) - \xi(H \otimes (H_2-H)\otimes y_1)\\
		&=\xi(H \otimes (g_1H_1-H_2)\otimes y_1) = \xi(H \otimes y).
	\end{align*}
	Extending these computations $\RR$-linearly shows that $\tau_{H,k+1}\circ \partial_\FF^k+\partial^{k-1}_\FF\circ \tau_{H,k}=\id_{C^k}$ for all $k \in \NN$.
	In addition, we compute for all $\xi \in C^0=B_G(\RGF,V)$, $H_0 \in \FF$  and $g_0 \in G$ that
	\begin{align*}
		(\tau_{H,1}\circ \partial_\FF^0)(\xi)(g_0H_0)
		&=(\xi \circ d_1^\FF) (H \otimes (g_0H_0-H)) = \xi(g_0H_0) - \xi(H), \\
		(\varepsilon_\FF'\circ \tau_{H,0})(\xi)(g_0H_0) &=  \varepsilon_\FF(g_0H_0)\cdot \tau_{H,0}(\xi) = \xi(H),
	\end{align*}
	so extending $\RR$-linearly shows that $\tau_{H,1}\circ \partial_\FF^0+\varepsilon_\FF'\circ \tau_{H,0}=\id_{C^0}$. 
	Moreover, 
	$$	(\tau_{H,0} \circ \varepsilon_\FF')(v) = (\varepsilon'_\FF(v))(H) = \varepsilon_\FF(H) \cdot v = v \qquad \forall v \in V.$$
	Putting our computations together, we have shown that the maps $(\tau_{H,k})_{k \in \NN_0}$ form a contraction of $(C^*,\partial_\FF^*)$. The existence of a contracting homotopy immediately implies the exactness of the complex and since $H \in \FF$ was chosen arbitrarily, we have shown that $(C^*,\partial_\FF^*)$ is bounded $\FF$-strong.  
\end{proof}

\begin{corollary}
	\label{CorBoundedRes}
	Let $V$ be a Banach $G$-module. Then $(B(\RR[G]\otimes_\RR J_{G}^*,V),\partial^*)$ is a relatively injective strong resolution of $V$. 
\end{corollary}
\begin{proof}
	This follows from Theorem \ref{TheoremSpecialInjResol} since for $\FF_0=\{1\}$ there is an obvious isomorphism of cochain complexes $(B(\RR[G]\otimes_\RR J_{G}^*,V),\partial^*) \cong (B(\RR[G/\FF_0]\otimes_\RR J_{\FF_0}^*,V),\partial^*_{\FF_0})$.
\end{proof}

\begin{definition}
	Given a Banach $G$-module $V$ and $k \in \NN_0$ we consider the normed real vector spaces	\begin{align*}
		C^k_b([G:\FF];V) &:= B_G(\RR[G/\FF]\otimes_\RR J_{\FF}^k,V)= (B(\RR[G/\FF]\otimes_\RR J_{\FF}^k,V))^G, \\
		C^k_b(G;V) &:= B_G(\RG \otimes_\RR J_G^k,V)=(B(\RG \otimes_\RR J_G^k,V))^G,
	\end{align*}
	equipped with  the norm $ \|f\|_{C^k_b}= 2^{k} \cdot \|f\|_\op$.
\end{definition}
By Theorem \ref{TheoremSpecialInjResol} and Corollary \ref{CorBoundedRes}, we can compute the bounded Adamson cohomology of $(G,\FF)$ and the bounded cohomology of $G$, respectively, with coefficients in a Banach $G$-module $V$ as
\begin{align*}
H^k_b([G:\FF];V) &= H^k(C^*_b([G:\FF];V),\partial^*_\FF), \\
H^k_b(G;V) &= H^k(C^*_b(G;V),\partial^*_G), \end{align*}
for every $k \in \NN_0$.

\subsection{Relations between bounded and unbounded cohomology groups} We are going to establish a map from bounded Adamson cohomology to bounded cohomology, study the construction of comparison maps in bounded cohomology and establish their analogues in bounded Adamson cohomology. \medskip

Let $\sigma_{H,k}:= \sigma_{H,k}^\RR\colon \RG \otimes_\RR J^k_G \to \RGF \otimes_\RR  J_\FF^k$ be given as in \eqref{EqsigmaR}. To describe bounded analogues of the maps $\rho_{\FF}^k$ from Definition \ref{DefCanonicalBerstein}, we need to establish additional properties. 

\begin{lemma}
	\label{LemmasigmaHB}
	For all $H_1,H_2 \in \FF$ the maps $\sigma_{H_1,*}', \sigma_{H_2,*}'\colon C^*_b([G:\FF];V) \to C^*_b(G;V)$ are bounded linear and chain homotopic.
\end{lemma}
\begin{proof}
	Given $H \in \FF$ we consider $\pi_H\in \Hom_{\RG}(\RG, \RGF)$ given by $\pi_H(g)=gH$ for all $g \in G$. A straightforward calculation shows that $\|\pi_H\|_{\op}=1$, hence $q_H:=\pi_H|_{J_G}\colon J_G \to J_\FF$ satisfies $\|q_H\|_{\op}\leq 1$ as well. It is evident that $\sigma_{H,k}= \pi_H \otimes q_H^{\otimes k}$, so it follows from \cite[Proposition 2.2.3]{Ryan} or \cite[p. 27]{DefantFloret}, which generalizes straightforwardly to tensor products of more than two factors, that $\|\sigma_{H,k}\|_\op = \|\pi_H\|_\op \cdot \|q_H\|_\op^k\leq 1$. Thus, the induced map $\sigma_{H,k}'\colon  C^*_b([G:\FF];V) \to C^*_b(G;V)$ is bounded as well.	Since $ C^*_b([G:\FF];V)$ is a bounded $\FF$-injective resolution of $V$, it follows from an analogue of the fundamental lemma of homological algebra in the bounded setting, see \cite[Proposition 3.9]{LiBounded}, that $\sigma_H'$ is unique up to a bounded chain homotopy. Thus, for all $H_1,H_2 \in \FF$, the maps $\sigma_{H_1,*}'$ and $\sigma'_{H_2,*}$ are bounded chain homotopic.
\end{proof}

\begin{definition}
	Given $H \in \FF$ we consider the cochain map $\sigma_{H,*}'\colon C_b^*([G:\FF];V) \to C_b^*(G;V)$ and let
	$$ \rho^k_{b,\FF}\colon H^k_b([G:\FF];V) \to H^k_b(G;V)$$
denote the map induced by $\sigma_{H,*}'$ in cohomology for each $k \in \NN_0$. By Lemma \ref{LemmasigmaHB}, $\rho^k_{b,\FF}$ is independent of the choice of $H$. 
\end{definition}

Important constructions from bounded cohomology are the so-called \emph{comparison maps} between bounded and ordinary cohomology, which are induced by the inclusion of bounded cochains, see \cite{FrigerioBook} or \cite[Section 9.2]{Monod}. In our setting their definition reads as follows.

\begin{definition}
	\label{DefComp}
	Let $V$ be a Banach $G$-module and let $k \in \NN_0$. 
	\begin{enumerate}[a)]
		\item The \emph{comparison map of $G$ with coefficients in $V$} is the map
		$$ \comp^k_{V}\colon H^k_b(G;V)\to H^k(G;V),$$
		which is induced by the inclusion of subcomplexes $C^k_b(G;V) \hookrightarrow C^k(G;V)$. \item The \emph{comparison map of $(G,\FF)$ with coefficients in $V$} is the map
		$$ \comp^k_{\FF,V}\colon H^k_b([G:\FF];V)\to H^k([G:\FF];V),$$
		which is induced by the inclusion of subcomplexes $C^k_b([G:\FF];V) \hookrightarrow C^k([G:\FF];V)$.
	\end{enumerate}
\end{definition}

The following diagram evidently commutes for all $k \in \NN_0$ and all Banach $G$-modules $V$, since the underlying cochain maps are easily seen to commute:
\begin{equation}
	\label{EqDiagComp}
	\begin{tikzcd}
		H^k_b([G:\FF];V)\ar[r, "\comp_{\FF,V}^k"] \ar[d, "\rho_{b,\FF}^k", swap]&  H^k([G:\FF];V) \ar[d, "\rho_{\FF}^k"]\\
	H^k_b(G;V) \ar[r, "\comp_V^k", swap] & H^k(G;V).
	\end{tikzcd}	
\end{equation}
Moreover, it is evident from their definitions that comparison maps commute with pushforwards, i.e. that the following diagrams commute for all Banach $G$-modules $V$ and $W$, all $F \in B_G(V,W)$ and all $k \in \NN_0$:
\begin{equation}
	\label{EqCompPushfor}
	\begin{tikzcd}
		H^k_b(G;V) \ar[r, "\comp_V^k"] \ar[d, "F_*", swap] & H^k(G;V) \ar[d, "F_*"]\\
		H^k_b(G;W) \ar[r, "\comp_W^k", swap] & H^k(G;W),
	\end{tikzcd} \qquad  \qquad 
	\begin{tikzcd}
		H^k_b([G:\FF];V)\ar[r, "\comp_{\FF,V}^k"] \ar[d, "F_*", swap]& H^k([G:\FF];V)\ar[d, "F_*"] \\
		H^k_b([G:\FF];W)\ar[r, "\comp_{\FF,W}^k",swap] & H^k([G:\FF];W) .
	\end{tikzcd}
\end{equation}
\section{Johnson classes and Adamson--Johnson classes} We want to introduce bounded cohomology classes which generalize a class introduced by B. Johnson in \cite{Johnson} whose vanishing is equivalent to the amenability of $G$, see \cite[Section 3.4]{FrigerioBook} or  \cite[Theorem B]{BNN}. We note that Johnson's class is the class denoted by $\gamma_1$ below. \medskip

Throughout this section let $\FF$ be a family of subgroups of $G$. In the following, given a bounded cocycle $\sigma \in C^*_b(G;V)$ or $\sigma \in C^*_b([G:\FF];V)$, we shall always denote its \emph{bounded} cohomology class by $[\sigma]_b \in H_b^*(G;V)$ or $[\sigma]_b \in H_b^*([G:\FF];V)$, respectively.

\subsection{Construction of bounded cohomology classes} For any normed vector space $V$, we let
$$\vartheta_V\colon V \to V'', \qquad \vartheta_V(v) = \delta_v,$$
denote the canonical embedding of $V$ into its bidual space. Here, $\delta_v\colon V' \to \RR$, $\delta_v(f)=f(v)$, for each $v \in V$. It is easy to check that if $V$ is a normed $G$-module, then $\vartheta_V \in B_G(V,V'')$.\medskip

Given $k \in \NN$, we put 
$$\II^k_\FF := (J^k_\FF)'', \qquad \vartheta_{\FF,k}\colon J^k_\FF \to \II^k_\FF, \ \ \vartheta_{\FF,k} := \vartheta_{J^k_\FF},$$
and
$$\II^k_G := (J^k_G)'', \qquad \vartheta_{k}\colon J^k_G \to \II^k_G, \ \ \vartheta_{k} := \vartheta_{J^k_G}.$$

\begin{definition} 
	\label{DefJohnson}
	Let $k \in \NN$.
	\begin{enumerate}[a)]
		\item  Consider the cochain $j_{\FF,k} \in C^k_b([G:\FF];\II_\FF^k)$,
		that is given by $\RR$-linearly extending 
		$$j_{\FF,k}(\sigma \otimes y) =   \varepsilon_\FF(\sigma) \cdot \vartheta_{\FF,k}(y) \qquad \forall \sigma \in \RR[G/\FF], \ \ y \in J^k_{\FF}.$$ 
		One checks without difficulties that $j_{\FF,k}$ is a bounded cocycle and we call the class
		$$\Gamma_{\FF,k} := [j_{\FF,k}]_b \in H^k_b([G:\FF];\II^k_\FF)$$
		the \emph{$k$-th Adamson--Johnson class of $G$ relative to $\FF$}. 
		\item We call the bounded cohomology class
		$$\gamma_{\FF,k} := \rho^k_{b,\FF}(\Gamma_{\FF,k}) \in H^k_b(G;\II_\FF^k) $$
		the \emph{$k$-th Johnson class of $G$ relative to $\FF$.}
		\item We consider the cochain $j_k \in C^k_b(G;\II^k_G)$,
		given by $\RR$-linearly extending
		$$j_k(x \otimes y) = \varepsilon_G(x) \cdot \vartheta_{k}(y).$$
		Again one checks without difficulties that $j_k$ is a bounded cocycle and we call 
		$$\gamma_k := [j_k]_b \in H^k_b(G;\II_G^k)$$
		the \emph{$k$-th Johnson class of $G$}.
	\end{enumerate}
\end{definition}
\begin{remark}
\begin{enumerate}
	\item 
	In \cite[Chapter 3]{FrigerioBook} the class $\gamma_1$ is actually defined as a bounded cohomology class with coefficients in $((\RG)'/\RR)'$. In fact, $((\RR[S])'/\RR)'$ is isomorphic as a Banach $G$-module to $J_S''$ for each $G$-set $S$, so our definition encompasses the original one. To see this, fix $s_0 \in S$ and consider  $F_S\in B(\RR[S],J_S)$, $F_S(x) = x-\varepsilon_S(x)\cdot s_0$.	With $\Psi_S\colon (\RR[S])' \to J_S'$, $\Psi_S(\varphi)=\varphi|_{J_S}$, a simple calculation shows that  $\Psi_S \circ F_S' = \id_{J_S'}$, so that $F_S'$ is injective. Let $V_0:= \{\varphi \in (\RR[S])' \mid \varphi(s_0)=0\}$. Evidently, $\im F_S' \subset V_0$. Conversely, one checks that $F_S'(\varphi|_{J_S})=\varphi$ for all $\varphi \in V_0$, so $F_S'$ restricts to an isomorphism $F_{0}\colon J_S' \stackrel{\cong}{\longrightarrow} V_0$. The augmentation $\varepsilon_S \colon \RR[S] \to \RR$ induces $\varepsilon'_S\colon \RR \to (\RR[S])'$ and we equip $$(\RR[S])'/\RR := (\RR[S])'/\im \varepsilon_S'$$
with the quotient norm.  The restriction $q_0\colon V_0 \to (\RR[S])'/\RR$ of the projection is injective, since one easily checks that $V_0 \cap \im \varepsilon_S' = \{0\}$. Thus, 	$$\Phi_0\colon J_S' \to (\RR[S])'/\RR, \qquad \Phi_0:= q_0 \circ  F_0 ,$$
	is a bounded isomorphism. 	It remains to show that $\Phi_0$ is $\RR[G]$-linear. A straightforward calculation shows that 
	$$F_S'(g \cdot f)(x)=(g\cdot F_S'(f))(x) + \varepsilon'_S(f(s_0-g^{-1}s_0))(x) \in (g \cdot F_S'(f))(x) + \im \varepsilon'_S$$
	for all $x \in \RR[S]$, $g \in G$ and $f \in J_S'$. Hence, $\Phi_0$ is an isomorphism of Banach $G$-modules and so is $\Phi_0'\colon ((\RR[S])'/\RR)' \to J_S''$.
\item Likewise, the class $\Gamma_{\FF,1}\in H^1_b([G:\FF];\II_\FF)$ coincides with the class $[J_\FF]$ introduced by K. Li in \cite[Theorem 4.5]{LiBounded} under the name of relative Johnson class under the identification of coefficient modules from part (1).
\end{enumerate}
\end{remark}

For each $H \in \FF$ and $k \in \NN$, the map  $\sigma_{H,k-1}$ from \eqref{EqsigmaR} restricts to  $\sigma_{H,{k-1}}|_{J_G^k}\in B_G(J_G^k,J_\FF^k)$. 
Passing to bidual spaces, this map induces  $(\sigma_{H,k-1}|_{J_G^k})''\in B_G(\II_G^k,\II_\FF^k)$. In the following we consider the case of $H=\{1\}$, the trivial subgroup, and put 
$$\sigma_{\FF,k}:= \sigma_{\{1\},k-1}|_{J_G^k}\in B_G(J^k_G, J^k_\FF), \qquad \psi_{\FF,k} := (\sigma_{\{1\},k-1}|_{J_G^k})''\in B_G(\II^k_G, \II^k_\FF).$$
The following diagram then commutes by definition of the maps involved:
\begin{equation}
	\label{EqDiagpsibidual}
	\begin{tikzcd}
		J^k_G \ar[rr, "\sigma_{\FF,k}"] \ar[d, "\vartheta_{k}", swap] & &J^k_\FF \ar[d, "\vartheta_{\FF,k}"]\\
		\II^k_G\ar[rr, "\psi_{\FF,k}", swap]  & &\II^k_\FF.
	\end{tikzcd}
\end{equation}

The next result summarizes relations between (Adamson-)Johnson classes and the Berstein classes relative to $\FF$ and Adamson canonical classes from Section 2. 

\begin{prop}
	\label{PropClassRelations}
	For all $k \in \NN$ the following equations hold:
	\begin{align*}
		&\text{a)} \ \ \gamma_{\FF,k} = (\psi_{\FF,k})_*(\gamma_k)\in H^k_b(G;\II_\FF^k), \\
		&\text{b)} \ \ \comp^k_{\FF,\II^k_\FF}(\Gamma_{\FF,k}) = (\vartheta_{\FF,k})_*(\phi^{k}_{\FF,\RR}) \in H^k([G:\FF];\II_\FF^k),\\
		&\text{c)} \ \ \comp^k_{\II^k_\FF}(\gamma_{\FF,k})= (\vartheta_{\FF,k})_*(\omega^k_{\FF,\RR}) \in H^k(G;\II_\FF^k), \\
		&\text{d)} \ \ \comp^k_{\II^k_G}(\gamma_k) = (\vartheta_k)_*(b^k_\RR) \in H^k(G;\II_G^k).
	\end{align*}
\end{prop}
\begin{proof}
	Let $j_k \in C^k_b(G;\II^k_G)$ and $j_{\FF,k} \in C^k_b([G:\FF];\II^k_\FF)$ be given as in Definition \ref{DefJohnson}.
	\begin{enumerate}[a)]
		\item  By definition of the maps involved, $\gamma_{\FF,k}=\rho_{b,\FF}^k(\Gamma_{\FF,k})$ is represented by the cocycle $j_{\FF,k} \circ \sigma_{\{1\},k} \colon \RG \otimes_\RR J^k_G \to \II_\FF^k$. One further checks from the definitions of the maps involved that the diagram $$	\begin{tikzcd}
			\RG\otimes_\RR J^k_G \ar[d, "\varepsilon_G \otimes \id", swap] \ar[r, "\sigma_{\{1\},k}"] & \RGF \otimes_\RR J^k_\FF  \ar[d, "\varepsilon_\FF \otimes \id" ]\\
		J^k_G \ar[r, "\sigma_{\FF,k}",swap] & J^k_\FF
		\end{tikzcd}$$ commutes. Using this commutativity and the one of \eqref{EqDiagpsibidual}, we derive that
		\begin{align*}
			j_{\FF,k}\circ \sigma_{\{1\},k} &= \vartheta_{\FF,k} \circ (\varepsilon_\FF \otimes \id_{J^k_\FF}) \circ \sigma_{\{1\},k} 
			= \psi_{\FF,k} \circ \vartheta_{k} \circ (\varepsilon_G \otimes \id_{J^k_G}) = \psi_{\FF,k} \circ j_k.
		\end{align*}
		Since $\psi_{\FF,k} \circ j_k$ is a representative of $(\psi_{\FF,k})_*(\gamma_k)$, the claim follows directly when passing to bounded cohomology.
		\item As discussed after Definition \ref{RealAdamsonBersteinclass}, the class $(\vartheta_{\FF,k})_*(\phi^{k}_{\FF,\RR})$ is represented by the cocycle
		$$\vartheta_{\FF,k} \circ (\varepsilon_\FF \otimes \id_{J^k_\FF}) = j_{\FF,k}.$$
		Since $\Gamma_{\FF,k}$ is given as the \emph{bounded} cohomology class of $j_{\FF,k}$ and since the comparison map is induced by the inclusion of bounded cocycles, the claim follows immediately.
		\item This follows immediately from part b) and the commutativity of \eqref{EqDiagComp}.
		\item By Definition \ref{RealAdamsonBersteinclass}, $(\vartheta_k)_*(b^k_\RR)$ is represented by $\vartheta_k \circ(\varepsilon_{G} \otimes \id_{J^k_G})=j_k,$
		while $\gamma_k$ is given by the \emph{bounded} cohomology class of this cocycle. The claim thus follows as in b).
	\end{enumerate}
\end{proof}

\subsection{Universality properties of Adamson--Johnson classes} A crucial property of the cohomology classes introduced in the previous section is a universality property, i.e. an analogue of Proposition \ref{PropUniv} in the bounded setting. Such a result requires some technical preparations.

\begin{lemma}
	\label{LemmaExtendBidualNormed}
	Let $X$ and $Y$ be normed $G$-modules and $f \in B_G(X,Y')$. Then there exists $F \in B_G(X'',Y')$ with $F \circ \vartheta_X = f$. 
\end{lemma}
\begin{proof}
	Consider the map $\widetilde{f}\colon X \times Y \to \RR$, $\widetilde{f}(x,y):= (f(x))(y)$. Evidently, $\widetilde{f}$ is bilinear and we compute for all  $g \in G$, $x\in X$ and $y \in Y$ that
	$$\widetilde{f}(gx,gy)= (f(gx))(gy)= (g\cdot f(x))(gy) =(f(x))(g^{-1}gy) = (f(x))(y) = \widetilde{f}(x,y), $$
	i.e. $\widetilde{f}$ is $G$-invariant with respect to the diagonal $G$-action. Moreover, 
	$$\left|\widetilde{f}(x,y) \right|=|(f(x))(y)|\leq \|f\|_\op \cdot \|x\|_X \cdot \|y\|_Y \qquad \forall x \in X, \ y \in Y, $$
	so $\widetilde{f}$ is bounded. Let $S\colon Y \to X'$, $(S(y))(x):= \widetilde{f}(x,y)$, for all $x \in X$, $y \in Y$.  We compute from the $G$-invariance of $\widetilde{f}$ for all $g \in G$, $x \in X$ and $y \in Y$ that 
	\begin{equation*}
		(S(g\cdot y))(x)= \widetilde{f}(x,gy)= \widetilde{f}(g^{-1}x,y) = S(y)(g^{-1}x) = (g \cdot S(y))(x),
	\end{equation*}
	hence $S$ is $G$-equivariant. It further follows easily that $S \in B_G(Y,X')$.

	Consider $F\colon X'' \to Y'$, $(F(x''))(y):= x''(S(y))$. Then $F$ is linear and for all $g \in G$, $x'' \in X''$ and $y \in Y$ we compute that
	$$F(g \cdot x'')(y) = (g\cdot x'')(S(y)) = x''(g^{-1} S(y)) = x''(S(g^{-1}y))= F(x'')(g^{-1}y) = (g\cdot F(x''))(y).$$
	One further derives from the boundedness of $S$ that $F \in B_G(X'',Y')$. We compute for all $x \in X$ and $y \in Y$ that
	$$(F \circ \vartheta_X(x))(y)= (\vartheta_X(x))(S(y))= (S(y))(x)= (f(x))(y), $$
	which shows that $F \circ \vartheta_X = f$. 
\end{proof}

The following is a direct consequence of Lemma \ref{LemmaExtendBidualNormed}.

\begin{lemma}
	\label{LemmaExtendIk} Let $k \in \NN$ and let $V$ be a normed $G$-module.
	\begin{enumerate}[a)]
		\item For each $f \in B_G(J_\FF^k,V')$ there exists $F \in B_G(\II_\FF^k,V')$ with $F \circ \vartheta_{\FF,k}=f$.
		\item  For each $f \in B_G(J_G^k,V')$ there exists $F \in B_G(\II_G^k,V')$ with $F \circ \vartheta_{k} = f$. \end{enumerate}
\end{lemma}
\begin{proof}
	We obtain part a) immediately from taking $X=J_\FF^k$ and part b) from taking $X=J_G^k$ in Lemma \ref{LemmaExtendBidualNormed}.
\end{proof}

The next theorem is the universality result for Johnson classes and Adamson--Johnson classes mentioned in the introduction. Note that the theorem is only stated for \emph{dual} coefficient modules, i.e. Banach $G$-modules whose underlying Banach space is a dual space.
\begin{theorem}
	\label{TheoremBoundedUniv} Let $V$ be a normed $G$-module and let $k \in \NN$.
	\begin{enumerate}[a)]
		\item For each $u \in H^k_b([G:\FF];V')$ there exists $\psi\in B_G(\II^k_{\FF},V')$ with 
		$u = \psi_*(\Gamma_{\FF,k}).$
		\item For each $v \in H^k_b(G;V')$ there exists $\eta \in B_G(\II^k_G,V')$ with $v=\eta_*(\gamma_k)$.
	\end{enumerate}
	
\end{theorem}
\begin{proof}
	\begin{enumerate}[a)]
		\item 	Let $c \in C^k_{b}([G:\FF];V')$ be a cocycle representing $u$. By the cocycle condition, $c \circ f_k^{\FF} \circ g_{k+1}^{\FF}=0$ and since $g_{k+1}^\FF$ is surjective by the exactness of \eqref{EqSESreal}, $c$ vanishes on $\im f_k^\FF$. Hence, $c$ induces a linear $G$-map $\hat{c}\colon J^k_{\FF} \to V'$ with $\hat{c} \circ g^{\FF}_k=c$. Since $c$ is bounded, we compute for any fixed $H \in \FF$ and for all $y \in J_{\FF}^k$ that
		$$\|\hat{c}(y)\|_{V'} = \|\hat{c}(g_k^\FF(  H \otimes y))\|_{V'} = \|c(H \otimes y) \|_{V'} \leq \|c\|_{\op} \cdot \|H \otimes y\|_{\pi} =  \|c\|_\op \cdot \|y\|_{J_\FF^k}. $$
		Hence, $\hat{c} \in B_G(J_{\FF}^k,V')$. By Lemma \ref{LemmaExtendIk}.a), there exists  $\psi \in B_G(\II_\FF^k,V')$, such that $\psi \circ \vartheta_{\FF,k} = \hat{c}$. 		Thus, 
		\begin{align*}
			c = \hat{c} \circ g_k^\FF = \psi \circ \vartheta_{\FF,k} \circ g_k^\FF = \psi \circ j_{\FF,k}.
		\end{align*}
		Passing to bounded cohomology shows that
		$$u= [c]_b = [\psi \circ j_{\FF,k}]_b = \psi_*([j_{\FF,k}]_b) = \psi_*(\Gamma_{\FF,k}).$$
		\item This is shown along the same lines as part a) by employing Lemma \ref{LemmaExtendIk}.b).
	\end{enumerate}
\end{proof}

The universality theorem has an interesting consequence for comparison maps.

\begin{theorem}
	\label{TheoUnivComp}
	Let $k \in \NN$ and let $V$ be a normed $G$-module.
	\begin{enumerate}[a)]
		\item If $\comp_{\FF,\II^k_\FF}^k(\Gamma_{\FF,k})= 0,$ then $\comp^k_{\FF,V'}\colon H^k_b([G:\FF];V')\to H^{k}([G:\FF];V')$ vanishes.
		\item If $\comp^k_{\II^k_G}(\gamma_k)=0$, then $\comp^k_{V'}\colon H^k_b(G;V') \to H^k(G;V')$ vanishes. 
	\end{enumerate}
\end{theorem}
\begin{proof}
	\begin{enumerate}[a)]
		\item Let $u \in H^k_b([G:\FF];V')$. By Theorem \ref{TheoremBoundedUniv}.a), there exists $\psi \in B_G(\II^k_{\FF},V')$, such that $u =\psi_*(\Gamma_{\FF,k})$. By the commutativity of the right-hand diagram in \eqref{EqCompPushfor} in the case of $F=\psi$, we derive from $\comp^k_{\FF,\II^k_{\FF}}(\Gamma_{\FF,k})=0$ that 
		$$\comp^k_{\FF,V'}(u) = \comp^k_{\FF,V'}(\psi_*(\Gamma_{\FF,k}))= \psi_*(\comp_{\FF,\II_{\FF}^k}^k(\Gamma_{\FF,k}))=0.$$
		\item This is shown along the lines of part a), using Theorem \ref{TheoremBoundedUniv}.b).
	\end{enumerate}
\end{proof}

\subsection{Non-vanishing of comparison maps and simplicial volume} We want to apply the universality results for Johnson classes to derive a criterion for the non-vanishing of top-degree comparison maps of oriented Poincar\'{e} duality groups. 

A first auxiliary result is that for those groups the $\cd(G)$-th power of the \emph{real} Berstein class does not vanish. To establish this, we make use of a standard fact in group cohomology. We provide an explicit proof, since we are not aware of such a proof occurring in the literature.

\begin{lemma}
	\label{LemmaExtScal}
	Let $k \in \NN$, let $G$ be a group of type $FP_k$ and let $M$ be a $\ZZ[G]$-module. Consider $\alpha\in \Hom_{\ZG}(M,\RR\otimes M)$, $\alpha(m)=1\otimes m$. If $u \in H^k(G;M)$ is not an additive torsion class, then $\alpha_*(u)\neq 0$.
\end{lemma}
\begin{proof}

Let $(P_*, p_*)$ be a projective resolution over $\ZZ$ over $\ZZ[G]$ for which $P_i$ is finitely generated for every $i \leq k$. Let $\delta_M$ and $\delta_{\RR\otimes M}$ denote the differentials of the cochain complexes $\Hom_{\ZG}(P_*,M)$ and $\Hom_{\ZG}(P_*,\RR \otimes M)$, respectively, which are induced by $p_*$. Let $u \in H^k(G;M)$ be a non-torsion class, and choose a representing cocycle $\xi \in \Hom_{\ZG}(P_k,M)$. Suppose that $\alpha_*(u) = 0$. Since $P_i$ is finitely generated and projective, the map 
$$ \gamma_i \colon \RR \otimes \Hom_{\ZG}(P_i,M) \longrightarrow \Hom_{\ZG}(P_i, \RR \otimes M), \qquad \gamma_i(r \otimes f)(x) = r \otimes f(x), $$ is an isomorphism for each $i \leq k$. Since $\RR$ is a flat $\ZZ$-module, the sequence $$ 0 \to \RR \otimes \im \delta^{k-1}_M \to \RR \otimes \ker \delta^k_M \to \RR \otimes H^k(G;M) \to 0. $$ is short exact. By assumption, $u$ is not a torsion class, so that $1 \otimes u \neq 0 \in \RR \otimes H^k(G;M)$ and consequently $1 \otimes \xi \notin \im(\id \otimes \delta^{k-1}_M) $. However, $\alpha \circ \xi$ is by assumption a coboundary, and since $\alpha \circ \xi = \gamma_k(1 \otimes \xi)$, there exists $\eta \in \Hom_{\ZG}(P_{k-1},\RR \otimes M)$ with $\delta^{k-1}_{\RR \otimes M}(\eta) = \gamma_k(1 \otimes \xi)$. But $\gamma_{k-1}$ is an isomorphism, so there exists a unique $\mu \in \RR \otimes \Hom_{\ZG}(P_{k-1},M)$ with $\eta = \gamma_{k-1}(\mu)$. Hence, 
$$ (\gamma_k\circ (\id_{\RR} \otimes \delta^{k-1}_M))(\mu) = (\delta^{k-1}_{\RR \otimes M}\circ\gamma_{k-1})(\mu) = \gamma_k(1 \otimes \xi) = \alpha \circ \xi. $$ Since $\gamma_k$ is an isomorphism, we get that $(\id_{\RR} \otimes \delta^{k-1}_M)(\mu) = 1 \otimes \xi$, resulting in a contradiction. Therefore, $\alpha_*(u) \neq 0$.   \end{proof}

\begin{prop} 
	\label{PropRealBerstein}
	Assume that $G$ is a duality group and let $n= \cd(G)$. Then $b^n_\RR \neq 0$. 	
\end{prop}
\begin{proof}
	By \cite[Theorem 3.4]{DranishRudyak}, it holds that $b^n \neq 0 \in H^n(G;I^n_G)$. Since $G$ is a duality group,  $H^n(G;\ZG)$ has no $\ZZ$-torsion, see \cite[Theorem VIII.10.1]{Brown}. It further holds that $G$ is of type $FP_n$ and that $H^n(G;\ZG)\neq \{0\}$, see \cite[Theorem VIII.10.1]{Brown}. By the universality of $b^n$, see \cite[Corollary 3.5]{DranishRudyak}, there exists $\psi \in \Hom_{\ZG}(I_G^n,\ZG)$, such that $\psi_*(b^n)\neq 0\in H^n(G;\ZG)$.
Hence, $b^n$ is not a $\ZZ$-torsion class in $H^n(G;I^n_G)$. By Lemma \ref{LemmaExtScal}, this implies that $b^n_\RR \neq 0$. 
\end{proof}	

In the remainder of this section, given a normed $G$-module $A$ we will denote the quotient seminorm that its norm induces on the coinvariant module $A_G$ by $\|\cdot \|_{A_G}$. The following is a necessary and sufficient condition for the non-vanishing of comparison maps in top degree which does not involve bounded cohomology.

\begin{theorem}
\label{TheoremCompBerstein}
Let $G$ be an oriented Poincar\'{e} duality group, let $n=\cd(G)$ and let $z$ be a generator of $H_n(G;\ZZ)$. Then:
$$\comp^n_\RR \neq 0 \qquad \Leftrightarrow \qquad \|b^n_{\RR} \cap z\|_{(J^n_G)_G} >0. $$
\end{theorem}
\begin{proof}

By Proposition \ref{PropRealBerstein} it holds that $b^n_{\RR} \neq 0 \in H^n(G;J_G^n)$. Since $G$ is an oriented Poincar\'{e} duality group, for each $\ZG$-module $A$ the cap product yields an isomorphism 
	$$P_A\colon  H^n(G;A) \to H_0(G;A)=A_G, \qquad P_A(u)=u \cap z.$$
	``$\Rightarrow$'': \enskip Let $u \in H^n_b(G;\RR)$ with $\comp^n_{\RR}(u)\neq 0$. By Theorem \ref{TheoremBoundedUniv}.b) there exists $\psi \in B_G(\II^n_G;\RR)$ with $u= \psi_*(\gamma_n)$. We derive from the commutativity of the left-hand diagram in \eqref{EqCompPushfor} and from Proposition \ref{PropClassRelations}.d) that
	$$ 0 \neq \comp^n_\RR(\psi_*(\gamma_n)) = \psi_*(\comp^n_{\II^n_G}(\gamma_n)) = \psi_* \circ (\vartheta_n)_*(b^n_{\RR})= f_*(b^n_{\RR}), $$
	where $f:= \psi \circ \vartheta_n \in B_G(J_G^n,\RR)$. Since $f$ is $G$-equivariant, it induces a map $\overline{f}\colon K_n \to \RR$, where $K_n:= (J^n_G)_G$, which is bounded with respect to the quotient seminorm on $K_n$.  By naturality of the Poincar\'{e} duality isomorphisms, we obtain that
	$$P_\RR(f_*(b^n_\RR)) = \overline{f}(P_{J^n_G}(b^n_\RR))= \overline{f}(b^n_\RR \cap z). $$
	Since $P_\RR(f_*(b^n_\RR)) \neq 0$, it follows that 
	$$0< \left|P_\RR(f_*(b^n_\RR))\right|=|\overline{f}(b^n_\RR \cap z)| \leq \|\overline{f}\|_\op \cdot \|b^n_\RR \cap z \|_{K_n}$$
	where $\|\overline{f}\|_{\op}$ is the operator seminorm of $\overline{f}$. Thus, we derive that $\|b^n_\RR\cap z\|_{K_n}>0$.  
	
``$\Leftarrow$'': \enskip Put $\widetilde{b}:= b^n_\RR \cap z\in K_n$, and let $q_n \colon J^n_G \to K_n$ be the projection. Since $\|\widetilde{b}\|_{K_n} > 0$, we can use the Hahn--Banach theorem with respect to the seminorm $\|\cdot \|_{K_n}$ to find $\varphi \in K_n'$ with $\varphi(\widetilde{b}) = \|\widetilde{b}\|_{K_n} > 0$. The commutativity of 	
\begin{equation*}
		\begin{tikzcd}
			H^n(G;J^n_G) \otimes H_n(G;\ZZ) \ar[d,"(q_n)_*\otimes \id_{H_n(G;\ZZ)}"] \ar[r,"\cap"] & K_n\ar[d,"\id_{K_n}"] \\
		H^n(G;K_n) \otimes H_n(G;\ZZ) \ar[r,"\cap"] & K_n.
		\end{tikzcd}
	\end{equation*}
yields $\widetilde{b} = (q_n)_*(b^n_\RR)\cap z = P_{K_n}((q_n)_*(b^n_\RR))$. We thus derive from the naturality of the Poincar\'e duality isomorphisms that 
$$0 \neq \varphi(\widetilde{b}) = \varphi(P_{K_n}((q_n)_*(b^n_\RR)) )= P_{\RR}((\varphi \circ q_n)_*(b^n_\RR)).$$
With $\psi := \varphi \circ q_n \in B_G(J^n_G,\RR)$, this shows that $\psi_*(b^n_\RR)\neq 0$. By Lemma \ref{LemmaExtendIk}.b) there exists $\Psi \in B_G(\II^n_G,\RR)$ with $\Psi\circ \vartheta_n = \psi$ and thus 
$$\Psi_*(\vartheta_n)_*(b^n_\RR) \neq 0.$$
By Proposition \ref{PropClassRelations}.d), this is equivalent to $\Psi_*(\comp_{\II^n_G}^n(\gamma_n)) \neq 0$ and finally from the commutativity of the left-hand square of \eqref{EqCompPushfor} that
$$\comp^n_\RR(\Psi_*(\gamma_n))= \Psi_*(\comp^n_{\II^n_G}(\gamma_n))\neq 0.$$
\end{proof}

The vanishing of the comparison map occurring in Theorem \ref{TheoremCompBerstein} is related to the vanishing of the simplicial volume of an aspherical manifold. 

\begin{corollary}
\label{CorSimpVolBerstein}
		Let $M$ be a closed connected orientable aspherical manifold with $G=\pi_1(M)$ and let $n=\dim M$. Then $$\|M\|>0 \qquad \Leftrightarrow \qquad \big\|b^n_{\mathbb R}\cap[M] \big\|_{(J^n_G)_G}>0 ,$$   
	where $\|M\|$ denotes the simplicial volume of $M$.
\end{corollary}
\begin{proof}
This follows from the standard identification of the singular cohomology of $M$ with the group cohomology of $G$, Theorem \ref{TheoremCompBerstein} and Gromov's duality principle, which says that $\|M\|>0$ if and only if $\comp^n_\RR \neq 0$, see e.g. \cite[Corollary 7.11]{FrigerioBook}.
\end{proof}

\section{$\FF$-categories and (bounded) cohomology}

\subsection{Relative Berstein classes and lower bounds for $\FF$-categories.} We begin by reviewing the notion of $\FF$-category. This concept was first introduced by  P. Capovilla, C. L\"oh and M. Moraschini in \cite{CLM} and is based on a construction of J. C. G\'omez-Larra\~naga, F. Gonz\'alez-Acu\~na and W. Heil from \cite{GLGAH} that we will discuss in the next section. 

\begin{definition}
\label{DefFFcategory}
	Let $X$ be a path-connected topological space, let $x_0$ be a basepoint with $\pi_1(X,x_0)=G$ and let $\FF$ be a family of subgroups of $G$. The \emph{$\FF$-category of $X$}, denoted by $\cat_{\FF}(X)$,  is given as the minimal value of $n\in \NN_0$ for which there exists an open cover $U_0,U_1,\dots,U_n$ of $X$, such that $$\mathrm{im} \left[\pi_1(i_{U_j})\colon \pi_1(U_j,u_j)\to \pi_1(X,u_j)\right] \in \FF$$ for all $u_j \in U_j$ and all $j \in \{0,1,\dots,n\}$. Here, for each $A \subset X$ we let $i_A\colon A \hookrightarrow X$ denote the inclusion. 
\end{definition}

Note that $\im( \pi_1(i_{U_j}))$ is not given as a subgroup of $G=\pi_1(X,x_0)$, but, since $X$ is path-connected, as a subgroup of a conjugate of $G$. Since $\FF$ is closed under conjugation, the above notion is well-defined.  It follows immediately from the definition, see also \cite[Remark 2.8]{CLM}, that
\begin{equation}
\label{EqcatFFupper}
\cat_\FF(X) \leq  \cat(X),
\end{equation}
where $\cat(X)$ denotes the Lusternik-Schnirelmann category of $X$. If $X$ is aspherical with $G=\pi_1(X)$, this yields that, together with the Eilenberg--Ganea theorem from \cite{EG57}, 
$$\cat_\FF(X) \leq \cd(G).$$
We want to apply the cohomological techniques developed so far in order to obtain lower bounds about $\FF$-categories of aspherical spaces via bounded cohomology. In order to proceed further, let us first recall a topological notion. 

\begin{definition}
	Let $G$ be a group and $\FF$ a family of subgroups of $G$. A \textit{classifying space of $G$ with respect to $\FF$} is a $G$-CW complex $E_\FF G$ satisfying the following conditions:
	\begin{itemize}
		\item every isotropy group of $E_{\mathcal{F}}G$ belongs to $\mathcal{F}$,
		\item for any $G$-CW complex $X$ with all isotropy groups in $\mathcal{F}$ there exists a $G$-equivariant continuous map $X \to E_\FF G$,  which is unique up to $G$-equivariant homotopy.
	\end{itemize}
\end{definition} 
It is immediate from the definition that any two models of $E_{\mathcal{F}}G$ are $G$-homotopy equivalent and that $EG := E_{\{1\}}G$ recovers the usual total space of the universal bundle of $G$. Furthermore, there is a $G$-equivariant continuous map from $EG \to E_\FF G$ which is unique up to $G$-homotopy. \medskip 

Many different models of $E_\FF G$ can be constructed, see \cite[Sections 1-3]{Luck}. For our purposes  we are particularly interested in the Arciniega--Cisneros model, used in \cite{ArNev} as a simplified version of the better known Farrell--Jones model from \cite{FaJo93}. 

In what follows, we denote the cellular chain complex of a CW complex $Y$ with $\ZZ$-coefficients by $$C^{\cell}_k(Y):= H_k(Y_k,Y_{k-1}), \quad d^\cell_k\colon C^\cell_k(Y)\to C^\cell_{k-1}(Y)\quad \forall k \in \NN_0,$$
where $H_*(\cdot, \cdot)$ denotes relative singular homology groups with $\ZZ$-coefficients. Given a family of subgroups $\FF$ of $G$, denote by $\EE \FF$ the semi-simplicial set whose $n$-simplices are the ordered $(n+1)$-tuples of elements of $G/\FF$ and whose face operators are given by 
\begin{align*}
d_i(x_0, \cdots, x_n) &= (x_0, \dots, x_{i-1}, x_{i+1}, \dots, x_n).
\end{align*}
The action of $g \in G$ on an $n$-simplex $(x_0, \cdots, x_n)$ gives the simplex $(gx_0, \cdots, gx_n)$. The \emph{Arciniega--Cisneros model for $E_\FF G$} is the geometric realization $\vert \EE \FF \vert$ of the semi-simplicial set $\EE \FF$. It is shown in \cite[Proposition 4.16]{ArNev} that this is indeed a model for $E_\FF G$.  Viewing it as a $\Delta$-complex, since $\vert \EE \FF \vert$ has exactly one $n$-cell for each $(n+1)$-tuple of $G/\FF$, it is immediate to see that the cellular chain complex $C^\cell_*(E_\FF G)$ corresponds to the $\FF$-relative standard resolution from \eqref{EqFFrelstandard}. By Lemma \ref{fundamentalrelhom}, passing to cellular cohomology immediately shows for each $\ZG$-module $A$ that 
\begin{equation}
\label{EqIsoAdamsonEFG}
H^k(\Hom_{\ZG}(C^\cell_*(E_\FF G),A)) \cong H^k([G\colon \FF];A) \qquad \forall k \in \NN_0.
\end{equation} 
Let $\varrho_\FF \colon EG \rightarrow E_\FF G$ be a continuous $G$-map, which, as noted above, is unique up to $G$-homotopy. As a consequence of the definition, for any $H \in \FF$ such a map factors up to $G$-homotopy as 
\begin{equation} \label{FactorThetaMap} \begin{tikzcd}
	EG \ar[r, "\varrho_{\langle H \rangle}"] \ar[rr, bend right=20, "\varrho_\FF"'] & E_{\langle H \rangle} G \ar[r] & E_\FF G.
\end{tikzcd}  \end{equation} 
where both $EG$ and $E_{\langle H \rangle}G$ are also considered under the Arciniega-Cisneros model for the subgroup families $\{1\}$ and $\langle H \rangle$, respectively. Observe that $\varrho_\FF$ is then the realization of a semi-simplicial map between the corresponding geometric realizations of the models, and is thus a cellular map. 

Let $X$ be a connected CW complex with $\pi_1(X) = G$, and consider the classifying map $$\cc_X \colon X \to BG\simeq \quot{EG}{G}$$ associated to its universal cover $\widetilde{X}$. Since $\widetilde{X}$ is a free $G$-space, we have a $G$-equivariant map $\mathfrak{u} \colon \widetilde{X} \to E_\FF G$, unique up to $G$-homotopy. Considering an equivariant lift of the classifying map $\widetilde{\mathfrak{c}}_X \colon \widetilde{X} \to E G$, we observe that the composition 
$$ \widetilde{X} \xrightarrow{\widetilde{\cc}_X} EG \xrightarrow{\varrho_\FF} E_\FF G $$ 
indicates, by the universal property of $E_\FF G$, that $\mathfrak{u}$ and $\varrho_\FF\circ \widetilde{\cc}_X $ are $G$-homotopic maps. We have induced chain homomorphisms 
$$ C^\cell_*(\widetilde{X}) \xrightarrow{(\cc_X)_*} \mathcal{E}^G_* \xrightarrow{\overline{\varrho^*_\FF}} \mathcal{E}^{\FF}_* $$ 
where $(\cc_X)_*\colon C_*^\cell(\widetilde{X}) \to C_*^\cell(E G)$ is the map induced by $\widetilde{\cc}_X$.
\medskip 

The following proposition provides a lower bound on $\FF$-category in terms of Berstein classes relative to $\FF$, which applies to spaces that are not necessarily aspherical. One main ingredient of its proof is the following characterization of $\FF$-categories given in \cite[Proposition 7.5]{CLM}: Let $X$ be a connected CW complex with $G=\pi_1(X)$ and let $\widetilde{X}$ be its universal cover, seen as a $G$-CW complex. Then $\cat_\FF(X)$ coincides with the smallest $k \in \NN$, for which a continuous $G$-map $\widetilde{X} \to E_\FF G $ is $G$-homotopic to a map whose image lies in $(E_\FF G)_k$.

\begin{prop}\label{prop:lower-bound-catF}
	Let $\FF$ be a family of subgroups of $G$ and let $X$ be a connected CW complex  with $\pi_1(X)=G$. 
	\begin{enumerate}[a)]
	\item  Then $$\cat_\FF(X) \geq  \sup \{k \in \NN_0 \mid \cc_X^*(\omega_\FF^k) \neq 0\}.$$
	\item If $X$ is aspherical, then $$\cat_\FF(X) \geq  \sup \{k \in \NN_0 \mid \omega_\FF^k \neq 0\}.$$
	\end{enumerate}
\end{prop}
\begin{proof}
\begin{enumerate}[a)]
\item Let $k \in \NN$ and assume  $\cat_{\FF}(X)<k$. We want to show that this implies $\cc_X^*(\omega_\FF^k) = 0$. Consider the Arciniega--Cisneros model of $E_\FF G$. By the discussion above the map $\varrho_\FF$ induces a chain homomorphism between the corresponding standard resolutions \begin{equation}
	\label{Eqvarrho}	
	\overline{\varrho}^n_\FF \colon \mathcal{E}^G_{n} \longrightarrow \mathcal{E}^\FF_{n}
\end{equation} following the notation from Example \ref{ExampleStandard}. For any $H \in \FF$ we can choose $\varrho_\FF$ as the realization of the $G$-map $G \to G/\FF$, $g \mapsto gH$. The induced chain homomorphism $\overline{\varrho^{*}_{\FF}}$ is then given by $\ZZ$-linearly extending 
$$ \overline{\varrho}^{s}_{\FF} \colon \ZZ[G^{s+1}] \rightarrow \ZZ[(G/\FF)^{s+1}], \qquad \overline{\varrho^{s}_{\FF}}(g_0, g_1, \cdots, g_s) = (g_0 H, g_1H, \cdots, g_s H) \qquad \forall s \in \NN_0.$$ 
For any $\ZG$-module $A$, the chain map $\overline{\varrho}^*_{\FF}$ induces a map 
\begin{equation}
	\label{Eqvarrho2}
	\varrho_\FF^s \colon H^s([G:\FF];A) \rightarrow H^s(G;A) \qquad \forall s \in \NN_0. 
\end{equation}
By composing with the chain maps \eqref{chainmapsresolutions} and since those maps determine chain homotopy equivalences between the respective resolutions, and thus they induce isomorphisms at cohomological level, it is straightforward to check that 
\begin{equation}
\label{EqProp53comm}
 \psi_{\mathcal{F}}^{*} \circ \overline{\varrho^{*}_{\FF}} =  \sigma_{H,*} \circ \psi^{*}_G 
 \end{equation}
  and thus $(\varrho_\FF)^* = \rho^*_\FF. $ Consequently, for any $\ZZ[G]$-module $A$ we have a homomorphism induced at the cohomological level 
\begin{equation} \label{umap}
	\uu^* = \cc_X^* \rho^*_\FF \colon H^*([G : \FF];A) \to H^*(X;A).
\end{equation} 
	By \cite[Proposition 7.5]{CLM} the map $\uu$ factors up to $G$-homotopy through the $(k-1)$-skeleton $(E_{\FF} G)_{k-1}$, i.e. we have a commutative diagram up to $G$-homotopy of the form $$\begin{tikzcd} 
	\widetilde{X} \arrow[r, "\widetilde{\cc}_X"'] \arrow[rd] \arrow[rr, bend left=20, "\uu"] & EG \arrow[r, "\varrho_\FF"'] & E_\FF G \\ 
	& (E_{\FF}G)_{k-1}. \arrow[ru]
\end{tikzcd}$$ which in turn induces a commutative diagram of chain complexes  \begin{equation} \label{lower:diagramchainmaps}
\begin{tikzcd}
	C^\cell_*(\widetilde{X}) \arrow[r, "\widetilde{\cc}_X"'] \arrow[rd] \arrow[rr, bend left=20, "\overline{\uu}_*"] & \mathcal{E}^G_{*} \arrow[r, "\overline{\varrho}^{*}_{\FF}"'] & \mathcal{E}^{\FF}_{*} \\ 
	& C^\cell_*((E_{\FF}G)_{k-1}). \arrow[ru]
\end{tikzcd}
\end{equation} Since the chain complex $C^\cell_*((E_{\FF}G)_{k-1})$ coincides up to degree $k-1$ with $\EE^{\FF}_{*}$, and it is trivial for any $q \geq k$, we have that  $ \uu^* = \cc_X^* \rho^*_\FF \colon H^q([G:\FF];I^q_{\FF}) \to H^q(X;I^q_{\FF}) $ is trivial for $q \geq k$. Hence, $\cc^*_X(\omega^q_{\FF}) = 0$ for all $q \geq k$, so that $$\cat_\FF(X) \geq \sup \{q \in \NN_0 \mid \cc^*_X(\omega^q_{\FF}) \neq 0\}.$$ 

\item This follows immediately from a) since in the aspherical case, $\cc_X$ is a homotopy equivalence, hence induces isomorphisms in cohomology.
\end{enumerate} \end{proof}

\subsection{Maximality criteria for $\FF$-categories}

We will show that the obstruction class from the proof of \cite[Proposition 7.5]{CLM} can be described explicitly, generalizing \cite[Theorem 2.7]{BCE22} and its proof, which deals with a family generated by a single subgroup. The obstruction class is, indeed, an image of a power of a relative Berstein class under a suitable change of coefficient modules. Such a description leads to a maximality criterion for $\cat_\FF$, given in terms of images of powers of relative Berstein classes under pullbacks of classifying maps.   

\begin{theorem}\label{thm:CostaFarberCatF}
Let $X$ be a connected finite-dimensional CW complex with $\pi_1(X)=G$ and $\dim(X) \geq 3$. Let $\mathfrak{c}_X\colon X \to BG$ be a classifying map of the universal cover of $X$. Then
	$$\cat_\FF(X) = \dim(X) \qquad \Leftrightarrow \quad \mathfrak{c}_X^*(\omega_\FF^{\dim(X)}) \neq 0 \in H^{\dim(X)}(X;\IF^{\dim(X)}).$$
\end{theorem}
\begin{proof}
 The implication ``$\Leftarrow$'' is an immediate consequence of Proposition \ref{prop:lower-bound-catF} a) and of the elementary inequalities $\cat_\FF(X)\leq \cat(X) \leq \dim X$.   

Let $k:=\dim(X)$. In order to show the reverse implication, suppose $\cc^*_X(\omega^k_\FF) = 0$. 
Let again $p\colon \widetilde{X} \to X$ be the universal cover of $X$, given as a $G$-CW complex, let $E_\FF G$ be the Arciniega--Cisneros model for a classifying space of $G$ with respect to $\FF$ and let $\varrho_\FF\colon EG \to E_\FF G$ be a continuous $G$-map. Let $\widetilde{\cc}_X\colon \widetilde{X} \to EG$ be a $G$-equivariant lift of $\cc_X \circ p$. 

The map $\mathfrak{u}:= \varrho_\FF \circ \widetilde{\cc}_X \colon \widetilde{X} \to E_\FF G$ is a continuous $G$-map. We further assume w.l.o.g. that $\cc_X$ is cellular, since this can be achieved by equivariant cellular approximation. Consider the $G$-equivariant extension problem
\begin{equation*}\begin{tikzcd}
		\widetilde{X}_{k-1} \ar[d, hook'] \ar[rrd,"\uu_{k-1}"] \\
		\widetilde{X}_k \ar[rr,dashed] & & (E_{\FF}G)_{k-1}
	\end{tikzcd}
\end{equation*}  where $\uu_{k-1}:= \uu|_{\widetilde{X}_{k-1}}$. By obstruction-theoretic arguments it has a solution provided the obstruction cocycle of this extension problem is cohomologous to zero in $H^k(X;\pi_{k-1}((E_\FF G)_{k-1}))$. We want to find an explicit description of this cocycle. 	We consider the Hurewicz homomorphisms 
\begin{align*}
	\overline{h}_k \colon \pi_k(\widetilde{X}_k, \widetilde{X}_{k-1}) &\stackrel{\cong}{\longrightarrow} H_k(\widetilde{X}_k, \widetilde{X}_{k-1}),\\ 
	h_{k-1} \colon \pi_{k-1}(\widetilde{X}_{k-1}) &\longrightarrow H_{k-1}(\widetilde{X}_{k-1}),\\
	h_{k-1}^{\FF} \colon \pi_{k-1}\big((E_\FF G)_{k-1}\big) &\stackrel{\cong}{\longrightarrow} H_{k-1}\big((E_\FF G)_{k-1}\big).
\end{align*}
By relative cellular approximation the pair $(\widetilde{X}_k, \widetilde{X}_{k-1})$ is $(k-1)$-connected, and since $k\geq 3$ and the skeleton $(E_{\FF}G)_{k-1}$ is $(k-2)$-connected, both $\overline{h}_k$ and $h_{k-1}^\FF$ are isomorphisms. 
Let 
$$\partial_k\colon \pi_k(\widetilde{X}_k,\widetilde{X}_{k-1}) \to \pi_{k-1}(\widetilde{X}_{k-1})$$ 
be the boundary operator of the long exact sequence of homotopy groups of $(\widetilde{X}_k, \widetilde{X}_{k-1})$. The obstruction cocycle representing the extension is then given by the cellular cocycle 
$$c^k(\uu_{k-1})\in \Hom_{\ZG}(C_k^\cell(\widetilde{X}), \pi_{k-1}((E_\FF G)_{k-1})),\qquad  c^k(\uu_{k-1}) = \pi_{k-1}(\uu_{k-1}) \circ \partial_k \circ \overline{h}_k^{-1}, $$ 
where $$\pi_{k-1}(\uu_{k-1}) \colon \pi_{k-1}(\widetilde{X}_{k-1}) \rightarrow \pi_{k-1}((E_\FF G)_{k-1})$$  is the map induced by $\uu_{k-1}$.  Observe that the cocycle can be described at cellular level as follows: if $\phi = \{\phi_j\}_{j \in J}$ is the collection of characteristic maps of the cell attachment in $(\widetilde{X}_k, \widetilde{X}_{k-1})$ and $e_j$ is the basis element representing $\phi_j$ in the cellular chain complex, then $c^k(\uu_{k-1})(e_j) = [\uu_{k-1} \circ \phi_{j}|_{S^{k-1}}]$. 

We claim that this obstruction cocycle is the image of a cocycle representing $\cc_X^*(\omega^k_\FF)$ under a change-of-coefficients homomorphism. To show this, let $\psi^*_{\mathcal{F}}$ be the chain map from \eqref{chainmapsresolutions}, let $\varphi^*_{\mathcal{F}}$ be a chain homotopy inverse to $\psi^*_{\FF}$ and let $ P_* \colon \EE^{\FF}_* \longrightarrow \EE^{\FF}_{*+1}$ be a chain homotopy with 
\begin{equation}
	\label{EqChainhom}
	\id_{\EE^{\FF}_n} - \varphi^n_{\FF} \circ \psi^n_{\FF} =  b^\FF_{n+1} \circ P_n + P_{n-1} \circ b^\FF_n \qquad \forall n \in \NN_0, 
\end{equation} 
where we adapt the notation from Example \ref{ExampleStandard}. Let $$ j_\FF \colon H_{k-1}((E_\FF G)_{k-1}) \hookrightarrow C^\cell_{k-1}(E_\FF G) \cong \mathcal{E}^\FF_{k-1}, $$ be the inclusion into the cellular chains, with the latter equivalence being the identification provided by the Arciniega--Cisneros model. Since $(E_\FF G)_{k-1}$ has no cells of dimension greater than or equal to $k$, it follows that 
\begin{equation} \label{imagejF}
H_{k-1}((E_\FF G)_{k-1}) \cong \im j_\FF = \ker \left( b^\FF_{k-1} \colon \EE^\FF_{k-1} \to \EE^\FF_{k-2} \right).
\end{equation} Denote by $j^{-1}_\FF$ the inverse over the image of $j_\FF$. If we further identify the standard resolution of $G$ with the cellular chain complex of $EG$, i.e. we let $C^\cell_k(EG) \cong \EE^G_k$ for each $k \in \NN_0$, then $\cc_X^*(\omega^k_{\FF})$ is by definition given as the cohomology class of the cocycle 
$$ \overline{c}(\cc_X^*(\omega^k_\FF)) = g_k^{\FF} \circ \sigma_{H,k} \circ \psi^k_{G} \circ (\cc_X)_k, $$ 
where $(\cc_X)_*\colon C_*^\cell(\widetilde{X}) \to C_*^\cell(E G)  $ is the map induced by $\widetilde{\cc}_X$ and where $H \in \FF$ is chosen arbitrarily. We consider the $\ZG$-map
$$ \tau \colon I^k_\FF \to \pi_{k-1}((E_\FF G)_{k-1}), \qquad \tau := (h^{\FF}_{k-1})^{-1} \circ j^{-1}_\FF \circ \varphi^{k-1}_{\FF} \circ f^{\FF}_{k-1}, $$ 
where $f^{\FF}_{k-1}$ is given as in \eqref{EqBasicSESk}. This homomorphism is well defined, since $$ b^\FF_{k-1} \circ \varphi^{k-1}_\FF \circ f^\FF_{k-1} = \varphi^{k-2}_\FF \circ d^\FF_{k-1} \circ f^\FF_{k-1} = \varphi^{k-2}_\FF \circ f^\FF_{k-2} \circ g^\FF_{k-1} \circ f^\FF_{k-1} = 0. $$ The coefficient homomorphism $\tau$ induces a cochain homomorphism $$ \overline{\tau} \colon \Hom_{\ZG}(C_k^\cell(\widetilde{X});I^k_\FF) \to \Hom_{\ZG}(C_k^\cell(\widetilde{X});\pi_{k-1}((E_\FF G)_{k-1})). $$ Let $\overline{\varrho}^k_\FF$ be given as in \eqref{Eqvarrho}. Using \eqref{EqProp53comm}, we compute that $\tau_*(\cc_X^*(\omega_\FF^k))$ is represented by 
\begin{align*}
\overline{\tau}(\overline{c}(\cc_X^*(\omega^k_\FF))) &= \left((h^{\FF}_{k-1})^{-1} \circ j^{-1}_\FF \circ \varphi^{k-1}_\FF \circ f^{\FF}_{k-1}\right) \circ \left( g^{\FF}_k \circ \sigma_{H,k} \circ \psi^k_G \circ (\cc_X)_k\right) \\
& = (h^{\FF}_{k-1})^{-1} \circ j^{-1}_\FF \circ \varphi^{k-1}_\FF \circ d_k^{\FF} \circ \psi^k_\FF \circ \overline{\varrho}^k_\FF\circ (\cc_X)_k \\
& = (h^\FF_{k-1})^{-1} \circ j^{-1}_\FF \circ b^\FF_k \circ \varphi^k_\FF \circ \psi^k_\FF \circ \overline{\varrho}^k_\FF \circ (\cc_X)_k .
\end{align*} 
Applying \eqref{EqChainhom} we observe 
\begin{align*}
b^\FF_k \circ \varphi^k_\FF \circ \psi^k_\FF \circ \overline{\varrho}^k_\FF \circ (\cc_X)_k & = b^\FF_k \circ \overline{\varrho}^k_\FF \circ (\cc_X)_k - b^\FF_k \circ b^\FF_{k+1} \circ P_k \circ \overline{\varrho}^k_\FF \circ (\cc_X)_k \\ & - b^\FF_k \circ P_{k-1} \circ b^\FF_k \circ \overline{\varrho}^k_\FF \circ (\cc_X)_k \\ 
& = (\overline{\varrho}^{k-1}_\FF \circ (\cc_X)_{k-1}) \circ  d^\cell_k - b^\FF_k \circ P_{k-1} \circ (\overline{\varrho}^{k-1}_\FF \circ (\cc_X)_{k-1}) \circ d^\cell_k
\end{align*} 
Defining 
$$ \alpha := (h^\FF_{k-1})^{-1} \circ j^{-1}_\FF \circ b^\FF_k \circ P_{k-1} \circ (\overline{\varrho}^{k-1}_\FF \circ (\cc_X)_{k-1}) $$ 
we have that $$ \overline{\tau}(\overline{c}(\cc_X^*(\omega^k_\FF))) = (h^\FF_{k-1})^{-1} \circ j^{-1}_\FF \circ (\overline{\uu}_{k-1}\circ d^\cell_k) - \alpha \circ d^\cell_k $$ with the second term being a coboundary. Observe that the first term is not in general a coboundary. Indeed, by \eqref{imagejF}, the map $j_\FF^{-1}$ is defined on $(k-1)$-cycles, while $\overline{\uu}_{k-1}$ is a chain map that can take chains to arbitrary chains, so its composition with $j_\FF^{-1}$ is only well defined after precomposition with $d^\cell_k$. Thus,  $$ \tau_*(\cc_X^*(\omega^k_\FF)) =  \left[ \overline{\tau}(\overline{c}(\cc_X^*(\omega^k_\FF))) \right] = \left[ (h^\FF_{k-1})^{-1} \circ j^{-1}_\FF \circ (\overline{\uu}_{k-1}\circ d^\cell_k) \right]. $$ 
We claim that this class coincides with the obstruction class $[c^k(\uu_{k-1})]$. In order to see that, first let 
$$ j_X \colon H_{k-1}(\widetilde{X}_{k-1}) \to C^\cell_{k-1}(\widetilde{X}) $$ denote the inclusion into the cellular chains. Since $\widetilde{X}_{k-1}$ has no cells of dimension greater than or equal to $k$, we see that $j_X$ is injective with $\im j_X = \ker d^\cell_{k-1}$. Naturality of the Hurewicz homomorphisms with respect to the boundary operator in the long exact sequence of homotopy groups gives us the equality 
$$ j_X \circ h_{k-1} \circ \partial_k \circ \overline{h}^{-1}_k = d^\cell_k .$$  
Denote by 
$$ (\uu_{k-1})_* \colon H_{k-1}(\widetilde{X}_{k-1}) \to H_{k-1}((E_\FF G)_{k-1}) $$ the map induced at the homology level of the $(k-1)$-skeleta. Once again, by naturality of the Hurewicz homomorphisms, we obtain 
$$ j_\FF \circ (\uu_{k-1})_* = \overline{\uu}_{k-1} \circ j_X  \qquad \mbox{and} \qquad h^\FF_{k-1} \circ \pi_{k-1}(\uu_{k-1}) = (\uu_{k-1})_* \circ h_{k-1},$$ 
and so it follows that
\begin{align*}
j_\FF \circ h^\FF_{k-1} \circ c^k(\uu_{k-1}) & = j_\FF \circ h^\FF_{k-1} \circ \pi_{k-1}(\uu_{k-1}) \circ \partial_k \circ \overline{h}_k^{-1} \\
&= j_\FF \circ (\uu_{k-1})_* \circ h_{k-1} \circ \partial_k \circ \overline{h}_k^{-1} \\
& = \overline{\uu}_{k-1} \circ j_X \circ h_{k-1} \circ \partial_k \circ \overline{h}_k^{-1}\\
& = \overline{\uu}_{k-1} \circ d^\cell_k. 
\end{align*} 
Since $j_\FF$ is a monomorphism and $h^\FF_{k-1}$ is an isomorphism, this shows  that 
$$ \left[c^k(\uu_{k-1}) \right] = \tau^*(\cc_X^*(\omega^k_\FF)) . $$ 
Since by assumption $\cc^*_X(\omega^k_\FF) = 0$, the above obstruction class vanishes, and so by standard obstruction-theoretic techniques the cellular map $\uu_{k-1}$ can be extended after modification on the $(k-1)$-skeleton relative to the $(k-2)$-skeleton to a $G$-equivariant cellular map $\uu_k \colon \widetilde{X}_k \to (E_\FF G)_{k-1}$, see for example \cite{Bredon67}. Composing $\uu_k$ with the inclusion $\iota\colon (E_\FF G)_{k-1} \to E_\FF G$ yields a $G$-equivariant map $\iota \circ \uu_k\colon \widetilde{X} \to E_\FF G$, which is $G$-homotopic to $\uu$ by the universal property of $E_\FF G$. Since $\dim X = k$, we derive from \cite[Proposition 7.5]{CLM} that $\cat_\FF(X) \leq k-1$, which finishes the proof. \end{proof}		

\begin{corollary} \label{cor:CostaFarberCatF}
Assume $G$ is torsion-free with $\cd(G) \geq 3$ and let $X$ be an aspherical CW complex with $\pi_1(X) = G$ and $\dim X=\cd(G)$. Then $$\cat_\FF(X) = \cd(G) \qquad \Leftrightarrow \quad \omega_\FF^{\cd(G)} \neq 0 \in H^{\cd(G)}(G;\IF^{\cd(G)}).$$ 
\end{corollary}
\begin{proof}
 Since $X$ is aspherical, the classifying map $\cc_X \colon X \to BG$ induces an isomorphism at cohomology level, hence $\cc^*_X(\omega^k_\FF) = \omega^k_\FF$. The claim then follows from Theorem \ref{thm:CostaFarberCatF}. 
\end{proof}

\subsection{Bounded cohomology and lower bounds for $\FF$-categories.}
	
 We are going to use Johnson classes and the bound from Proposition \ref{prop:lower-bound-catF} to establish lower bounds on $\cat_\FF(X)$ for aspherical $X$ given in terms of bounded cohomology. 

\begin{theorem}
	\label{TheoremBoundedLB}
	Let $\FF$ be a family of subgroups of $G$ and let $X$ be an aspherical CW complex with $\pi_1(X)=G$. Then 
	$$\cat_\FF(X) \geq \sup \Big\{k \in \NN \ \Big|\ \comp_{\II_\FF^k}^k(\gamma_{\FF,k})\neq 0\Big\}.$$
\end{theorem}
\begin{proof}
By Proposition \ref{PropClassRelations}.c), it holds that $\comp_{\II_\FF^k}^k(\gamma_{\FF,k})=(\vartheta_{\FF,k})_*(\omega_{\FF,\RR}^k)$.
	Thus, if $\comp_{\II_\FF^k}^k(\gamma_{\FF,k})\neq 0$, then it follows that $\omega^k_{\FF,\RR}\neq 0$, which yields  $\omega^k_\FF \neq 0$ by Lemma \ref{LemmaRealBerstein}. The claim then follows directly from Proposition \ref{prop:lower-bound-catF}.
\end{proof}

The following terminology was introduced in a topological context in \cite[Definition 5.1]{LoehSauer}. We give below an algebraically defined stronger version, which allows one to consider arbitrary dual Banach $G$-module coefficients instead of just trivial ones.

\begin{definition} \label{DefHbadmiss}
	A family $\FF$ of subgroups of $G$ is \emph{$H^*_b$-admissible} if the map $$\rho^k_{b,\FF}\colon H^k_b([G:\FF];V') \to H^k_b(G;V')$$ is surjective for each $k \in \NN$ and all normed $G$-modules $V$. 
\end{definition}

For all $H^*_b$-admissible families we can derive a lower bound on $\cat_\FF$ of an aspherical space in terms of Johnson classes from Theorem \ref{TheoremBoundedLB}, which is independent of the actual family. Note that the universality property of Adamson--Johnson classes from Theorem \ref{TheoremBoundedUniv}.a) plays a crucial role in its proof.

\begin{theorem}
	\label{TheoremcatFFlowergamma}
	Let $\FF$ be an $H^*_b$-admissible family of subgroups of $G$ and let $X$ be an aspherical CW complex with $\pi_1(X)=G$. Then 
	$$\cat_\FF(X) \geq \sup \{k \in\NN \mid \comp^k_{\II_G^k}(\gamma_k) \neq 0\}.$$
\end{theorem}
\begin{proof}
	In the light of Theorem \ref{TheoremBoundedLB}, it suffices to show that if $\FF$ is $H^*_b$-admissible, then it holds for all $k \in \NN$ that
	$$\comp^k_{\II_\FF^k}(\gamma_{\FF,k}) \neq 0 \quad \Leftrightarrow \quad \comp_{\II_G^k}^k(\gamma_k) \neq 0.$$
	
	Fix $k \in \NN$ and assume first that $\comp^k_{\II_\FF^k}(\gamma_{\FF,k}) \neq 0$.  Since  $\gamma_{\FF,k} = (\psi_{\FF,k})_*(\gamma_k)$ by Proposition \ref{PropClassRelations}.a), it follows from the commutativity of the left-hand diagram of \eqref{EqCompPushfor} that
	$$0 \neq \comp^k_{\II_\FF^k}(\gamma_{\FF,k}) = \comp^k_{\II_\FF^k}((\psi_{\FF,k})_*(\gamma_k)) = (\psi_{\FF,k})_*(\comp^k_{\II^k_G}(\gamma_k)), $$
	which yields $\comp^k_{\II^k_G}(\gamma_k) \neq 0$. Conversely, assume that $\comp^k_{\II^k_G}(\gamma_k) \neq 0$. Since $\FF$ is $H^*_b$-admissible, there exists $u \in H^k_b([G:\FF];\II_G^k)$ with $\rho^k_{b,\FF}(u)=\gamma_k$. 
	By Theorem \ref{TheoremBoundedUniv}.a), there exists $F \in B_G(\II_\FF^k,\II_G^k)$ with $u= F_*(\Gamma_{\FF,k})$. Since $\rho^k_{b,\FF}\circ F_*=F_*\circ \rho^k_{b,\FF}$ , we compute that
	$$ \gamma_k=\rho^k_{b,\FF}(u) =\rho^k_{b,\FF}(F_*(\Gamma_{\FF,k})) = F_*(\rho^k_{b,\FF}(\Gamma_{\FF,k})) = F_*(\gamma_{\FF,k})$$
	and thus, by the commutativity of the left-hand diagram of \eqref{EqCompPushfor},
	$$0 \neq \comp_{\II^k_G}^k(\gamma_k) = \comp^k_{\II^k_G}(F_*(\gamma_{\FF,k}))= F_*(\comp_{\II^k_\FF}^k(\gamma_{\FF,k})), $$
	which yields $\comp_{\II^k_\FF}^k(\gamma_{\FF,k})\neq 0$ and thereby concludes the proof.
\end{proof}

Together with the $H^*_b$-admissibility condition, the rigidity imposed by the universality of both the Johnson  and the Adamson canonical classes yields that the non-vanishing of a comparison map for a dual coefficient module has direct consequences for the powers of the relative Berstein classes. 

\begin{corollary} \label{cor:dualmodulescontrolJohnsonadmissibility}
Let $\FF$ be an $H^*_b$-admissible family of subgroups of $G$, let $V$ be a normed $G$-module and let $u \in H^k_b(G;V')$ for some $k \in \NN$. If $\comp^k_{V'}(u) \neq 0$ then \begin{enumerate}[a)]
\item $ \comp^k_{\II^k_G}(\gamma_k) \neq 0 $.
\item $\omega^k_{\FF,\RR} \neq 0$ and $\omega^k_\FF \neq 0.$
\end{enumerate} 
\end{corollary}
\begin{proof}
\begin{enumerate}[(a)]
\item By Theorem \ref{TheoremBoundedUniv} there exists $T \in B_G(\II^k_{V'})$ satisfying $u = T_*(\gamma_k)$. If $\comp^k_{V'}(u) \neq 0$, the commutativity of the left-hand square of \eqref{EqCompPushfor} implies then that $\comp^k_{\II^k_G}(\gamma_k) \neq 0$.
\item  By the $H^*_b$-admissibility of the family $\FF$, there exists a class $v \in H^k_b([G:\FF];V')$ such that $\rho^k_{b,\FF}(v) = u$. The universality of the real Adamson canonical class from \eqref{PropUnivReal} gives a $\RG$-homomorphism $\psi \in \Hom_{\RG}(J^k_\FF,V')$ such that $$ \comp^k_{\FF,V'}(v) = \psi_*(\phi^k_{\FF,\RR}). $$ The commutativity of \eqref{EqDiagComp} and naturality with respect to change of coefficients give the equalities $$ 0 \neq \comp^k_{V'}(u) = \rho^k_\FF \circ \psi_*(\phi^k_{\FF,\RR}) = \psi_*(\omega^k_{\FF,\RR}) $$ hence $\omega^k_{\FF,\RR} \neq 0$. The non-vanishing of $\omega^k_\FF$ then follows from Lemma \ref{LemmaRealBerstein}.
\end{enumerate}	
\end{proof}

\begin{remark}
While the proofs of Theorems \ref{TheoremBoundedLB} and \ref{TheoremcatFFlowergamma} show that the lower bounds obtained therein cannot improve the lower bound from Proposition \ref{prop:lower-bound-catF}, the virtue of Theorem \ref{TheoremcatFFlowergamma} is that the lower bound by non-vanishing of $\comp_{\II^k_G}(\gamma_k)$ is a bound which is independent of the actual family $\FF$, as long as $\FF$ is $H^*_b$-admissible.
\end{remark}

We can use Proposition \ref{prop:lower-bound-catF} and Gromov's Mapping Theorem from \cite{Gro82} to recover a generalization of Gromov's Vanishing Theorem from amenable categories to $\FF$-categories for $H^*_b$-admissible families of subgroups $\FF$, which was first shown by C. L\"oh and R. Sauer in \cite[Theorem 5.3]{LoehSauer}. While L\"oh and Sauer use topological arguments and the amenable nerve theorem, we give a proof using the universality of Johnson classes and our previous results. 

Before we proceed, we recall the definition of bounded cohomology of topological spaces, see e.g. \cite[Chapter 5]{FrigerioBook} for a detailed exposition. Let $X$ be a path-connected topological space with universal cover $\widetilde{X}$, and let $C_k(\widetilde{X};\RR)$ denote the singular $k$-chains in $\widetilde{X}$ with real coefficients for each $k \in \NN_0$. Then $C_k(\widetilde{X};\RR)$ is a chain complex of normed $\pi_1(X)$-modules with respect to the $\ell^1$-norm. For any Banach $\pi_1(X)$-module $V$ and $k \in \NN_0$ we define $C^k_b(X;V) := B_{\pi_1(X)}(C_k(\widetilde{X},\RR),V)$. The bounded cohomology of $X$ with coefficients in $V$ is then defined as $$H^*_b(X;V) := H^*(C^*_b(X;V))$$ with respect to the singular codifferential. In analogy with the setting for discrete groups, the map induced by the inclusion $C^*_b(X;V) \hookrightarrow C^*(X;V)$ in cohomology is called \emph{comparison map} and denoted by 
$$ \comp^k_V \colon H^k_b(X;V) \longrightarrow H^k(X;V), \quad k \in \NN_0. $$
 Indeed, the bounded cohomology of a group can be reformulated as the bounded cohomology of any CW-complex model of its classifying space, see \cite[Chapter 5]{FrigerioBook}.

\begin{corollary} \label{TheoremGromovVanishing}
	Let $X$ be a connected CW complex with $G=\pi_1(X)$, let $\comp_{X}^k\colon H^k_b(X;\RR) \to H^k(X;\RR)$ be the (topological) comparison map for each $k \in \NN_0$ and let $\FF$ be an $H^*_b$-admissible family of subgroups of $G$. Then
	$$\cat_\FF(X) \geq \sup \{k \in \NN \mid \comp_X^k \neq 0\}.$$
\end{corollary}
\begin{proof}
	Let $\cc_X\colon X \to BG$ be a classifying map for the universal cover of $X$ and let 
	$$\cc'_X\colon H^*_b(G;\RR)\to H^*_b(X;\RR), \qquad \cc^*_X\colon H^*(G;\RR) \to H^*(X;\RR), $$
	denote the maps induced by $c_X$ in (bounded) cohomology and the usual identifications of the (bounded) cohomology of $BG$ with the (bounded) group cohomology of $G$. By Gromov's Mapping Theorem, $\cc'_X$ is an isomorphism of graded vector spaces.
	
	Let $k \in \NN$ with $\comp_X^k \neq 0$. Then $\comp_X^k \circ \cc_X' \neq 0$ and we derive from the evident commutativity of 
	$$\begin{CD}
		H^k_b(G;\RR) @>{\comp^k_\RR}>> H^k(G;\RR) \\
		@V{\cc_X'}V{\cong}V @V{\cc_X^*}VV \\
		H^k_b(X;\RR) @>{\comp^k_X}>> H^k(X;\RR)
	\end{CD} $$
	that $\cc_X^*\circ \comp^k_\RR\neq 0$. Let $u \in H^k_b(G;\RR)$ with $\cc_X^*(\comp^k_\RR(u))\neq 0$. By Theorem \ref{TheoremBoundedUniv}.b) there exists $\psi \in B_G(\II^k_G,\RR)$ with $u=\psi_*(\gamma_k)$. Since $\FF$ is $H^*_b$-admissible, we can argue as in the proof of Theorem \ref{TheoremcatFFlowergamma} to obtain $F \in B_G(\II^k_\FF,\II^k_G)$ with $F_*(\gamma_{\FF,k})=\gamma_k$ and with $\varphi := \psi \circ F$, we obtain that $u=\varphi_*(\gamma_{\FF,k})$. Hence, 
	\begin{align*}
		0 \neq \cc_X^*(\comp^k_\RR(u)) &= \cc_X^*(\comp^k_\RR(\varphi_*(\gamma_{\FF,k}))) \\
		&\stackrel{\eqref{EqCompPushfor}}{=} \cc_X^*(\varphi_*(\comp^k_{\II^k_\FF}(\gamma_{\FF,k}))) \\
		&\stackrel{\text{Prop. \ref{PropClassRelations}.c)}}{=} \cc_X^*(\varphi_*((\vartheta_{\FF,k})_*(\omega^k_{\FF,\RR})))\\
		&\stackrel{\text{Lemma } \ref{LemmaRealBerstein}}{=} \cc_X^*((\varphi \circ \vartheta_{\FF,k}\circ \alpha_k^{\FF})_*(\omega_\FF^k)) \\
		&= (\varphi \circ \vartheta_{\FF,k}\circ \alpha_k^{\FF})_*(\cc_X^*(\omega_\FF^k)).
	\end{align*}
	Hence, $\cc_X^*(\omega_\FF^k) \neq 0$ and we derive from Proposition \ref{prop:lower-bound-catF} that $\cat_\FF(X) \geq k$.
\end{proof}

\section{Consequences for families induced by isq-classes}

\subsection{Subgroup families induced by isq-classes} In this section we will focus on a more restricted version of $\FF$-categories, which includes as a distinguished case that of amenable category. First we recall the following construction, which builds on the notion of $\FF$-category from Definition \ref{DefFFcategory} and was originally given in this generality in \cite{GLGAH}.

\begin{definition} \label{DefcatGG}
	\begin{enumerate}[a)]
		\item A class of groups $\GG$ is called an \emph{isq-class} if it is closed under taking isomorphisms and under passing to subgroups and quotients of elements of $\GG$. 
		\item Let $X$ be a path-connected topological space and let $\GG$ be an isq-class of groups. Then $\GG(X):= \{ H \leq \pi_1(X) \mid H \in \GG\}$ is a family of subgroups of $\pi_1(X)$ and we put 
		$$ \cat_\GG(X) := \cat_{\GG(X)}(X)$$
		and call it the \emph{$\GG$-category of $X$}.
		\item Given an isq-class of groups $\GG$ we further put 
		$$\cat_{\GG}(G) := \cat_{\GG}(X_G),$$
		where $X_G$ is an aspherical CW complex with $\pi_1(X_G)=G$.
	\end{enumerate}	
\end{definition}

By \cite[Lemma 2.11]{CLM}, if  $f\colon X \to Y$ and $g\colon Y \to X$ are continuous maps between path-connected topological spaces for which $\pi_1(f)\colon \pi_1(X) \to \pi_1(Y)$ and $\pi_1(g)\colon \pi_1(Y) \to \pi_1(X)$ are isomorphisms, then $\cat_{\GG}(X) = \cat_{\GG}(Y)$ for any isq-class of groups $\GG$. This yields that the value of $\cat_\GG(G)$ in Definition \ref{DefcatGG}.c) is independent of the choice of $X_G$, hence well-defined. \medskip

Examples of such isq-classes are the class of abelian groups, the class of nilpotent groups or, most important for our purposes, the class of amenable groups. \medskip

Our first aim in this section is to derive an upper bound for $\cat_\GG(G)$ by connecting it to sectional categories of subgroup inclusions as studied in \cite{BCE22} and \cite{EFMO}. We refrain from recalling the notion of sectional category of a subgroup inclusion and refer the interested reader to \cite[Section 2]{EFMO} instead.

\begin{theorem}
\label{TheoremcatGupper}
	Let $\GG$ be an isq-class and let $H \leq G$ be a subgroup with $H \in \GG$. Then
	$$\cat_\GG(G) \leq \secat(\iota\colon H \hookrightarrow G).$$
\end{theorem}
\begin{proof}
	 Let $X$ and $Y$ be aspherical CW complexes and let $x_0 \in X$ and $y_0 \in Y$ be fixed basepoints with $\pi_1(X,x_0)=G$ and $\pi_1(Y,y_0)=H$. Let $f \colon Y \to X$ be continuous with  $f(y_0)=x_0$ and $\pi_1(f,y_0)=\iota$, such that $\secat(\iota) = \secat(f)$. 
	
	Put $k := \secat(f)$ and let $\{U_0,U_1,\dots,U_k\}$ be an open cover of $X$, such that for each $j \in \{0,1,\dots,k\}$ there exists a continuous $s_j\colon U_j \to Y$ with $f \circ s_j \simeq i_{U_j}$, where $i_{U_j}\colon U_j \hookrightarrow X$ denotes the inclusion.  For each $j \in \{0,1,\dots,k\}$ we further let $u_j \in U_j$ be arbitrary. Let $\varphi_j\colon \pi_1(X,(f\circ s_j)(u_j)) \to \pi_1(X,u_j)$ be a change-of-basepoint isomorphism induced by an arbitrary path. Then $$\pi_1(i_{U_j},u_j) = \varphi_j \circ \pi_1(f,s_j(u_j))\circ  \pi_1(s_j,u_j),$$ such that $\im \pi_1(i_{U_j},u_j)$ is isomorphic to a subgroup of $\im \pi_1(f,s_j(u_j))=:H'.$
By basic algebraic topology, $H'$ and $\im \pi_1(f,y_0)=H$ are conjugate, such that $H' \in \GG(X)$. Since $\GG(X)$ is closed under taking subgroups and isomorphisms, this yields $\im \pi_1(i_{U_j},u_j) \in \GG(X)$. Hence, $\{U_0,U_1,\dots,U_k\}$ is an open cover with the properties occurring in the definition of $\cat_\GG(X)$, which shows that $\cat_\GG(X) \leq k$. 
\end{proof}

\begin{prop}
\label{PropcatGnormal}
Let $\GG$ be an isq-class and let $N \triangleleft G$ be a normal subgroup with $N \in \GG$. 
\begin{enumerate}[a)]
	\item Then
$$\cat_\GG(G) \leq \cd(G/N). $$
\item If $\cd(G/N)<\infty$, if $N$ is of type $FP$ and if $H^{\cd(N)}(N;\ZZ[N])$ is $\ZZ$-free, then
$$\cat_\GG(G) \leq \cd(G)-\cd(N).$$
\end{enumerate}
\end{prop}
\begin{proof}
\begin{enumerate}[a)]
	\item This follows immediately from combining Theorem \ref{TheoremcatGupper} with \cite[Theorem 3.9]{EFMO}.
	\item By \cite[Theorem 5.5]{BieriHomol}, it holds under the given assumption that $\cd(G/N)=\cd(G)-\cd(N)$, so the claim is a direct consequence of a).
 \end{enumerate}
\end{proof}

Note that Proposition \ref{PropcatGnormal} shows that, under good conditions on $N$, the presence of a normal subgroup of $G$ lying in $\GG$ implies that $\cat_\GG(G)$ lies below its general upper bound $\cd(G)$.

\subsection{The case of amenable category} We recall that, in our notation, a group $\Gamma$ is \emph{amenable} if it admits a \emph{left-invariant mean}, i.e. a functional $\mu_\Gamma \in (\RR[\Gamma])''$, such that
\begin{itemize}
	\item $\mu_\Gamma(f)\geq 0$ if $f(g) \geq 0$ for every $g \in \Gamma$, 
	\item $\mu_\Gamma(\varepsilon_\Gamma)=1$, 
	\item $\mu_\Gamma( g \cdot f)=\mu_\Gamma(f)$ for all $g \in \Gamma$ and $f \in (\RR[\Gamma])'$.
\end{itemize}

In the following, we will denote the class of amenable groups by $\Am$ and consider the family of subgroups $\Am(G):= \{H \leq G \mid H \text{ is amenable}\}$. For any path-connected topological space $X$, we call $\cat_{\Am}(X)$ the \emph{amenable category} of $X$, which was first introduced in \cite{GLGAH13}.

To obtain a lower bound on amenable category from our previous considerations, we want to show that $\Am(G)$ is $H^*_b$-admissible, which will follow from the next observation.

\begin{prop}
	\label{PropAmRes}
	Let $G$ be a group and let $V$ be a normed $G$-module. Then $(B(\RG\otimes_\RR J_G^*,V'),\partial^*)$ is a bounded $\Am(G)$-injective resolution of $V'$. 
\end{prop}
\begin{proof}
	Put $C^k:=B(\RG\otimes_\RR J_G^k,V')$ for each $k \in \NN_0$. It follows immediately from Corollary \ref{CorBoundedRes} that $C^k$ is relatively $\Am(G)$-injective for each $k \in \NN_0$. Thus, it only remains to show that for each $H \in \Am(G)$ there exists a contraction $T^H_* \in B_H(C^*,C^{*-1})$ of $(C^*,\partial^*)$. One checks from the constructions of the proof of Theorem \ref{TheoremSpecialInjResol} that the maps
	\begin{align*}
		&\tau_k\colon C^k \to C^{k-1}, \quad  \tau_k(f)(x \otimes y)= f(1 \otimes (x-\varepsilon_G(x)\cdot 1)\otimes y)\ \  \forall   x \in \RG, \ y \in J^{k-1}_G \ k \in \NN, \\
		&\tau_0\colon C^0 \to V', \quad \tau_0(f)= f(1),
	\end{align*}
	form a contraction of $(B(\RG\otimes_\RR J_G^*,V'),\partial^*)$ with $\|\tau_k\|_\op \leq 1$ for all $k \in \NN_0$. \medskip
	
	Let $H \in \Am(G)$ and let $\mu_H \in (\RR[H])''$ be a left-invariant mean. For each $k \in \NN$ we define
	\begin{align*}
		&T^H_k\colon C^k \to C^{k-1}, \quad ((T^H_k(f))(u))(v) = \mu_H(\beta_{f,u,v}) \qquad \forall u \in \RG \otimes_{\RR} J_G^{k-1}, \ 	v \in V,\\
		&\text{where  } \quad \beta_{f,u,v}\in (\RR[H])', \quad \beta_{f,u,v}(h_0):= (((h_0 \cdot \tau_k)(f))(u))(v).
	\end{align*}
	A careful computation shows for all $h_0 \in H$, $x \in \RG$, $y \in J^{k-1}_G$ and $v \in V$ that 
	\begin{align}
		\beta_{f,x\otimes y,v}(h_0)&= (h_0 \cdot \tau_k(h_0^{-1}\cdot f))(x\otimes y)(v) \notag \\
		&= \tau_k(h_0^{-1} \cdot f)(h_0^{-1}x \otimes h_0^{-1}y)(h_0^{-1}v) \notag \\
		&= (h_0^{-1} \cdot f)(1 \otimes (h_0^{-1}x - \varepsilon(h_0^{-1}x)\cdot 1) \otimes h_0^{-1}y)(h_0^{-1}v) \notag \\
		\Rightarrow \quad 
		\beta_{f,x\otimes y,v}(h_0) &= f(h_0 \otimes (x- \varepsilon(x)\cdot h_0) \otimes y)(v) \label{EqBetaExpl}
	\end{align}
	and thus 
	\begin{align*}
		\beta_{h \cdot f,x\otimes y,v} (h_0)
		&= f(h^{-1}h_0 \otimes (h^{-1}x-\varepsilon(x)\cdot h^{-1}h_0)\otimes h^{-1}y)(h^{-1}v) = \beta_{f,h^{-1}x\otimes h^{-1}y,h^{-1}v}(h^{-1}h_0), 
	\end{align*}
	which yields $\beta_{h\cdot f,x \otimes y,v}= h \cdot\beta_{f,h^{-1}x\otimes h^{-1}y,h^{-1}v}$. We use this identity to derive from the $H$-invariance of $\mu_H$ that 
	\begin{align*}	
		T^H_k(h \cdot f)(x \otimes y)(v) = \mu_H(\beta_{h \cdot f,x\otimes y,v})= \mu_H(\beta_{f,h^{-1}x\otimes h^{-1}y,h^{-1}v}) 
		= (h\cdot T^H_k(f))(x\otimes y)(v)
	\end{align*}
	for all $h \in H$, $x \in \RG$, $y \in J^{k-1}_G$, $v\in V$. Thus, $T^H_k$ is an $\RR[H]$-homomorphism. Moreover, one checks from the boundedness of $\tau_k$ and $\mu_H$  that $T^H_k \in B_H(C^k,C^{k-1})$ for each $k \in \NN$. We further put $T_0^H\colon C^0 \to V'$, $(T_0^H(f))(v)= \mu_H(\alpha_{f,v})$ where $\alpha_{f,v} \in (\RR[H])'$, $\alpha_{f,v}(h_0)= f(h_0)(v)$.	A computation similar to the above shows that 	$\alpha_{h \cdot f,v}= h \cdot \alpha_{f,h^{-1}v}$ for all $f \in C^0$, $h \in H$ and $v \in V$ and one derives from this observation that $T_0^H \in B_H(C^0,V')$. Moreover, it follows from $\|\tau_k\|_\op \leq 1$ and $\|\mu_H\|_\op =1$ that $\|T^H_k\|_{\op} \leq 1$ as well. Here, $\|\mu_H\|_\op =1$ is an elementary property of invariant means on amenable groups.
	
	We want to show that the maps $T^H_*\colon C^* \to C^{*-1}$ indeed form a contraction. For all $k \in \NN$, $f \in C^k$, $x \in \RG$, $y_1 \in J_G$, $y_2 \in J_G^{k-1}$ and $v \in V$ we compute that	\begin{align*}
		(\partial^{k-1} \circ T^H_k)(f)(x \otimes y_1 \otimes y_2)(v)  
		&= T^H_k(f)(d_k(x \otimes y_1\otimes y_2))(v) 
		= \varepsilon_G(x) \cdot \mu_H(\beta_{f,\iota(y_1)\otimes y_2,v}).
	\end{align*}
	Using \eqref{EqBetaExpl}, we obtain for each $h_0 \in H$ that
	\begin{align*}
		\beta_{f,\iota(y_1)\otimes y_2,v}(h_0) &= f(h_0 \otimes (\iota_G(y_1)-\varepsilon_G(\iota_G(y_1))\cdot h_0) \otimes y_2)(v)= f(h_0\otimes y_1\otimes y_2)(v) =: \varphi_0(h_0),
	\end{align*}
	so that
	\begin{equation}
		\label{EqTHkvarphi}
		(\partial^{k-1} \circ T^H_k)(f)(x \otimes y_1 \otimes y_2)(v)= \varepsilon_G(x) \cdot \mu_H(\varphi_0).
	\end{equation}
	We further compute that
	\begin{align*}
		(T^H_{k+1}\circ \partial^k)(f)(x \otimes y_1\otimes y_2)(v) 
		&= T^H_{k+1}(f \circ d_{k+1})(x \otimes y_1 \otimes y_2)(v) =\mu_H(\beta_{f \circ d_{k+1},x\otimes y_1\otimes y_2,v})
	\end{align*}
	and derive from \eqref{EqBetaExpl} for each $h_0 \in H$ that 
	\begin{align*}
		\beta_{f \circ d_{k+1},x\otimes y_1\otimes y_2,v}(h_0)
		&= \varepsilon_G(h_0) \cdot f((x-\varepsilon_G(x)\cdot h_0)\otimes y_1 \otimes y_2)(v) \\
		&=f(x \otimes y_1 \otimes y_2)(v) - \varepsilon_G(x) \cdot \varphi_0(h_0) .
	\end{align*}
	Since the first summand is independent of $h_0$, we derive that
	$$\beta_{f \circ d_{k+1},x\otimes y_1\otimes y_2,v} = f(x \otimes y_1 \otimes y_2)(v) \cdot \varepsilon_H - \varepsilon_G(x) \cdot \varphi_0$$
	and thus 
	\begin{align*}
		(T^H_{k+1}\circ \partial^k)(f)(x \otimes y_1\otimes y_2)(v)  &= \mu_H\left(f(x \otimes y_1 \otimes y_2)(v) \cdot \varepsilon_H - \varepsilon_G(x) \cdot \varphi_0\right) \\
		&\stackrel{\eqref{EqTHkvarphi}}{=} f(x\otimes y_1 \otimes y_2)(v) - (\partial^{k-1} \circ T^H_k)(f)(x \otimes y_1 \otimes y_2)(v).
	\end{align*}
	Since $x$, $y_1$, $y_2$ and $v$ were chosen arbitrarily, this shows that $T^H_{k+1}\circ \partial^k+\partial^{k-1}\circ T^H_k=\id_{C^k}$ for each $k \in \NN$.
	In the case of $k=0$, we compute for all $x \in \RG$ and $v \in V$  that
	\begin{align*}
		((\varepsilon'_G \circ T^H_0)(f))(x)(v) &=\varepsilon_G(x)\cdot T^H_0(f)(v)= \varepsilon_G(x)\cdot \mu_H(\alpha_{f,v}), \\
		(T_1^H\circ \partial^0)(f)(x)(v)&= T_1^H(f \circ d_1)(x)(v)= \mu_H(\beta_{f\circ d_1,x,v}).
	\end{align*}
	In analogy with the above computation, we obtain  that $\beta_{f \circ d_1,x,v}=f(x)(v)-\varepsilon_G(x) \cdot \alpha_{f,v}$, hence
	\begin{align*}
		(T_1^H\circ \partial^0)(f)(x)(v) &= \mu_H(f(x)(v) \cdot \varepsilon_H) - \varepsilon_G(x) \cdot \mu_H(\alpha_{f,v}) \\
		&= f(x)(v)-	((\varepsilon'_G \circ T^H_0)(f))(x)(v).
	\end{align*}
	This yields $T_1^H\circ \partial^0+ \varepsilon'_G \circ T^H_0=\id_{C^0}$. Finally, we compute for all $f \in V'$ and $v \in V$ that 
	$(T^H_0 \circ \varepsilon'_G)(f)(v) = \mu_H(\alpha_{\varepsilon'_G(f),v}).$
	One checks that $\alpha_{\varepsilon'_G(f),v}=f(v)\cdot \varepsilon_H$,	from which we derive that $T^H_0 \circ \varepsilon'_G=\id_{V'}$.
	Thus, $T^H_*$ is the desired contraction and since $H \in \Am(G)$ was arbitrary, this shows the claim.
\end{proof}

The following statement is the counterpart of one implication of \cite[Theorem 1.5]{LiBounded} for bounded Adamson cohomology. The proof given below is independent of the arguments of \cite{LiBounded}.  

\begin{prop}
	\label{PropAmAdmiss}
	For all normed $G$-modules $V$ and all $k \in \NN$, the map
	$$\rho^k_{b,\Am}\colon H^k_b([G:\Am];V') \to H^k_b(G;V')$$
	is an isomorphism. 
\end{prop}
\begin{proof}
	By definition, the chain map underlying $\rho^k_{b,\Am}$ is formed by the maps 
	$$\sigma_{\{1\},k}'\colon B(\RR[G/\Am]\otimes_\RR J^k_{\Am},V') \to B(\RG\otimes_\RR J^k_G,V'), $$
	where $\sigma_{\{1\},k}\colon \RG \otimes_{\RR} J_G^k \hookrightarrow \RR[G/\Am]\otimes_{\RR} J^k_{\Am}$ is given as in \eqref{EqsigmaR}. By Proposition \ref{PropAmRes}, $\sigma_{\{1\},k}'$ is a map between two bounded $\Am$-injective resolutions and one easily checks that it is a lift of $\id_{V'}\colon V' \to V'$. Thus, by standard arguments from homological algebra, making use of \cite[Proposition 3.9]{LiBounded}, it induces an isomorphism in cohomology. 
\end{proof}

An immediate consequence of Proposition \ref{PropAmAdmiss} is that $\Am$ is an $H^*_b$-admissible family of subgroups of $G$.

\begin{corollary} \label{cor:johnsonlowerbound}
Let $V$ be a normed $G$-module. Then 
$$\catAm(G) \geq \sup \left\{k \in\NN \ \middle|\ \comp^k_{\II_G^k}(\gamma_k) \neq 0\right\}\geq \sup \left\{k \in \NN\  \middle| \  \comp^k_{V'} \neq 0\right\}.$$
\end{corollary}

\begin{proof}
	This follows from Theorem \ref{TheoUnivComp}, Theorem \ref{TheoremcatFFlowergamma} and Proposition \ref{PropAmAdmiss}.
\end{proof}


The second inequality of Corollary \ref{cor:johnsonlowerbound} can indeed be strict, as the following example shows for $V=\RR$.

\begin{example}
\label{ExampleCircleBundle}
	Let $p \colon E \to \Sigma_g$ be an oriented circle bundle over a closed connected oriented surface of genus $g \geq 2$ with non-vanishing Euler class. Put $G := \pi_1(E)$ and $Q := \pi_1(\Sigma_g)$, and let $q:=\pi_1(p) \colon G \to Q$. Then $q$ gives a central group extension of the form 
	\begin{equation} \label{Eqcentralgroup}
	1 \to \ZZ \to G \xrightarrow{q} Q \to 1. 
	\end{equation}
	Given a normed $Q$-module $V$, the map $q^*_b\colon H^*_b(Q;V') \to H^*_b(G;q^*V')$ induced by $q$ is an isomorphism by \cite[Theorem 4.24]{FrigerioBook}. The naturality of comparison maps further yields that 
\begin{equation}
\label{naturalitycomp}
	\begin{tikzcd}
H^k_b(Q;V') \arrow[r, "\comp^k_{V'}"] \arrow[d, "q^*_b"', "\cong"]
  & H^k(Q;V') \arrow[d, "q^*"'] \\
H^k_b(G;q^*V') \arrow[r, "\comp^k_{q^*V'}"]
  & H^k(G;q^*V')
\end{tikzcd}
\end{equation}
	commutes for each $k \in \NN$.
	Let $e(p)\in H^2(\Sigma_g;\RR)$ be the real Euler class of $p$, which is non-zero as well.  
	For every $\RR[Q]$-module $A$, the cohomological five-term exact sequence associated with \eqref{Eqcentralgroup} contains the exact sequence 
	\begin{equation}
	\label{fiveterm} 
	H^1(\ZZ;A)^Q \cong A^Q \xrightarrow{e(p) \cup -} H^2(Q;A) \xrightarrow{q^*} H^2(G;q^*A), 
	\end{equation} 
	(see for example \cite[Theorem 4]{HochSerre}). For $A = \RR$ the first homomorphism in \eqref{fiveterm} is an epimorphism by Poincar\'e duality, hence $q^* \colon H^2(Q;\RR) \to H^2(G;\RR)$ is trivial. Since \eqref{naturalitycomp} yields  $\comp^k_\RR  = q^* \circ \comp^k_\RR\circ (q^*_b)^{-1}$ for each $k \in \NN$, this shows that $\comp^2_\RR=0$.
	Since $Q$ is a surface group, $H^k(Q;\RR) = 0$ for any $k \geq 3$, so \eqref{naturalitycomp} further yields  $$\comp^k_\RR = 0 \qquad \forall k \geq 2.$$ 
	However, the situation changes if we choose a different coefficient module. Let $\ell^1(Q)$ be the completion of $\RR[Q]$ with respect to $\|\cdot \|_1$, i.e. the space of formal sums $\sum_{g \in Q} r_g \cdot g$, where $r_g \in \RR$ for each $g \in Q$ and $\sum_{g \in Q} |r_g| < +\infty$, seen as a Banach $Q$-module.	Recall that $\ell^1(Q)$ is isometrically isomorphic to $(c_0(Q))'$, see \cite[Section 2.2]{FHHMZ}, and one checks without difficulties that the isomorphism is $Q$-invariant with respect to the standard actions, so $\ell^1(Q)$ is a dual $Q$-module. 
	
	
	We observe that a $Q$-invariant element $\sum_{g \in Q} r_g g$ of $\ell^1(Q)$ would satisfy that $r_g=r_h$ for all $g,h \in Q$. Since $Q$ is infinite, this is only satisfied by the zero element, so $(\ell^1(Q))^Q=\{0\}$. On the other hand, the coinvariant module $(\ell^1(Q))_Q$ is non-trivial:

	Let $\epsilon\colon \ell^1(Q) \to \RR$ be the continuous extension of the augmentation. Evidently, $\|\epsilon\|_\op \leq 1$ and one easily checks that $\epsilon \in B_G(\ell^1(Q),\RR)$, so it induces $\overline{\epsilon} \in ((\ell^1(Q))_Q)'$. Consider the unit element $1 \in \ell^1(Q)$. Since $\epsilon(1)=1$, it follows that $\overline{\epsilon}([1])=1$, so $[1]\neq 0$ and thus $(\ell^1(Q))_Q\neq \{0\}$.
Since $Q$ is a Poincar\'e duality group of dimension two, there is an isomorphism 
$$ H^2(Q;\ell^1(Q)) \cong H_0(Q;\ell^1(Q)) = (\ell^1(Q))_Q, $$ 
hence $H^2(Q;\ell^1(Q)) \neq \{0\}$. Since $Q$ is hyperbolic, $\comp_{\ell^1(Q)}^2\colon H^2_b(Q;\ell^1(Q)) \to H^2(Q;\ell^1(Q))$ is surjective by \cite[Theorem 3]{MineyevHyper}. Let $u \in H^2(Q;\ell^1(Q))$ with $u \neq 0$ and let $v \in H^2_b(Q;\ell^1(Q))$ with $\comp^2_{\ell^1(Q)}(v) = u$. By the commutativity of \eqref{naturalitycomp}, $$ \comp^2_{q^*\ell^1(Q)}(q^*_b(v)) = q^* (\comp^2_{\ell^1(Q)}(v)) = q^*(u) . $$ 
Since $(\ell^1(Q))^Q=\{0\}$, it follows from \eqref{fiveterm} that $q^*\colon H^2(Q;\ell^1(Q))\to H^2(G;q^*(\ell^1(Q)))$ is injective, so $q^*(u)\neq 0$. Thus, $\comp^2_{q^*\ell^1(Q)}\neq 0$. Since the pullback by $q$ takes dual $\RR[Q]$-modules to dual $\RR[G]$-modules, $q^*\ell^1(Q)$ is a dual $\RR[G]$-module. Thus, it follows from Corollary \ref{cor:dualmodulescontrolJohnsonadmissibility} that $$\comp^2_{\II^2_G}(\gamma_2) \neq 0.$$
\end{example}

\begin{remark}
In \cite[Remark 3.16]{CLM} it was asked whether any oriented closed connected aspherical manifold $X$ has efficient $\Am$-category, i.e. if it satisfies 
\begin{equation}
\label{EqCLM1}
 \catAm(X) = \sup \{ k \in \NN_0 \mid \comp^k_{\RR}\neq 0\},
\end{equation}
where $\comp_{\RR}^k\colon H^k_b(X;\RR) \to H^k(X;\RR)$ is the comparison map. We can answer this question negatively, since the space $E$ from Example \ref{ExampleCircleBundle} provides a counterexample: Since $E$ is a circle bundle over an aspherical surface, it is itself aspherical. By the computation of amenable categories of closed $3$-manifolds from \cite{GLGAH13}, see also \cite{KotschMonot}, $E$ satisfies $\catAm(E)=2$ in our reduced convention. As shown in Example \ref{ExampleCircleBundle}. it further satisfies $\comp_\RR^2 = 0$, so $E$ violates \eqref{EqCLM1}. However, together with Theorem \ref{TheoUnivComp} and Corollary \ref{cor:johnsonlowerbound}, the previous computations show  that indeed 
\begin{equation}
\label{EqCLM2}
 \catAm(E) = \sup \{ k \in \NN_0 \mid \exists V' \text{  dual $G$-module with } \comp^k_{V'}\neq 0\},
\end{equation}
where $\comp_{V'}^k\colon H^k_b(E;V') \to H^k(E;V')$ is the respective comparison map. One might therefore replace the above question by asking if any closed aspherical manifold satisfies \eqref{EqCLM2}. Note that by Corollary \ref{cor:johnsonlowerbound}, the inequality \glqq$\geq$\grqq{} always holds in \eqref{EqCLM2}.
\end{remark}

\subsection{An application to degree-one maps between manifolds.}
A classic open problem in the setting of Lusternik-Schnirelmann category posed by Y. Rudyak concerns its monotonicity under degree-one maps between manifolds, see \cite[Open problem 2.48]{CLOT}. As a generalization of said problem, Capovilla, L\"oh and Moraschini raised the following question, which they called the generalized monotonicity problem. 

\begin{question}[Question 6.3 of \cite{CLM}]
Let  $f \colon M \to N$ a degree-one map between closed ceonnected oriented manifolds. Does it hold that $ \cat_\Am(M) \geq \cat_\Am(N) $?
\end{question}  

Capovilla, L\"oh and Moraschini have answered this question affirmatively for $3$-manifolds in \cite[Theorem 6.6]{CLM}. A proof of this result by simpler arguments was given by D. Kotschick in \cite{KotschMonot}.

We are going to use Theorem \ref{thm:CostaFarberCatF} to give an answer to the question under the premise of maximality of $\cat_\GG(N)$ for an arbitrary isq-class $\GG$. In order to do so, first we have to introduce a notion of compatibility of subgroup families with respect to a continuous map between spaces.

\begin{definition}
Let $X$ and $Y$ be connected CW-complexes, and let $\FF_X$ and $\FF_Y$ be subgroup families of $\pi_1(X)$ and $\pi_1(Y)$, respectively. We say that a continuous map $f \colon X \to Y$ is \emph{$\left(\FF_X , \FF_Y\right)$-compatible} if the induced group homomorphism $\pi_1(f) \colon \pi_1(X) \to \pi_1(Y)$ satisfies $\pi_1(f)(K) \in \FF_Y$ for any $K \in \FF_X$.
\end{definition}

For any such ($\FF_X , \FF_Y$)-compatible map $f \colon X \to Y$ we have an obvious induced map 
$$ F \colon \bigsqcup_{H \in \FF_X} \quot{\pi_1(X)}{H} \longrightarrow \bigsqcup_{K \in \FF_Y} \quot{\pi_1(Y)}{K} $$ by assigning $gH \mapsto \pi_1(f)(g) \cdot \pi_1(f)(H)$ for any $H \in \FF_X$. Considering $f^* \ZZ[\pi_1(Y)/\FF_Y]$ as a $\Z[\pi_1(X)]$-module by the action induced by $f$, we can linearly extend $F$ to a $\ZZ[\pi_1(X)]$-module homomorphism 
$$ \ZZ[F] \colon \ZZ[\pi_1(X)/\FF_X] \to f^*\ZZ[\pi_1(Y)/\FF_Y]. $$ 
The map $\ZZ[F]$ then restricts to a $\ZZ[\pi_1(X)]$-module homomorphism between the relative augmentation ideals. For any $n \in \NN$, applying $\ZZ[F]$ to the factors of the corresponding tensor power produces a $\ZZ[\pi_1(X)]$-homomorphism  
$$ F_n \colon I^n_{\FF_X} \longrightarrow f^*I^n_{\FF_Y}. $$ 
Since $F$ sends the coset of the trivial subgroup of $\pi_1(X)$ to the coset of the trivial subgroup of $\pi_1(Y)$, the identification provided by Corollary \ref{Corphikrepr} for the choice $\sigma_{\{1\},k}$ shows that $$ (F_n)_*(\omega^n_{\FF_X}) = f^*(\omega^n_{\FF_Y}) \in H^n\left(\pi_1(X);f^*I^n_{\FF_Y}\right). $$ Since pullbacks by maps and those of coefficient homomorphisms commute in cohomology, 
 \begin{equation} \label{degonecoeffchange}
f^*(\mathfrak{c}^*_Y(\omega^n_{\FF_Y})) = \mathfrak{c}^*_X((F_n)_*(\omega^n_{\FF_X})) = (F_n)_*(\mathfrak{c}^*_X(\omega^n_{\FF_X})) \in H^n\left(X;f^*I^n_{\FF_Y}\right).  
\end{equation} 
The very same reasoning applies to the real counterparts of $\omega_{\FF_X}$ and $\omega_{\FF_Y}$ from Section 2 when applied to the local coefficient system $J^n_\FF$, to obtain  \begin{equation} \label{degonecoeffchangeR}
f^*(\mathfrak{c}^*_Y(\omega^n_{\FF_Y,\RR})) = \mathfrak{c}^*_X((F_{n,\RR})_*(\omega^n_{\FF_X,\RR})) = (F_{n,\RR})_*(\mathfrak{c}^*_X(\omega^n_{\FF_X,\RR})) \in H^n\left(X;f^*J^n_{\FF_Y}\right),  
\end{equation} where $F_{n,\RR} \colon J^n_{\FF_X} \to f^* J^n_{\FF_Y}$ is defined in analogy with $F_n$.  

\medskip

Let $M$ and $N$ be closed connected oriented $n$-manifolds, where $n \in \NN$, and let $f \colon M \to N$ a continuous map of degree $d\in \ZZ$. For any local coefficient system $A$ on $N$, we have the Poincar\'e duality isomorphisms $$P_M \colon H^k(M;f^*A) \to H_{n-k}(M;f^*A) \qquad P_N \colon H^k(N;A) \to H_{n-k}(N;A) $$ given by cap products with fundamental classes. Consider the map $$ f^{!} \colon H^*(M;f^*A) \to H^*(N;A) \qquad f^!(x) := \left(P^{-1}_N \circ f_* \circ P_M \right) (x). $$ Since for any $z \in H^k(N;A)$ we have $$ f^!f^*(z) = P^{-1}_N f_*(f^*z \cap [M]) = P^{-1}_N(z \cap (\deg f)[N]) = (\deg f) z ,$$ see \cite[Proposition 5.1]{DranishRudyak}, it follows that $f^*$ is a split monomorphism with left inverse $(\deg f)f^!$ if $\deg f=\pm1$. More generally, if the local coefficient system $A$ is taken over $R[G]$ instead of $\ZG$, for $R$ a commutative ring with unit, one has that $f^*$ is a split monomorphism if $\deg f$ is an invertible element in $R$, with left-inverse $(\deg f)^{-1} \cdot f^!$. 

\medskip

Using these considerations, we will establish a result for the generalized monotonicity of maps between manifolds via powers of relative Berstein classes. In order to do so, first we observe that non-vanishing of such powers for the codomain manifold offers a bound for the category of the domain.  

\begin{theorem} \label{monotonicitydegoneF}
Let $n \in \NN$ and let $f \colon M \to N$ be a map between closed connected  oriented smooth $n$-manifolds of degree $d \in \ZZ$. Let $\FF_M$ and $\FF_N$ be subgroup families of $\pi_1(M)$ and $\pi_1(N)$, respectively. Suppose that $f$ is $\left(\FF_M, \FF_N \right)$-compatible. Then 
 \begin{enumerate}[a)]
 	\item $\cat_{\FF_M}(M) \geq \sup \{ k \in \NN_0 \mid \cc^*_N(\omega^k_{\FF_N}) \neq 0 \}$ if $d = \pm 1$.
 	\item $\cat_{\FF_M}(M) \geq \sup \{ k \in \NN_0 \mid \cc^*_N(\omega^k_{\FF_N, \RR}) \neq 0 \}$ if $d \neq 0$.
 \end{enumerate}  
\end{theorem}
\begin{proof}
\begin{enumerate}[a)]
\item Suppose $\cc^*_N(\omega^k_{\FF_N}) \neq 0$. Since $\deg f = \pm 1$ the homomorphism $f^*$ is a split monomorphism, hence $$0 \neq f^*(\mathfrak{c}^*_N(\omega^k_{\FF_N})) \in H^k(M;f^*I^k_{\FF_N}).$$ As $f$ is assumed to be ($\FF_M , \FF_N$)-compatible, by \eqref{degonecoeffchange} we have the identification $$ 0 \neq f^*(\mathfrak{c}^*_N(\omega^k_{\FF_N})) = (F_k)_*(\mathfrak{c}^*_M(\omega^k_{\FF_M})). $$ Since every closed smooth manifold is triangulable, we can apply Proposition \ref{prop:lower-bound-catF} to derive from $\mathfrak{c}^*_M(\omega^k_{\FF_M}) \neq 0$ that $\cat_{\FF_M}(M) \geq k$. 
\item Assume $\cc^*_N(\omega^k_{\FF_N, \RR}) \neq 0$. Considering $J^k_{\FF_N}$ as a local coefficient system for $N$, and since $d = \deg f \neq 0$, the homomorphism $f^* \colon H^*(N;J^k_{\FF_N}) \to H^*(M;f^*J^k_{\FF_N})$ is a split monomorphism with left inverse $d^{-1} f^!$. By \eqref{degonecoeffchangeR} we have $$ 0 \neq f^*(\mathfrak{c}^*_N(\omega^k_{\FF_N, \RR})) = (F_{k,\RR})_*(\mathfrak{c}^*_M(\omega^k_{\FF_M, \RR})). $$ Thus $\mathfrak{c}^*_M(\omega^k_{\FF_M, \RR}) \neq 0$, which by Lemma \ref{LemmaRealBerstein} means $\mathfrak{c}^*_M(\omega^k_{\FF_M}) \neq 0$. Proposition \ref{prop:lower-bound-catF} then implies $\cat_{\FF_M}(M) \geq k$. 
\end{enumerate} \end{proof}

The following monotonicity result thus follows directly. 

\begin{corollary} \label{cor:monotonicitydegoneF}
Under the hypothesis of Theorem \ref{monotonicitydegoneF}, let $\cat_{\FF_N}(N) = k$. Then $\cat_{\FF_M}(M) \geq \cat_{\FF_N}(N) $ if one of the following conditions holds:
\begin{enumerate}[a)]
	\item $\cc^*_N(\omega^k_{\FF_N}) \neq 0$ and $d = \pm 1$.
	\item $\cc^*_N(\omega^k_{\FF_N, \RR}) \neq 0$ and $d \neq 0$.
\end{enumerate}   
\end{corollary}

We apply Theorem \ref{monotonicitydegoneF} above to isq-classes to derive the aforementioned partial answer to the general monotonicity problem under the premise of maximality of the $\GG$-category of the codomain.

\begin{corollary} \label{cor:monotonicitydegoneG}
Let $n \geq 3$, let $f \colon M \to N$ be a map	between closed connected oriented smooth $n$-manifolds with $\deg f=\pm 1$ and let $\GG$ be an isq-class of groups. If $\cat_\GG(N) = n$ then $\cat_\GG(M) = n$. 
\end{corollary}
\begin{proof}
Since for any $H \in \GG(M)$ we have $\pi_1(f)(H) \cong \quot{H}{(H \cap \ker \pi_1(f))}$ and since $\GG$ is closed under taking subgroups and quotients,  $\pi_1(f)(H) \in \GG(N)$, so $f$ is ($\GG(M), \GG(N)$)-compatible. Since $\cat_{\GG}(N)=n$,  Theorem \ref{thm:CostaFarberCatF} implies $\cc^*_N(\omega^n_{\FF_N}) \neq 0$. The claim thus follows directly from Theorem \ref{monotonicitydegoneF} and the elementary chain of inequalities $\cat_\GG(M) \leq \cat(M) \leq n$.
\end{proof}

Corollary \ref{cor:monotonicitydegoneG} gives an affirmative answer to the maximal case of the generalized monotonicity problem without using the simplicial volume of the manifold.

\bibliography{GC}	
\bibliographystyle{amsalpha}

\end{document}